\documentclass[10pt,reqno]{amsart}
\usepackage{amsmath} 
\usepackage{amssymb}
\usepackage{dsfont}
\usepackage[dvips,draft,final]{graphics}
\usepackage{amssymb,amsmath}
\usepackage[T1]{fontenc}
\usepackage{fancyhdr}
\usepackage{url}
\usepackage{mathrsfs}
\usepackage{soul}

\usepackage{color}
\usepackage{graphicx}

\def\squarebox#1{\hbox to #1{\hfill\vbox to #1{\vfill}}}

\newtheorem{Thm}{Theorem}[section]
 \newtheorem{cor}{Corollary}[section]

\newtheorem{lem}{Lemma}[section]

\numberwithin{equation}{section}

\newcommand{\bel}{\begin{equation} \label}
\newcommand{\ee}{\end{equation}}
\newcommand{\re}{\mathfrak R}
\newcommand{\pd}{\partial}

\newcommand{\R}{\mathbb{R}}

\def\epsilon{\varepsilon}
\def\phi {\varphi}

\newtheorem{rem}{Remark}[section]
\newtheorem{prop}{Proposition}[section]

\providecommand{\abs}[1]{\left\lvert#1\right\rvert}
\providecommand{\norm}[1]{\left\lVert#1\right\rVert}
\numberwithin{equation}{section}

\renewcommand{\leq}{\leqslant}
\renewcommand{\geq}{\geqslant}
\providecommand{\abs}[1]{\left\lvert#1\right\rvert}
\providecommand{\norm}[1]{\left\lVert#1\right\rVert}

\def\beq{\begin{equation}}
\def\eeq{\end{equation}}
\newcommand{\bea}{\begin{eqnarray}}
\newcommand{\eea}{\end{eqnarray}}
\newcommand{\beas}{\begin{eqnarray*}}
\newcommand{\eeas}{\end{eqnarray*}}

{

\begin{document}

\title[Determination of convection terms and quasi-linearities]{Determination of convection terms and quasi-linearities appearing in diffusion equations}
%

\author{Pedro Caro}
\address{BCAM - Basque Center for Applied Mathematics, 48009 Bilbao, Spain and Ikerbasque,
Basque Foundation for Science, 48011 Bilbao, Spain.}
\email{pcaro@bcamath.org}
\author{Yavar Kian}
\address{Aix Marseille Univ, Universit\'e de Toulon, CNRS, CPT, Marseille, France.}
\email{yavar.kian@univ-amu.fr}

\begin{abstract} We consider the highly nonlinear and ill-posed inverse problem of determining some general expression $F(x,t,u,\nabla_xu)$ appearing in the diffusion equation $\partial_tu-\Delta_x u+F(x,t,u,\nabla_xu)=0$ on $\Omega\times(0,T)$, with $T>0$ and $\Omega$ a bounded open subset of $\R^n$, $n\geq2$, from measurements of solutions on the lateral boundary $\partial\Omega\times(0,T)$. We consider both linear and nonlinear expression of $F(x,t,u,\nabla_xu)$. In the linear case, the equation can be seen as a convection-diffusion equation and our inverse problem corresponds to the   unique recovery, in some suitable sense, of a time evolving velocity field  associated with the moving quantity as well as the density of the medium  in some rough setting described by non-smooth coefficients on a Lipschitz domain. In the nonlinear case, we prove the recovery of more general quasi-linear expression appearing in a nonlinear parabolic equation associated with more complex model. Here the goal is to determine the underlying physical low of the system associated with our equation. In this paper, we consider for what seems to be the first time  the unique  recovery of a general vector valued  first order coefficient, depending on both time and space variable. Moreover, we provide  results of full recovery of some general class of quasi-linear terms admitting evolution inside the system independently of the solution  from measurements at the boundary. These last results improve earlier works of Isakov in terms of generality and precision. In addition, our results  give a partial positive answer, in terms of measurements restricted to the lateral boundary,  to an open problem posed by Isakov in his classic book (\textit{Inverse Problems for Partial Differential Equations}) extended to the recovery of  quasi-linear terms.


\end{abstract}

\maketitle
 \tableofcontents
\section{Introduction}
\label{sec-intro}
\setcounter{equation}{0}

\subsection{Statement of the problem}
In this paper we consider an inverse problem stated for a class of diffusion equations corresponding to the determination of different information about the moving quantities associated with these equations. We state our results in a general setting by allowing the information, that we want to determine, to be associated with non-smooth coefficients depending on both time and space variable or even more general quasi-linear terms.

More precisely, let $\Omega$ be a Lipschitz bounded domain  of $\R^n$, $n\geq2$, such that $\R^n\setminus\Omega$ is connected. We set $Q=\Omega\times(0,T)$ and $\Sigma=\pd\Omega\times(0,T)$ with $T>0$, and for convenience write $\Omega^s=\Omega\times\{s\}$ with $s=0,T$. In this paper, we study the inverse problems associated with an initial boundary value problem (IBVP in short) associated with a diffusion equation taking the form
\begin{equation}\label{eq111}\left\{\begin{array}{ll}\partial_tu-\Delta_x u+F(x,t,u,\nabla_xu)=0,\quad &\textrm{in}\ Q,\\  u (\cdot, 0)=u_0,\quad &\textrm{in}\ \Omega,\\ u=g,\quad &\textrm{on}\ \Sigma.\end{array}\right.\end{equation}
These problems are often associated with models of transfer or movement of different physical quantities (see Section 1.2 for more details). In this context, our goal is to prove the unique determination of some information about these moving quantities encoded in the expression  $F(x,t,u,\nabla_xu)$  from measurements  of its solutions on the lateral boundary $\Sigma$ for arbitrary many 
inputs $g$ and $u_0$ on $\Sigma$ and $\Omega^0$ respectively. We consider both linear expressions  of the form $F(x,t,u,\nabla_xu)=A(x,t)\cdot\nabla_xu+\nabla_x\cdot [B(x,t)]u+q(x,t)u$, and more general quasi-linear expressions. In the linear case, we consider the recovery of $F(x,t,u,\nabla_xu)$ from inputs and measurements restricted to the lateral boundary $\Sigma$ starting with zero initial data.

There seem to be no literature addressing the question of unique 
determination of $F(x,t,u,\nabla_xu)$ when this expression is linear and time-
dependent. In the quasi-linear case, the unique determination has only been established for expressions independent of $t$ and $x$ determined on an abstract set (see Section 1.7 for more details). The present paper extends the previous literature in terms of generality and precision.  In particular, we prove unique determination for  quasi-linear terms admitting variation on the inaccessible part $\Omega$ of the system from measurements on the accessible  part $\partial \Omega$. Such result gives a partial positive answer (in terms of measurements) to the open problem \cite[Problem 9.6, pp.  296]{I5} extended to quasi-linear nonlinearities.


\subsection{Motivations}

In the linear case,  the IBVP \eqref{eq111} takes the form
\begin{equation}\label{eq1}\left\{\begin{array}{ll}\partial_tu-\Delta_x u+A(x,t)\cdot\nabla_xu+\nabla_x\cdot [B(x,t)]u+q(x,t)u=0,\quad &\textrm{in}\ Q,\\  u(\cdot, 0)=0,\quad &\textrm{in}\ \Omega,\\ u=g,\quad &\textrm{on}\ \Sigma,\end{array}\right.\end{equation} 
with $A,B\in L^\infty(Q)^n$ and $q\in L^\infty(0,T;L^{\frac{2n}{3}}(\Omega))$. This IBVP  is associated with a convection-diffusion equation which describes a combination of two physical processes: diffusion and convection (advection). This equation is extensively used in the context of transfer of mass or heat, due to diffusion and convection, of different physical quantities  (particles, energy,...) inside a physical system (see for instance \cite{St}).  The problem \eqref{eq1} can also describe the velocity of a particle (Fokker-Planck equations) or the price evolution of a European call (Black-Scholes equations). Here the coefficient $A$ corresponds  to the velocity field associated with the moving quantity and our inverse problem corresponds to the recovery of this field from information given by an application of source and measurement of the flux at the boundary of the domain.

The quasi-linear problem \eqref{eq111} corresponds to more complex
 models where the linear expression
has to be replaced by a more general nonlinear term.
This inverse problem turns out to be interesting in different frameworks as  physics of high temperatures, chemical kinetics and aerodynamics.

\subsection{Obstruction to uniqueness}

Let us state our inverse problem for the linear IBVP \eqref{eq1}. For this purpose, we need to introduce the Dirichlet-to-Neumann (DN in short) map associated to \eqref{eq1}. Before this, we define $N_{A,B,q}u$, with $u$ solving \eqref{eq1}, as
\bel{tutu1}\left\langle N_{A,B,q}u,w_{|\Sigma}\right\rangle:=\int_Q[-u\pd_tw+\nabla_xu\cdot\nabla_xw+A\cdot\nabla_xuw-B\cdot\nabla_x (uw)+quw]dxdt,\ee
where $w\in H^1(Q)$ satisfies $w_{|\Omega^T}=0$. Here the  term $A, B$ take values in $\R^n$. Then, the DN map is given by
$$\Lambda_{A,B,q}:g\mapsto N_{A,B,q}u.$$

We refer to Section 2 for more details and a rigorous  definition of this map. Note that for $g$, $A$, $B$, $q$ and $\Omega$ sufficiently smooth, we have that
$$N_{A,B,q}u=[\partial_\nu u-(B\cdot\nu)u]_{|\Sigma}$$
with $\nu$ the outward unit normal vector to $\pd\Omega$. This means that $N_{A,B,q}$ and $\Lambda_{A,B,q}$ are the natural extensions of the normal derivative of the solution of \eqref{eq1} and the corresponding DN map to non-smooth setting.


We recall that there is an obstruction to uniqueness for this inverse problem, which is given by some gauge invariance. More precisely, consider
\bel{p} p_1:=\left\{\begin{array}{l} n\quad \textrm{for }n\geq3\\ 2+\epsilon\quad \textrm{for }n=2,\end{array}\right.\ee
with $\epsilon\in(0,1)$.
Let $A_1,B_1\in L^\infty(Q)^n$, $q_1\in L^\infty(0,T;L^{\frac{2n}{3}}(\Omega))$ and $$\phi\in L^\infty(0,T;W^{1,\infty}(\Omega))\cap W^{1,\infty}(0,T;L^{p_1}(\Omega))\cap L^\infty(0,T;H^1_0(\Omega))\setminus\{0\}.$$   Now consider $A_2\in L^\infty(Q)^n$, $q_2\in L^\infty(0,T;L^{\frac{2n}{3}}(\Omega))$ given by
\bel{gauge1}A_2=A_1+2\nabla_x\phi,\quad B_2=B_1+\nabla_x\phi,\quad q_2=q_1-\pd_t\phi-|\nabla_x\phi|^2-A_1\cdot\nabla_x\phi.\ee
Then, one can check that $\Lambda_{A_1,B_1,q_1}=\Lambda_{A_2,B_2,q_2}$ but obviously $A_1\neq A_2$. We can also prove that, for any
$$\phi\in \{h\in L^\infty(0,T;W^{1,\infty}(\Omega))\cap W^{1,\infty}(0,T;L^{p_1}(\Omega)):\  h_{|\Sigma}=0\},$$
the DN map of problem \eqref{eq1} satisfies the following gauge invariance 
$$\Lambda_{A,B,q}=\Lambda_{A+2\nabla_x\phi,B+\nabla_x\phi,q-\pd_t\phi-|\nabla_x\phi|^2-A\cdot\nabla_x\phi}.$$
According to this obstruction, the best result that one can expect is the recovery of the gauge class of the coefficients $(A,B,q)$ given by the relation \eqref{gauge1}.  Note also that, without additional information about the coefficient $B$ it is even impossible to recover the gauge class of $(A,B,q)$ given by \eqref{gauge1}. Indeed, for any $\phi\in W^{2,\infty}(Q)$ satisfying $\phi_{|\Sigma}=\partial_\nu\phi_{|\Sigma}=0$, the DN map of problem \eqref{eq1} satisfies the following gauge invariance 
$$\Lambda_{A,B,q}=\Lambda_{A+2\nabla_x\phi,B,q-\pd_t\phi-|\nabla_x\phi|^2-A\cdot\nabla_x\phi+\Delta_x\phi}.$$
This means that for 
$$A_2=A_1+2\nabla_x\phi,\quad B_2=B_1,\quad q_2=q_1-\pd_t\phi-|\nabla_x\phi|^2-A_1\cdot\nabla_x\phi+\Delta_x\phi,$$
with $\phi\in W^{2,\infty}(Q)$ satisfying $\phi_{|\Sigma}=\partial_\nu\phi_{|\Sigma}=0$ and $\phi\neq0$, we have $\Lambda_{A_1,B_1,q_1}=\Lambda_{A_2,B_2,q_2}$ but condition \eqref{gauge1} is not fulfilled. In light of these obstructions, in the present paper we consider the recovery of some information about the gauge class of the coefficient $(A,B,q)$ given by \eqref{gauge1} from the DN map $\Lambda_{A,B,q}$. By considering additional assumptions on the low order coefficients $B,q$ we will derive more precised results for the recovery of the convection term $A$ from 
$\Lambda_{A,B,q}$.
\subsection{State of the art}

The recovery of coefficients appearing in parabolic equations has attracted many
attention these last decades.  We refer to \cite{Ch,I5}  for an overview  of such problems. While numerous  authors considered the recovery of the zero order coefficient $q$, only few authors studied the determination of the convection term $A$. We can mention the work of \cite{DYY,ZL} for the treatment of this problem in the 1 dimensional case as well as the work of \cite{CY} dealing with the unique recovery of a time-independent convection term for $n=2$ from a single boundary measurement. 

Recall that, for time-independent coefficients $(A,B,q)$ and with suitable regularity assumptions, one can apply the analyticity in time of the solutions of \eqref{eq1}, with suitable boundary conditions $g$, and the Laplace transform with respect to the time variable, in order to transform our inverse problem into  the recovery of coefficients appearing in a steady state convection-diffusion equation  (see for instance \cite{KKL} for more details about this transformation of the inverse problem). This last inverse problem has been studied by \cite{CY1,CY2,KU2,P} and it is strongly connected to the recovery of  magnetic Schr\"odinger operator from boundary measurements which has been intensively studied these last decades. Without being exhaustive, we refer to the work of \cite{CP,DKSU,Ki5,Pot1,Pot2, Sa1,Suuu}. In particular, we mention the work of \cite{KU1} where the recovery of  magnetic Schr\"odinger operators has been addressed for bounded electromagnetic potentials which is the weakest regularity assumption so far for general bounded domains. Let us also observe that there is a strong connection between this problem and the so called Calder\'on problem studied by \cite{CDR1,CDR2, CM,DKLS,KSU,SU} and extended to the non-smooth setting in \cite{AP,CR,H2015,HT}.

Several authors considered also the determination of time-dependent coefficients appearing in parabolic equations. In \cite{I1}, the author extended the construction of complex geometric optics solutions, introduced by \cite{SU}, to various PDEs including hyperbolic and parabolic equations to prove density of products of solutions. From the results of \cite{I1} one can deduce the unique determination of a coefficient  $q$ depending on both space and time variables, when $A=B=0$, from measurements on the lateral boundary $\Sigma$ with additional knowledge  of all solutions on $\Omega^0$ and $\Omega^T$. In Subsection 3.6 of \cite{Ch}, the  author extended the uniqueness result of \cite{I1} to a log-type stability estimate.  In the special case of cylindrical domain, \cite{GK} proved recovery of a time-dependent coefficient, independent of one spatial direction, from single boundary measurements. In \cite{CK} the authors addressed recovery of a parameter depending only on the time variable from single boundary measurements. More recently, \cite{CK2} proved that the result of \cite{I1} remains true from measurements given by $\Lambda_{A,B,q}$ when $A=B=0$. More precisely, \cite{CK2} proved, what seems to be, the first result of stability in the determination of a coefficient, depending on the space variable, appearing in a parabolic equation with measurements restricted to the lateral boundary $\Sigma$. We recall also the works of \cite{BB,A,HK,Ki2,Ki3,KO} related to the recovery of  time-dependent coefficients for hyperbolic equations and the stable recovery of coefficients appearing in Schr\"odinger equations established by \cite{CKS,KS}.

For the recovery of nonlinear terms, we mention the series of works  \cite{I2,I3,I4}  of Isakov dedicated to this problem for elliptic and parabolic equations. In \cite{I2,I3} the author considered the recovery of a semi-linear term of the form $F(x,u)$ inside the domain (i.e  $F(x,u)$ with $x\in\Omega$, $u\in \R$) or restricted to the lateral boundary (i.e  $F(x,u)$ with $x\in\partial\Omega$, $u\in \R$)  while in \cite{I4} he considered the recovery of  a quasilinear term of the form $F(u,\nabla_xu)$. In all these works, the approach developed by Isakov is based on a linearization of the inverse problem for nonlinear equations and results based on recovery of coefficients for linear equations. More precisely, in \cite{I2} the author used his work \cite{I1}, related to the recovery of a time-dependent coefficient $q$ on $Q$, while in \cite{I3,I4} he used results of recovery of coefficients on the lateral boundary $\Sigma$. The approach of Isakov, which seems to be the most efficient for recovering general nonlinear terms from boundary measurements, has also been considered by  \cite{IN,SuU} for the recovery of more general  nonlinear terms appearing in  nonlinear elliptic equations and by \cite{Ki5} who proved, for what seems to be the first time, the recovery of a general semi-linear term appearing in a semi-linear hyperbolic equation from boundary measurements. In \cite{CK2}, the authors proved a log-type stability estimate associated with the uniqueness result of \cite{I2} but with measurements restricted only to the lateral boundary $\Sigma$.  Finally, for results stated with single measurements we refer to  \cite{COY,KK,Kl} and for results stated  with measurements given by the source-to-solution map associated with semilinear hyperbolic equations  we refer to  \cite{HUW,KLU,KLOU,LUW}.

\subsection{Main result for the linear problem}

Our main result for the linear IBVP \eqref{eq1}, corresponds to the recovery of partial information of the gauge class of $(A,B,q)$ given by the relation \eqref{gauge1}. This result can be stated as follows.

\begin{Thm}\label{t5} 
For $j=1,2$, let    $q_j\in L^\infty(0,T;L^{p}(\Omega))\cup \mathcal C([0,T];L^{\frac{2n}{3}}(\Omega))$, with $p >2n / 3$,  and let $A_j,B_j\in L^\infty(Q)^n$.
The condition 
\bel{t5a} \Lambda_{A_1,B_1,q_1}=\Lambda_{A_2,B_2,q_2}\ee
implies 
\bel{t5b}dA_1=dA_2.\ee  
Here for $A=(a_1,\ldots,a_n)$, $dA$ denotes the $2$-form given by
$$dA:=\sum_{1\leq i<j\leq n} (\partial_{x_i}a_j-\partial_{x_j}a_i)dx_i\wedge dx_j.$$   
Let us also consider the additional conditions
\bel{ttt5a} A_1-A_2\in W^{1,\infty}(0,T; L^{p_1}(\Omega))^n,\quad \nabla_x\cdot(A_1-A_2),\ \nabla_x\cdot(B_1-B_2),\  q_1-q_2\in L^\infty(Q),\ee
\bel{tttt5b} (B_1-B_2)\cdot \nu_{|\Sigma}=2(A_2-A_1)\cdot\nu_{|\Sigma},\ee
where $(B_1-B_2)\cdot \nu$ $($resp. $(A_1-A_2)\cdot \nu$$)$ corresponds to the normal trace of $B_1-B_2$ $($resp. $(A_1-A_2)$$)$ restricted to $\Sigma$ which is well defined as an element of $L^\infty(0,T;\mathcal B(H^{\frac{1}{2}}(\partial\Omega);H^{-\frac{1}{2}}(\partial\Omega)))$.
Assuming that \eqref{ttt5a}-\eqref{tttt5b} are fulfilled,  \eqref{t5a} implies that there exists $\phi\in W^{1,\infty}(Q)$ such that
\bel{t5c}\left\{\begin{array}{ll}A_2=A_1+2\nabla_x\phi,\quad &\textrm{in}\ Q,\\  \nabla_x\cdot B_2+q_2=\nabla_x\cdot (B_1+\nabla_x\phi)+q_1-\pd_t\phi-|\nabla_x\phi|^2-A_1\cdot\nabla_x\phi,\quad &\textrm{in}\ Q,\\ \phi=0,\quad &\textrm{on}\ \Sigma.\end{array}\right.\ee
\end{Thm}

Note that, if we assume that $B_2=B_1+\nabla_x\phi$, condition \eqref{t5a} implies that $(A_1,B_1,q_1)$ and $(A_2,B_2,q_2)$ are gauge equivalent in the sense of \eqref{gauge1}. Since in \eqref{t5a} the coefficients 
$(B_1,q_1)$ and $(B_2,q_2)$ are in relation through one equality, without additional assumptions, there is no hope to deduce from \eqref{t5a} that $(A_1,B_1,q_1)$ and $(A_2,B_2,q_2)$ are gauge equivalent. According to our discussions of Section 1.3, this is the best one can expect to recover from the DN map. However, with additional assumptions on the low order coefficients $(B,q)$, this result will lead to the full recovery of the convection term $A$. This consequence of the main result can be stated as follows.
\begin{cor}\label{c1} Let $\Omega$ be connected and let    $A_1,A_2\in  L^\infty(Q)^n$ be such that $\nabla_x\cdot(A_1-A_2) \in  L^\infty(Q)$ and $(A_2-A_1)\cdot\nu_{|\Sigma}=0$. 
Then, for any $q\in L^\infty(0,T;L^{p}(\Omega))\cup \mathcal C([0,T];L^{\frac{2n}{3}}(\Omega))$, $p > 2n/3$, and $B\in L^\infty(Q)^n$, we have
$$\Lambda_{A_1,B,q}=\Lambda_{A_2,B,q}\Rightarrow A_1=A_2.$$\end{cor}

Let us mention that there is another way to formulate convection or advection-diffusion equations  given by the following IBVP
\begin{equation}\label{eeq1}\left\{\begin{array}{ll}\partial_tv-\Delta_x v+\nabla_x\cdot(A(x,t)v)=0,\quad &\textrm{in}\ Q,\\  v(0,\cdot)=0,\quad &\textrm{in}\ \Omega,\\ v=g,\quad &\textrm{on}\ \Sigma.\end{array}\right.\end{equation}
The corresponding inverse problem consists in recovering the velocity field described by the coefficient $A$ from the associated DN map $$\tilde{\Lambda}_A:g\mapsto \tilde{N}_Av,$$
where, for $v\in L^2(0,T;H^1(\Omega))\cap H^1(0,T;H^{-1}(\Omega))$ solving \eqref{eeq1} and for $w\in H^1(Q)$ satisfying $w_{|\Omega^T}=0$, $\tilde{N}_Av$ is defined by
$$\left\langle \tilde{N}_Av,w_{|\Sigma} \right\rangle:=\int_Q[-u\pd_tw+\nabla_xu\cdot\nabla_xw+\frac{1}{2}(A\cdot\nabla_xu)w-\frac{1}{2}u(A\cdot\nabla_xw)+\frac{1}{2}\nabla_x\cdot(A)uw]dxdt.$$
Using the identity $\tilde{\Lambda}_A=\Lambda_{A,\frac{A}{2},\frac{\nabla_x\cdot(A)}{2}}$ and applying Theorem \ref{t5} we obtain the following.

\begin{cor}\label{c11} Let $\Omega$ be connected and let   $A_1,A_2\in  L^\infty(Q)^n$ be such that $$\nabla_x\cdot(A_1),\ \nabla_x\cdot(A_2)\in L^\infty(0,T;L^{p}(\Omega))\cup \mathcal C([0,T];L^{\frac{2n}{3}}(\Omega)), \quad p > 2n/3 .$$ Assume also that  \eqref{ttt5a}-\eqref{tttt5b} are fulfilled, for $B_j=\frac{A_j}{2}$ and $q_j=\frac{\nabla_x\cdot(A_j)}{2}$, $j=1,2$. Then,  the condition $\tilde{\Lambda}_{A_1}=\tilde{\Lambda}_{A_2}$ implies $A_1=A_2$.\end{cor}

In the spirit of \cite{AU}, by assuming that the coefficients are known in the neighborhood of $\Sigma$, we can improve Theorem \ref{t5} into the recovery of the coefficients from measurements in an arbitrary portion of the boundary. More precisely, for any open set $\gamma$ of  $\partial\Omega$, we denote by $\mathcal H_\gamma$ the subspace of $H^1(Q)$ given by
$$\mathcal H_\gamma:=\{h_{|\Sigma}:\ h\in H^1(Q),\ h_{|\Omega^T}=0,\ \textrm{supp}(h_{|\Sigma})\subset\gamma\times[0,T]\}.$$
We fix $\gamma_1, \gamma_2$  two arbitrary open and not empty subsets of $\partial\Omega$.
Then,  we can consider, for $A, B\in  L^\infty(Q)^n$ and $q\in L^\infty(0,T;L^{p_1}(\Omega))$, the partial DN map
$$\Lambda_{A,B,q,\gamma_1,\gamma_2}:\mathcal H_+\cap \mathcal E'(\gamma_1\times [0,T])\ni g\mapsto N_{A,B,q}u_{|\mathcal H_{\gamma_2}},$$
with $u$ the solution of \eqref{eq1} and $\mathcal H_+$ the space defined in Section 2.
 Then, we can improve Theorem \ref{t5}  in the following way.

\begin{cor}\label{c3} Let $\Omega$ be connected. We fix $q_1,q_2\in L^\infty(0,T;L^{p_1}(\Omega))$, with $p_1$ given by \eqref{p}, and we consider $A_j,B_j\in  L^\infty(Q)^n$, $j=1,2$, satisfying $\nabla_x\cdot(A_j),\nabla_x\cdot(B_j)\in L^\infty(0,T;L^{p_1}(\Omega))$. Assume that there exists an open connected set $\Omega_*\subset\Omega$, satisfying $\partial\Omega\subset\partial\Omega_*$, such that
\bel{c3a} A_1(x,t)=A_2(x,t),\quad B_1(x,t)=B_2(x,t),\quad q_1(x,t)=q_2(x,t),\quad (x,t)\in \Omega_*\times (0,T).\ee
 Then the condition
\bel{c3b}\Lambda_{A_1,B_1,q_1,\gamma_1,\gamma_2}=\Lambda_{A_2,B_2,q_2,\gamma_1,\gamma_2}\ee
implies  that $dA_1=dA_2$. If in addition \eqref{ttt5a} is fulfilled,
 \eqref{c3b} implies that there exists $\phi\in W^{1,\infty}(Q)$ satisfying \eqref{t5c}.\end{cor}

\subsection{Recovery of quasilinear terms}
In this subsection, we will state our results related to the recovery of general quasi-linear terms $F(x,t,u,\nabla_xu)$ appearing in \eqref{eq111}.
We denote by $\Sigma_p$ the parabolic boundary of $Q$ defined by $\Sigma _p=\Sigma \cup\Omega^0$. Moreover, for all $\alpha\in(0,1)$, we denote by
$\mathcal C^{\alpha,\frac{\alpha}{2}}(\overline{Q})$ the space of functions $f\in\mathcal C(\overline{Q})$ satisfying
$$[f]_{\alpha,\frac{\alpha}{2}}=\sup\left\{\frac{|f(x,t)-f(y,s)|}{(|x-y|^2+|t-s|)^{\frac{\alpha}{2}}}:\ (x,t),(y,s)\in\overline{Q},\ (x,t)\neq(y,s)\right\}<\infty.$$
Then we define the space $\mathcal C^{2+\alpha,1+\frac{\alpha}{2}}(\overline{Q})$ as the set of functions $f$ lying in $\mathcal C([0,T];\mathcal C^2(\overline{\Omega}))\cap \mathcal C^1([0,T];\mathcal C(\overline{\Omega}))$ such that $$\partial_tf,\partial_x^\beta f\in \mathcal C^{\alpha,\frac{\alpha}{2}}(\overline{Q}),\quad \beta\in\mathbb N^n,\ |\beta|=2.$$
We consider on these spaces the usual norm and we refer to \cite[pp. 4]{Ch} for more details.
We consider the nonlinear parabolic equation
\begin{equation}\label{1.1}
\left\{
\begin{array}{ll}
(\partial_tu-\Delta u+F(x,t,u,\nabla_xu)=0\quad &\mbox{in}\; Q,
\\
u=G &\mbox{on}\; \Sigma_p .
\end{array}
\right.
\end{equation}
For $\alpha\in(0,1)$ and $\Omega$ a $\mathcal C^{2+\alpha}$ bounded domain, $F:(x,t,u,v)\mapsto F(x,t,u,v)\in C^1(\overline{Q}\times\mathbb{R} \times\R^n)$ satisfying \eqref{nl1}-\eqref{nl3}, $G\in \mathcal X=\{K_{|\overline{\Sigma}_p};\;\  K\in C^{2+\alpha ,1+\alpha /2}(\overline{Q}),\ (\partial_tK-\Delta K)_{|\partial\Omega\times\{0\}}=0\}$,  problem \eqref{1.1} admits a unique solution $u_{F,G}\in C^{2+\alpha ,1+\alpha /2}(\overline{Q})$ (see Section 6 for more detail). Then, for $\nu$ the outward unit normal vector to $\pd\Omega$, we can introduce the DN map associated with \eqref{1.1} given by
$$\mathcal N_F:\mathcal X\ni G\mapsto {\partial_\nu u_{F,G}}_{|\Sigma}\in L^2(\Sigma)$$
and we consider the recovery of $F$ from partial knowledge of $\mathcal N_F$. More precisely, we prove in Proposition \ref{pS1} that for $\partial_uF\in \mathcal C^1(\overline{Q}\times\R^n\times \mathbb{R};\R)$ and  $\partial_vF\in  \mathcal C^1(\overline{Q}\times\R^n\times \mathbb{R};\R^n)$, $\mathcal N_F$ is continuously Fr\'echet differentiable. Then, fixing 
$$\mathcal X_0:=\{G\in\mathcal X:\ G_{|\Omega^0}=0\},\quad k_v:x\mapsto x\cdot v,\quad h_{a,v}:x\mapsto x\cdot v+a,$$
where $a\in\R$, $v\in\R^n$, we consider the recovery of $F$ from $$\mathcal N'_F(k_v|_{\Sigma_p})H\quad \textrm{and }\quad\mathcal N'_F(h_{a,v}|_{\Sigma_p})H,\quad H\in\mathcal X_0,\ a\in\R,\ v\in\R^n,$$
where $\mathcal N'_F$ denotes the Fr\'echet differentiation of $\mathcal N_F$.

We obtain two main results for this problem. In our first main result, we are interested in the recovery of information about general nonlinear terms of the form $F(x,t,u,\nabla_xu)$ form the knowledge of 
$$\mathcal N'_F(h_{a,v}|_{\Sigma_p})H,\quad H\in\mathcal X_0,\ a\in\R,\ v\in\R^n.$$
 Our first main result can be stated as follows.

\begin{Thm}\label{tnl3} Let $\Omega$ be a $\mathcal C^{2+\alpha}$ bounded and connected domain and let  $F_1, F_2\in C^{2+\alpha,1+\frac{\alpha}{2}}(\overline{Q};\mathcal C^3(\mathbb{R} \times\R^n))$ satisfy \eqref{nl1}-\eqref{nl3}. Let also, for $j=1,2$,    $\partial_vF_j\in  \mathcal C^{1+\frac{\alpha}{2}}([0,T]; \mathcal C^1(\overline{\Omega}\times\R^n\times \mathbb{R});\R^n)$ and let
\bel{tnl1aa}\partial_x^\ell F_j(x,0,u,v)=0,\quad    j=1,2,\ k=0,1,\ x\in\partial\Omega,\ u\in\R,\ v\in\R^n,\ \ell\in\mathbb N^n,\ |\ell|\leq2.\ee
Then, the condition
\bel{tnl3a}\mathcal N'_{F_1}(h_{a,v}|_{\Sigma_p})H=\mathcal N'_{F_2}(h_{a,v}|_{\Sigma_p})H,\quad H\in\mathcal X_0,\  a\in\R,\ v\in\R^n\ee
imply that there exists $$\phi:Q\times\R\times\R^n\ni (x,t,u,v)\mapsto\phi(x,t,u,v)\in\mathcal C^1([0,T];\mathcal C(\overline{\Omega}\times\R\times\R^n))\cap\mathcal C^2(\overline{\Omega};\mathcal C([0,T]\times\R\times\R^n))$$ such that, for all $(u,v)\in \R\times\R^n$, we have
\bel{tnl3b}\left\{\begin{array}{ll}\partial_v(F_2-F_1)(x,0,u,v)=2\partial_x\phi(x,0,u-x\cdot v,v),\  &x\in \Omega,\\  \partial_u(F_2-F_1)(x,0,u,v)=-(\pd_t\phi-|\nabla_x\phi|^2-\partial_vF_1(x,0,u,v)\partial_x\phi)(x,0,u-x\cdot v,v),\  &x\in \Omega,\\ \phi(x,t,u,v)=0,\ &\ (x,t)\in \Sigma.\end{array}\right.\ee
 \end{Thm}
From this result, we deduce the following.
\begin{cor}\label{cnl2} Let the condition of Theorem \ref{tnl3} be fulfilled. Assume also that, for all $x\in \Omega,\ u\in\R,\ v\in\R^n$, the following condition 
\bel{cnl2a}\sum_{j=1}^n\left[\partial_{x_j}\partial_{v_j}F_1(x,0,u,v)+\partial_{u}\partial_{v_j}F_1(x,0,u,v)v_j\right]=\sum_{j=1}^n\left[\partial_{x_j}\partial_{v_j}F_2(x,0,u,v)+\partial_u\partial_{v_j}F_2(x,0,u,v)v_j\right]\ee
is fulfilled. Then condition \eqref{tnl3a} implies
\bel{tnl1c}\partial_vF_1(x,0,u,v)=\partial_vF_2(x,0,u,v),\quad x\in\Omega,\  u\in\R,\ v\in\R^n.\ee
In particular, if there exists $v_0\in\R^n$ such that
\bel{tnl1e}F_1(x,0,u,v_0)=F_2(x,0,u,v_0),\quad x\in\Omega,\ u\in\R,\ee
and \eqref{tnl1a} are fulfilled, then condition \eqref{tnl3a} implies 
\bel{tnl1f}F_1(x,0,u,v)=F_2(x,0,u,v),\quad x\in \Omega,\ u\in\R,\ v\in\R^n.\ee\end{cor}

\begin{rem}\label{r1} The result of Corollary \ref{cnl2} can be applied to the unique full recovery of quasilinear terms of the form \bel{cn1}F(x,t,u,\nabla_xu)=G_1(x,u,\nabla_xu)+tG_2(x,t,u,\nabla_xu),\ee 
with $G_2$ and
$$H:(x,u,v)\mapsto \sum_{j=1}^n\left[\partial_{x_j}\partial_{v_j}G_1(x,u,v)+\partial_{u}\partial_{v_j}G_1(x,u,v)v_j\right]$$
two known functions.\end{rem}

For our second main result we consider the full recovery of the nonlinear term $F(x,t,u,\nabla_xu)$ from the data
$$\mathcal N'_F(k_v|_{\Sigma_p})H,\quad H\in\mathcal X_0,\ v\in\R^n.$$
For this purpose, taking into account the natural invariance for the recovery of such nonlinear terms, described by condition \eqref{tnl3b}, our result will require some additional assumptions on the  class of nonlinear terms under consideration. Our   second main result related to this problem can be stated as follows.
\begin{Thm}\label{tnl1} Let $\Omega$ be a $\mathcal C^{2+\alpha}$ bounded and connected domain and let  $F_1, F_2\in C^{2+\alpha,1+\frac{\alpha}{2}}(\overline{Q};\mathcal C^3(\mathbb{R} \times\R^n))$ satisfy \eqref{nl1}-\eqref{nl3}. Let also, for $j=1,2$,    $\partial_vF_j\in  \mathcal C^{1+\frac{\alpha}{2}}([0,T]; \mathcal C^1(\overline{\Omega}\times\R^n\times \mathbb{R});\R^n)$ and let \eqref{tnl1aa} be fulfilled.
Then, the conditions
\bel{tnl1b}\mathcal N'_{F_1}(k_v|_{\Sigma_p})H=\mathcal N'_{F_2}(k_v|_{\Sigma_p})H,\quad H\in\mathcal X_0,\  v\in\R^n\ee
and 
\bel{tnl1a} \partial_uF_1(x,t,u,v)=\partial_uF_2(x,t,u,v),\quad (x,t)\in Q,\ u\in\R,\ v\in\R^n,\ee
\bel{tnl2a}\partial_u^2F_1(x,t,u,v)=0,\ \partial_v\partial_uF_1(x,t,u,v)=0,\quad (x,t)\in Q,\ u\in\R,\ v\in\R^n,\ee
imply \eqref{tnl1c}. In addition, if there exists $v_0\in\R^n$ such that \eqref{tnl1e}
and \eqref{tnl1a} are fulfilled, then condition \eqref{tnl1b} implies \eqref{tnl1f}.\end{Thm}

\begin{rem}\label{r2} The result of Theorem \ref{tnl1} can be applied to the unique full recovery of quasilinear terms of the form \bel{cn2}F(x,u,\nabla_xu)=G_1(x,\nabla_xu)+G_2(x,t)u,\ee
when $G_2$ is known.\end{rem}

A direct consequence of Theorem \ref{tnl1} will be a partial data result in the spirit of Corollary \ref{c3}. More precisely, for any open set $\gamma$ of $\partial\Omega$, we define $\mathcal N_{F,\gamma}$ by
$$\mathcal N_{F,\gamma}G:=\mathcal N_F(G)_{|\gamma\times(0,T)},\quad G\in \mathcal X.$$
We denote also by $\mathcal X_{0,\gamma}$ the set $\mathcal X_{0,\gamma}:=\{G\in\mathcal X_0:\ \textrm{supp}(G_{|\Sigma})\subset \gamma\times[0,T]\}$.
Then, we deduce from Theorem \ref{tnl1} the following result.

\begin{cor}\label{cnl} Let the condition of Theorem \ref{tnl1} be fulfilled. 
Let $\Omega_*$ be an open connected subset of $\Omega$ satisfying $\partial\Omega\subset\partial\Omega_*$ and let $\gamma_1,\gamma_2$ be two arbitrary open subset of $\partial\Omega$. We assume that $F_1,F_2$ fulfill 
\bel{cnl1a}\partial_vF_1(x,t,u,v)=\partial_vF_2(x,t,u,v)=0,\quad (x,t)\in \Omega_*\times(0,T),\ u\in\R,\ v\in\R^n\ee
and \eqref{tnl1a}. Then the condition
\bel{cnl1b}\mathcal N'_{F_1,\gamma_2}(k_v|_{\Sigma_p})H=\mathcal N'_{F_2,\gamma_2}(k_v|_{\Sigma_p})H,\quad H\in\mathcal X_{0,\gamma_1},\  v\in\R^n,\ee
implies \eqref{tnl1c}.\end{cor}

\subsection{Comments about our results}

Let us first observe that, to the best of our knowledge, Theorem  \ref{t5} and its consequences, stated in Corollary \ref{c1} and \ref{c11}, are the first results of unique determination of general convection  terms  depending on both time and space variables. Actually, in Theorem  \ref{t5} we prove the  recovery of information  about the three coefficients $A,B,q$, provided  by   \eqref{t5c}, from $\Lambda_{A,B,q}$. According to the obstruction described in Section 1.3, this is the best one can expect for the simultaneous recovery of the three coefficients $A,B,q$.  Assuming that coefficients $B$ is known, Theorem  \ref{t5}  would be equivalent to the 
 unique determination of $A,q$ modulo the gauge invariance given by \eqref{t5c} with $B_1=B_2=0$. Moreover, combining the result of Theorem  \ref{t5} with unique continuation results, we obtain in Corollary \ref{c1} and \ref{c11} the full recovery of the convection term $A$ when $B,q$ are fixed. Note also that, in contrast to time-independent coefficients, our inverse problem can not be reduced to the recovery of coefficients  appearing in a steady state convection-diffusion equation from the associated DN map.

Not only Theorem \ref{t5} provides, for what seems to be the first time, a result of recovery of general first and zero order time-dependent coefficients appearing in a parabolic equation but it is also stated in a non smooth setting. Indeed, we only require the two vector valued coefficients $A,B$ to be bounded and we allow $q$ to have singularities with respect to the space variable. Moreover, we state our result in a general Lipschitz domain $\Omega$. Such general setting  make Theorem \ref{t5} suitable for many potentials application and the regularity of the coefficients $A,B,q$ can be compared to \cite{KU1} where one can find the best result known so far, in terms of regularity of the coefficients, about the recovery of similar coefficients for elliptic equations in a general bounded domain (see also \cite{HH}). Note that, assuming that $A,B$ are known and $A\in L^\infty(0,T;W^{2,\infty}(\Omega))^n\cap W^{1,\infty}(0,T;L^\infty(\Omega))^n$, $\nabla_x\cdot B\in L^\infty(Q)$, we can prove the recovery of more general zero order coefficient $q$. Actually, in that context,  using our approach, one can prove the recovery of coefficients $q$ lying in $L^\infty(0,T;L^{p}(\Omega))\cup \mathcal C([0,T];L^{\frac{2n}{3}}(\Omega))$, with $p > 2n/3$. However, like for elliptic equations (see \cite{KU1}) we can not reduce simultaneously the smoothness assumptions for the first and zero order coefficients under consideration. For this reason, we have proved first the recovery of the 2-form $dA$ associated with the convection term $A$ with the weakest regularity that allows our approach for all the coefficients $A$, $B$, $q$. Then, we have proved the recovery of information about the coefficients $(A,B,q)$, given by \eqref{t5c}, by increasing the regularity of the unknown part of the coefficients $B$ and $q$ (see \eqref{ttt5a}-\eqref{tttt5b}).

One of the main tools for the proof of Theorem \ref{t5}  are suitable solutions of \eqref{eq1} also called geometric optics (GO in short) solutions. Similar type of solutions have already been considered by \cite{CK2,I1} for the recovery of bounded zero order coefficients $q$. None of these constructions work with arbitrary variable coefficients of order 1 or non-bounded coefficient $q$. Therefore, we introduce a new construction, inspired by the approach of \cite{DKSU,KSU,KU1,STz} for elliptic equations, in order to overcome the presence of variable coefficients of order 1. More precisely, we derive first a new Carleman estimate stated  in Proposition \ref{pp1} from which we obtain  Carleman estimates in negative order Sobolev space stated in Proposition \ref{l1}, \ref{l8}. Applying Proposition \ref{l1}, \ref{l8}, we built our GO solutions by a duality argument and an application of the Hahn Banach theorem. In contrast to the analysis of  \cite{DKSU,KU1,STz} for elliptic equations, we need to consider  GO solutions that vanish on the top $\Omega^T$ or on the bottom $\Omega^0$ of the space-time cylindrical domain $Q$. For this purpose, we freeze the time variable and we work only with respect to the space variable for the construction of our GO solutions. Then, using the estimate  on $\Omega^T$ or  $\Omega^0$ of the Carleman estimates of  Proposition \ref{l1} we can apply Proposition \ref{l1}, \ref{l8} to functions vanishing only at  $t=T$ or  $t=0$. This additional constraint on $\Omega^T$ or  $\Omega^0$, makes an important difference between the construction of the so called  complex geometric optics solutions considered by \cite{DKSU,KU1,STz} for elliptic equations and our construction of the GO solutions for parabolic equations. Like in \cite{KU1} for elliptic equations, thanks to the estimate of the Laplacian in Proposition \ref{pp1},  we can apply our construction to coefficients with low regularity. Actually, we improve the construction of \cite{KU1} by extending our approach to unbounded zero order coefficients $q$. Note also that, quite surprisingly, in Proposition \ref{l1}, \ref{l1}, \ref{l8}, we obtain better estimates with respect to the space variables than what has been proved in \cite{KU1,STz},  for the 3-dimensinal case an averaging procedure provides an equivalent gain to ours, see \cite{HH}.

From the recovery of   $(A,B,q)$, in the sense  of \eqref{t5c}, stated in Theorem \ref{t5}, we derive three different results for the linear problem stated in Corollary \ref{c1}, \ref{c11}, \ref{c3}. In all these three results, we use unique continuation results for parabolic equations in order to derive conditions that guaranty $\phi=0$ in \eqref{t5c} or to obtain a density arguments in norm $L^2$ on a subdomain of $Q$. Using such arguments we can prove the full recovery of the convection term $A$ and prove the recovery of the gauge class of $(A,B,q)$ from measurements on an arbitrary portion of $\partial\Omega$ when $(A,B,q)$ are known on some neighborhood of $\Sigma$.

According to \cite[Lemma 8.1]{I4}, with additional regularity assumptions imposed to the coefficients $(A,B,q)$ and to the domain $\Omega$, the DN map $\Lambda_{A,B,q}$ determines $A\cdot\nu$ on $\Sigma$. Therefore, for sufficiently smooth coefficients $(A_j,B,q)$, $j=1,2$, and sufficiently smooth domain $\Omega$, the condition \eqref{tttt5b} can be removed from the statement of Corollary \ref{c1} and \ref{c11}. We believe that  the condition \eqref{tttt5b} can also be removed with  less regular coefficients and domain. However, we do not treat that issue in the present paper.

To the best of our knowledge, in Theorem \ref{tnl3}  and \ref{tnl1} we have stated the first results of recovery of a general quasi-linear  term  of the form $F(x,t,u,\nabla_xu)$ with $(x,t)\in Q$, that admits variation independent of the solutions inside the domain (i.e we recover the part $F(x,0,u,v)$ with $x\in\Omega$, $u\in\R$, $v\in\R^n$ of such functions) from measurements restricted to the lateral boundary. Indeed, one can apply our results to the unique  full recovery of nonlinear terms of the form \eqref{cn1} and \eqref{cn2} (see Remark \ref{r1} and \ref{r2} for more details). The only other comparable result is the one stated in \cite{I4} where the author proved the recovery of quasilinear terms $F(u,\nabla_xu)$ depending only on the solutions  on the abstract set
$$E:=\{(u(x,t),\nabla_xu(x,t)):\ (x,t)\in\Sigma,\ \partial_tu-\Delta u+F(u,\nabla_xu)=0\}.$$
 Therefore, our results, which correspond to   global recovery results, provide a more precise  information about the nonlinear terms under consideration than \cite{I4} where the uniqueness result is stated on the above  set $E$ which is not explicitly given. Moreover, our results  can  also be applied to  more general quasi-linear terms admitting variation independent of the solution inside the domain, while \cite{I4} restricts its analysis to  quasi-linear terms depending only on the solutions.

We prove Theorem \ref{tnl3}  and \ref{tnl1} by combining Theorem \ref{t5} and Corollary \ref{c1} with the linearization procedure described in  \cite{CK2,I3,I4,I5}. More precisely, we transform the recovery of the nonlinear term $F(x,t,u,\nabla_xu)$ into the recovery of time-dependent coefficients $A(x,t)=\partial_vF(x,t,u(x,t),\nabla_xu(x,t))$ and $q(x,t)=\partial_uF(x,t,u(x,t),\nabla_xu(x,t))$, where $u$ solves \eqref{1.1}. Here the variable $v\in\R^n$ corresponds to $\nabla_xu$ in the expression $F(x,t,u,\nabla_xu)$. In contrast to all other results stated for nonlinear parabolic equations (e.g. \cite{CK2,I2,I3,I4}), we do not need the full map  $\mathcal N_F$ for proving the recovery of the quasi-linear terms but only some partial knowledge of its Fr\'echet derivative $\mathcal N'_F$. More precisely, our results require only the knowledge of $\mathcal N'_F$ applied to restriction of linear or affine functions on $\Sigma_p$. By taking into account the important amount of information contained in the map $\mathcal N_F$, this makes an important restriction on the data used for solving the inverse problem. For this purpose, in contrast to \cite{I2,I3,I4}, we need to explicitly derive the Fr\'echet derivative of $\mathcal N_F$. A similar idea has been considered in \cite{Ki5} for the recovery of a semilinear term appearing in nonlinear hyperbolic equations. 

The result of Theorem \ref{tnl3} is stated for more general quasi-linear terms than the one of Theorem \ref{tnl1}. However, Theorem \ref{tnl3} can not be applied directly to the full recovery of the nonlinear term like in Theorem \ref{tnl1}. Indeed, Theorem \ref{tnl3} provide only some knowledge of the nonlinear term $F(x,t,u,\nabla_xu)$ given by the conditions \eqref{tnl3b}. On the other hand, with the additional condition \eqref{cnl2a}, we can derive form Theorem \ref{tnl3} the unique full recovery stated in Corollary \ref{cnl2}.

Applying Corollary \ref{c3}, we  also prove in Corollary \ref{cnl} recovery of nonlinear terms known on the neighborhood of the boundary from measurements on some arbitrary portion of the boundary $\partial\Omega$.
\subsection{Outline}

This paper is organized as follows. In Section 2, we obtain some preliminary results related to problem \eqref{eq1} and we give a rigorous definition of the DN map $\Lambda_{A,B,q}$. In Section 3, we derive a Carleman estimate associated with \eqref{eq1}. Section 4 is devoted to the construction of suitable GO solutions for \eqref{eq1} and its adjoint problem. The construction of these GO solutions requires Carleman estimates in negative order Sobolev spaces that we derive by applying the results of Section 3. In Section 5, we complete the proof of the results  for the linear problem stated in Theorem \ref{t5} and Corollary \ref{c1}, \ref{c11}, \ref{c3}. In Section 6, we prove our results for the nonlinear problem \eqref{eq111} stated in Theorem \ref{tnl3}, \ref{tnl1} and Corollary \ref{cnl2}, \ref{cnl}.

\section{Preliminary results}
We recall that $\Omega^0=\Omega\times\{0\}\subset Q$ and $\Omega^T=\Omega\times\{T\}\subset Q$.
Let us first consider the spaces 
$$\mathcal H_+:=\{v_{|\Sigma}:\ v\in H^1(0,T;H^{-1}(\Omega))\cap L^2(0,T;H^1(\Omega)),\ v_{|\Omega^0}=0\},$$ $$\mathcal H_-:=\{v_{|\Sigma}:\ v\in H^1(0,T;H^{-1}(\Omega))\cap L^2(0,T;H^1(\Omega)),\ v_{|\Omega^T}=0\}$$
which are subspaces of $L^2(0,T;H^{\frac{1}{2}}(\pd\Omega))$. We  introduce also the spaces
$$S_+=\{u\in H^1(0,T;H^{-1}(\Omega))\cap L^2(0,T;H^1(\Omega)):\ (\pd_t-\Delta_x)u=0,\ u_{|\Omega^0}=0\},$$
$$S_-=\{u\in H^1(0,T;H^{-1}(\Omega))\cap L^2(0,T;H^1(\Omega)):\ (-\pd_t-\Delta_x)u=0,\ u_{|\Omega^T}=0\}.$$
In order to define an appropriate topology  on $\mathcal H_\pm$ for our problem, we consider the following result.

\begin{prop}\label{p1} For all $f\in \mathcal H_\pm$  there exists a unique $u\in S_\pm$ such that $ u_{|\Sigma}=f$.\end{prop}
\begin{proof} Without lost of generality we assume that the functions are real valued. We will only prove the result for $f\in\mathcal H_+$, using similar arguments one can extend the result to $f\in\mathcal H_-$. Let $f\in \mathcal H_+$ and consider $F \in H^1(0,T;H^{-1}(\Omega))\cap L^2(0,T;H^1(\Omega))$ such that $ F_{|\Sigma}=f$ and $F_{|\Omega^0}=0$. Fix $G=-(\pd_t-\Delta_x)F\in L^2(0,T;H^{-1}(\Omega))$ and $w$ the solution of the IBVP
$$\left\{ \begin{array}{rcll} \partial_tw-\Delta_x w& = & G, & (x,t) \in Q ,\\ 

w_{\vert \Omega^0}&=&0,&\\
w_{\vert\Sigma}& = & 0.& \end{array}\right.$$
According to \cite[Theorem 4.1, Chapter 3]{LM1} this  IBVP  admits a unique solution $w\in H^1(0,T;H^{-1}(\Omega))\cap L^2(0,T;H^1(\Omega))$. Thus $v=w+F\in S_+$ and clearly $v_{|\Sigma}=w_{|\Sigma}+F_{|\Sigma}=f$. This prove the existence of $u\in S_+$ such that $ u_{|\Sigma}=f$. For the uniqueness, let $v_1,v_2\in S_+$ satisfies $\tau_0v_1=\tau_0v_2$. Then, $v=v_1-v_2$ solves 
$$\left\{ \begin{array}{rcll} \partial_tv-\Delta_x v& = & 0, & (x,t) \in Q ,\\ 

v_{\vert \Omega^0}&=&0,&\\
v_{\vert\Sigma}& = & 0.& \end{array}\right.$$
which from the uniqueness of this IBVP implies that $v_1-v_2=0$.\end{proof}

Following Proposition \ref{p1}, we consider the norm on $\mathcal H_\pm$ given by
$$ \norm{ F_{\vert\Sigma}}_{\mathcal H_\pm}^2=\norm{F}_{L^2(0,T;H^1(\Omega))}^2+\norm{F}_{H^1(0,T;H^{-1}(\Omega))}^2,\quad F\in S_\pm.$$
We introduce  the IBVPs
\begin{equation}\label{eeeq1}\left\{\begin{array}{ll}\partial_tu-\Delta_x u+A(x,t)\cdot\nabla_xu+[\nabla_x\cdot B(x,t)] u+q(x,t)u=0,\quad &\textrm{in}\ Q,\\  u(0,\cdot)=0,\quad &\textrm{in}\ \Omega,\\ u=g_+,\quad &\textrm{on}\ \Sigma,\end{array}\right.\end{equation}
\begin{equation}\label{eq2}\left\{\begin{array}{ll}-\partial_tu-\Delta_x u-A(x,t)\cdot\nabla_xu+(q(x,t)+[\nabla_x\cdot (B-A)(x,t)])u=0,\quad &\textrm{in}\ Q,\\  u(\cdot,T)=0,\quad &\textrm{in}\ \Omega,\\ u=g_-,\quad &\textrm{on}\ \Sigma.\end{array}\right.\end{equation}

We are now in position to state existence and uniqueness of solutions of these IBVPs for $g_\pm\in \mathcal H_\pm$.

\begin{prop}\label{p2} Let $g_\pm\in \mathcal H_\pm$,  $A,B\in L^\infty(Q)^n$, $q\in L^\infty(0,T;L^{\frac{2n}{3}}(\Omega))$. Then, the IBVP \eqref{eeeq1} $($respectively \eqref{eq2}$)$ admits a unique weak solution $u\in  H^1(0,T;H^{-1}(\Omega))\cap L^2(0,T;H^1(\Omega))$ $($respectively $u\in H_-$$)$ satisfying 
\begin{equation}\label{p2a}
\begin{aligned}&\norm{u}_{L^2(0,T;H^1(\Omega))}+\norm{u}_{H^1(0,T;H^{-1}(\Omega))}\leq C\norm{g_+}_{\mathcal H_+}\\
 &(\textrm{respectively }\norm{u}_{L^2(0,T;H^1(\Omega))}+\norm{u}_{H^1(0,T;H^{-1}(\Omega))}\leq C\norm{g_-}_{\mathcal H_-}),\end{aligned}
\end{equation}
where $C$ depends on $\Omega$, $T$ and $M\geq\norm{q}_{L^\infty(0,T;L^{\frac{2n}{3}}(\Omega))}+\norm{A}_{L^\infty (Q)^n}$.

\end{prop}
 \begin{proof} Since the proof of  the well-posedness result is similar for \eqref{eeeq1} and \eqref{eq2}, we will only treat  \eqref{eeeq1}. According to Proposition \ref{p1}, there exists a unique  $G\in S_+$ such that $G_{|\Sigma}=g_+$ and  $$\norm{G}_{L^2(0,T;H^1(\Omega))}\leq\norm{g_+}_{\mathcal H_+}.$$ We split  $u$ into two terms $u=w+G$ where $w$ solves
\bel{eq5}\left\{ \begin{array}{rcll} \partial_tw-\Delta_x w+A\cdot\nabla_xw+(\nabla_x\cdot B) w+qw& = & -A\cdot\nabla_x G-(\nabla_x\cdot B) G-qG, & (x,t) \in Q ,\\ 
w_{\vert \Omega^0}&=&0,&\\
w_{\vert\Sigma}& = & 0.& \end{array}\right.\ee
From the Sobolev embedding theorem we have $-A\cdot\nabla_x G-(\nabla_x\cdot B) G-qG\in L^2(0,T;H^{-1}(\Omega))$ with 
$$\begin{aligned}&\norm{-A\cdot\nabla_x G-(\nabla_x\cdot B) G-qG}_{L^2(0,T;H^{-1}(\Omega))}\\
&\leq C\left(\norm{A}_{L^\infty (Q)^n}+\norm{B}_{L^\infty (Q)^n}+\norm{q}_{L^\infty(0,T;L^{\frac{2n}{3}}(\Omega))}\right)\norm{G}_{L^2(0,T;H^{1}(\Omega))},\end{aligned}$$
with $C$ depending only on $\Omega$. Let $H=L^2(\Omega)$, $V=H^1_0(\Omega)$ and consider the time-dependent sesquilinear form $a(t,\cdot,\cdot)$ with domain $V$ and defined by
$$a(t,h,g)=\int_\Omega \nabla_x h(x)\cdot\overline{\nabla_x g(x)}+(A(x,t)\cdot\nabla_x h(x)+q(x,t)h(x))\overline{g(x)}-B(x,t)\cdot\nabla_x(h\overline{g})(x)dx,\quad h,g\in V.$$
Note that here for all $h,g\in V$, we have $t\mapsto  a(t,h,g)\in L^\infty(0,T)$ and an application of the Sobolev embedding theorem implies
\bel{con}|a(t,h,g)|\leq C\norm{h}_{H^1(\Omega)}\norm{g}_{H^1(\Omega)},\ee
with $C>0$ depending on $\norm{A}_{L^\infty(Q)}$, $\norm{B}_{L^\infty(Q)}$ and $\norm{q}_{L^\infty(0,T;L^{\frac{2n}{3}}(\Omega))}$.
In addition, there exists $\lambda, c>0$ such that, for any $h\in V$, we have
 \bel{cor}\re a(t,h,h)+\lambda\norm{h}^2_{L^2(\Omega)}\geq c\norm{h}^2_{H^1(\Omega)},\quad t\in(0,T).\ee
Indeed, for $h\in V$, $t\in(0,T)$ and $\epsilon_1\in(0,1)$, we have
$$\begin{aligned}\re a(t,h,h)&\geq\int_{\Omega}|\nabla_x h|^2dx-(\norm{A}_{L^\infty(Q)^n}+2\norm{B}_{L^\infty(Q)^n})\int_{\Omega}|\nabla_x h||h|dx-\int_\Omega|q(t,\cdot)||h|^2dx\\
\ &\geq (1-\epsilon_1)\int_{\Omega}|\nabla_x h|^2dx-\left(\frac{\left(\norm{A}_{L^\infty(Q)^n}+2\norm{B}_{L^\infty(Q)^n}\right)^2}{\epsilon_1}\right)\int_\Omega|h|^2dx-\int_\Omega|q||h|^2dx.\end{aligned}$$
In addition applying the H\"older inequality, the Sobolev embedding theorem and an interpolation between Sobolev spaces, for all $t\in(0,T)$, we get
$$\begin{aligned} \int_\Omega|q(t,\cdot)||h|^2dx&\leq \norm{q}_{L^\infty(0,T;L^{\frac{2n}{3}}(\Omega))}\norm{h}_{L^{\frac{2n}{n-\frac{3}{2}}}(\Omega)}^2\\
\ &\leq C\norm{q}_{L^\infty(0,T;L^{\frac{2n}{3}}(\Omega))}\norm{h}_{H^{\frac{3}{4}}(\Omega)}^2\\
\ &\leq C\norm{q}_{L^\infty(0,T;L^{\frac{2n}{3}}(\Omega))}\left(\norm{h}_{H^1(\Omega)}^2\right)^{\frac{3}{4}}\left(\norm{h}_{L^2(\Omega)}^2\right)^{\frac{1}{4}}\\
\ &\leq C\norm{q}_{L^\infty(0,T;L^{\frac{2n}{3}}(\Omega))}\left(\epsilon_1\norm{h}_{H^1(\Omega)}^2\right)^{\frac{3}{4}}\left(\epsilon_1^{-3}\norm{h}_{L^2(\Omega)}^2\right)^{\frac{1}{4}}\\
\ &\leq C\norm{q}_{L^\infty(0,T;L^{\frac{2n}{3}}(\Omega))}\left(\frac{3\epsilon_1}{4}\norm{h}_{H^1(\Omega)}^2+\frac{\epsilon_1^{-3}}{4}\norm{h}_{L^2(\Omega)}^2\right),\end{aligned}$$
with $C>0$ depending only on $\Omega$. Therefore, choosing
$$\epsilon_1=\left(C\norm{q}_{L^\infty(0,T;L^{\frac{2n}{3}}(\Omega))}+2\right)^{-1},\quad c=1-\epsilon_1(C\norm{q}_{L^\infty(0,T;L^{\frac{2n}{3}}(\Omega))}+1),$$
$$\lambda= \frac{\left(\norm{A}_{L^\infty(Q)^n}+2\norm{B}_{L^\infty(Q)^n}\right)^2}{\epsilon_1}+C\norm{q}_{L^\infty(0,T;L^{\frac{2n}{3}}(\Omega))}\epsilon_1^{-3}+2,$$
we get \eqref{cor}. Combining \eqref{con}-\eqref{cor} with the fact that 
$$a(t,h,g)=\left\langle -\Delta_x h+A(\cdot,t)\cdot\nabla_xh+(\nabla_x\cdot B(\cdot,t))h+q(\cdot,t)h,g\right\rangle_{H^{-1}(\Omega),H^1_0(\Omega)},\quad t\in(0,T),$$
we deduce from \cite[Theorem 4.1, Chapter 3]{LM1} that  problem \eqref{eq5} admits a unique solution $w\in H^1(0,T;H^{-1}(\Omega))\cap L^2(0,T;H^1(\Omega))$ satisfying
$$\begin{aligned}\norm{w}_{L^2(0,T;H^1(\Omega))}+\norm{w}_{H^1(0,T;H^{-1}(\Omega))}&\leq C\norm{-A\cdot\nabla_x G-(\nabla_x\cdot B) G-qG}_{L^2(0,T;H^{-1}(\Omega))}\\
\ &\leq C\norm{G}_{L^2(0,T;H^1(\Omega))}\leq C\norm{g_+}_{\mathcal H_+},\end{aligned}$$
where $C$ depends on $\Omega$, $T$ and $M$. Therefore, $u=w+G$ is the unique solution of \eqref{eeeq1} and the above  estimate  implies \eqref{p2a}. \end{proof}

Using these properties, we would like to give a suitable definition of the normal derivative of solutions of  \eqref{eq1}. For this purpose, following \cite{KU1} we will give a variational sense to the normal derivative for solutions of these problems. For this purpose, we start by considering the spaces
$$
H^{1/2}(\sqcup) := \{ \tilde{g}|_{\Sigma} : \tilde{g} \in H^{1/2}(\partial Q),\, \textrm{supp} (\tilde{g}) \subset \partial Q \setminus \Omega^T \}.
$$
We use the symbols $\sqcup$ because it turns out to be convenient to keep 
in mind that the corresponding functions vanish on  $\Omega^T:=\Omega\times\{T\}$. Note that the norms
$$\| g \|_{H^{1/2}(\sqcup)} := \inf \{ \| \tilde{g} \|_{H^{1/2}(\partial Q)} : 
\tilde{g}|_\Sigma = g, \, \textrm{supp} (\tilde{g}) \subset \partial Q \setminus \Omega^T  
\}$$
makes $H^{1/2}(\sqcup)$ be a Banach space. We recall that there exists a lifting operator $L:H^{1/2}(\sqcup)\longrightarrow \{w\in H^1(Q):\ w_{|\Omega^T}=0\}$ such that $L$ is a bounded operator and $$Lg_{|\Sigma}=g,\quad g\in H^{1/2}(\sqcup).$$

For any $g_+\in \mathcal H_+$ and $u\in H^1(0,T;H^{-1}(\Omega))\cap L^2(0,T;H^1(\Omega))$ the solution of \eqref{eeeq1}, we define $N_{A,B,q}u\in H^{1/2}(\sqcup)^*$, where $H^{1/2}(\sqcup)^*$ denotes the dual space of $H^{1/2}(\sqcup)$, by 
\bel{Nor}\begin{aligned}&\left\langle N_{A,B,q}u,g_-\right\rangle_{H^{1/2}(\sqcup)^*,H^{1/2}(\sqcup)}\\
&=\int_Q[-u\pd_tLg_-+\nabla_xu\cdot\nabla_xLg_-+A\cdot\nabla_xuLg_--B\cdot\nabla_x (uLg_-)+quLg_-]dxdt.\end{aligned}\ee
Note that, for $w\in H^1(Q)$ satisfying $w_{|\Omega^T}=0$ and $w_{|\Sigma}=0$, since $u\in H^1(0,T;H^{-1}(\Omega))\cap L^2(0,T;H^1(\Omega))$ solves \eqref{eeeq1}, we have
$$\begin{aligned}&\int_Q[-u\pd_tw+\nabla_xu\cdot\nabla_xw+A\cdot\nabla_xuw-B\cdot\nabla_x (uw)+quw]dxdt\\
\ &= \left\langle \pd_tu-\Delta u+A\cdot\nabla_xu+(\nabla_x\cdot B)u+qu,w\right\rangle_{L^2(0,T;H^{-1}(\Omega)),L^2(0,T;H^1_0(\Omega))}=0.\end{aligned}$$
Therefore, \eqref{Nor} is well defined since the right hand side of this identity depends only on $g_-$. We define the DN map associated with \eqref{eeeq1} by
$$\Lambda_{A,B,q}:\mathcal H_+\ni g_+\mapsto N_{A,B,q}u\in H^{1/2}(\sqcup)^*$$
and, applying Proposition \ref{p2} one can check that this map is continuous from $\mathcal H_+$ to $H^{1/2}(\sqcup)^*$.
By density, we derive the following representation formula

\begin{prop}\label{p3} For $j=1,2$, let $A_j,\ B_j\in L^\infty(Q)^n$, $q_j\in L^\infty(0,T;L^{\frac{2n}{3}}(\Omega))$.  Then, the operator $\Lambda_{A_1,B_1,q_1}-\Lambda_{A_2,B_2,q_2}$ can be extended to a bounded operator from $\mathcal H_+$ to   $\mathcal H_-^*$, where $\mathcal H_-^*$ denotes the dual space of $\mathcal H_-$. Moreover, for $g_+\in \mathcal H_+$, $g_-\in \mathcal H_-$, we consider $u_1\in H^1(0,T;H^{-1}(\Omega))\cap L^2(0,T;H^1(\Omega))$ the solution of \eqref{eeeq1} with $A=A_1$, $B=B_1$, $q=q_1$ and $u_2\in H^1(0,T;H^{-1}(\Omega))\cap L^2(0,T;H^1(\Omega))$ be the solution of \eqref{eq2} with $A=A_2$, $B=B_2$, $q=q_2$. Then we have
\begin{equation}\label{p3a}\begin{aligned} &\left\langle(\Lambda_{A_1,B_1,q_1}-\Lambda_{A_2,B_2,q_2})g_+, g_-  \right\rangle_{\mathcal H_-^*,\mathcal H_-}\\
&=\int_Q(A_1-A_2)\cdot\nabla_xu_1u_2dxdt-\int_Q(B_1-B_2)\cdot\nabla_x(u_1u_2)dxdt+\int_Q(q_1-q_2)u_1u_2dxdt.\end{aligned}\end{equation}
\end{prop}
\begin{proof} Without loss of generality we assume that all the functions are real valued.
We consider $v_2\in H^1(0,T;H^{-1}(\Omega))\cap L^2(0,T;H^1(\Omega))$ solving  
$$\left\{\begin{array}{ll}\partial_tv_2-\Delta_x v_2+A_2(x,t)\cdot\nabla_xv_2+[\nabla_x\cdot B_2(x,t)] v_2+q_2(x,t)v_2=0,\quad &\textrm{in}\ Q,\\  v_2(0,\cdot)=0,\quad &\textrm{in}\ \Omega,\\ v_2=g_+,\quad &\textrm{on}\ \Sigma.\end{array}\right.$$
 Then, for any $g_-\in H^{1/2}(\sqcup)$ fixing $w=Lg_-\in H^1(Q)$, we find
$$\begin{aligned} &\left\langle(\Lambda_{A_1,B_1,q_1}-\Lambda_{A_2,B_2,q_2})g_+, g_-  \right\rangle_{H^{1/2}(\sqcup)^*,H^{1/2}(\sqcup)}\\
&=\left\langle(N_{A_1,B_1,q_1}u_1-N_{A_2,B_2,q_2}v_2, g_-  \right\rangle_{H^{1/2}(\sqcup)^*,H^{1/2}(\sqcup)}\\
&=\int_Q[-(u_1-v_2)\pd_tw+\nabla_x(u_1-v_2)\cdot\nabla_xw+A_2\cdot\nabla_x(u_1-v_2)w-B_2\cdot\nabla_x ((u_1-v_2)w)+q_2(u_1-v_2)w]dxdt\\
&\ \ \ +\int_Q[(A_1-A_2)\cdot\nabla_xu_1w-(B_1-B_2)\cdot\nabla_x (u_1w)+(q_1-q_2)u_1w]dxdt\end{aligned}$$
Now using the fact that $(u_1-v_2)\in L^2(0,T; H^1_0(\Omega))$, we get
\begin{equation}
\begin{aligned}&\left\langle(\Lambda_{A_1,B_1,q_1}-\Lambda_{A_2,B_2,q_2})g_+, g_-  \right\rangle_{H^{1/2}(\sqcup)^*,H^{1/2}(\sqcup)}\\
&=-\left\langle \pd_tw,u_1-v_2 \right\rangle_{L^2(0,T;H^{-1}(\Omega)),L^2(0,T;H_0^1(\Omega))}+\int_Q\nabla_x (u_1-v_2)\cdot\nabla_xwdxdt\\
\ &\ \ \ +\int_QA_2\cdot\nabla_x (u_1-v_2)wdxdt-\int_QB_2\cdot\nabla_x [(u_1-v_2)w]dxdt+\int_Qq_2(u_1-v_2)wdxdt\\
&\ \ \ +\int_Q[(A_1-A_2)\cdot\nabla_xu_1w-(B_1-B_2)\cdot\nabla_x (u_1w)+(q_1-q_2)u_1w]dxdt\end{aligned}
\label{id:identity_extension}
\end{equation}
We can choose $\tilde w$ to be the unique element of $S_-$ satisfying $\tilde w_{|\Sigma}=g_-$, and since $w - \tilde{w} \in L^2(0,T;H^1_0(\Omega))$ and $w - \tilde{w}|_{\Omega^T} = 0$, $w$ can be replaced by $\tilde{w}$ in the identity \eqref{id:identity_extension}. Moreover, we have
$$\begin{aligned}\abs{\left\langle(\Lambda_{A_1,B_1,q_1}-\Lambda_{A_2,B_2,q_2})g_+, g_-  \right\rangle_{H^{1/2}(\sqcup)^*,H^{1/2}(\sqcup)}}&\leq C\norm{g_+}_{\mathcal H_+}(\norm{\tilde w}_{L^2(0,T;H^1(\Omega))}+\norm{\tilde w}_{H^1(0,T;H^{-1}(\Omega))})\\
\ &\leq C\norm{g_+}_{\mathcal H_+}\norm{g_-}_{\mathcal H_-},\end{aligned}$$
where $C$ depends only on $A_j$, $B_j$, $q_j$, $j=1,2$, $T$ and $\Omega$. From this identity, we deduce that the map $\Lambda_{A_1,B_1,q_1}-\Lambda_{A_2,B_2,q_2}$ can be extended continuously to a continuous linear map from $\mathcal H_+$ to   $\mathcal H_-^*$ and the identity \eqref{id:identity_extension} holds for $g_- \in \mathcal{H}_-$, whose extension $w$ to $Q$ belongs to $H^1(0,T; H^{-1}(\Omega)) \cap L^2(0,T; H^1(\Omega))$. Thus, we are allowed to replace $w$ in \eqref{id:identity_extension} by $u_2$. Since $u_2$ satisfies the identity below, the proposition is proved:
$$\begin{aligned}&-\ \left\langle \pd_tu_2,u_1-v_2 \right\rangle_{L^2(0,T;H^{-1}(\Omega)),L^2(0,T;H_0^1(\Omega))}+\int_Q\nabla_x (u_1-v_2)\cdot\nabla_xu_2dxdt\\
&=\left\langle -\pd_tu_2-\Delta u_2,u_1-v_2 \right\rangle_{L^2(0,T;H^{-1}(\Omega)),L^2(0,T;H_0^1(\Omega))}\\
&=\left\langle A_2\cdot\nabla_x u_2,u_1-v_2 \right\rangle+\left\langle \nabla_x\cdot(A_2-B_2),(u_1-v_2)u_2 \right\rangle_{L^2(0,T;H^{-1}(\Omega)),L^2(0,T;H_0^1(\Omega))}-\int_Qq_2(u_1-v_2)u_2dxdt\\
&=\left\langle \nabla_x\cdot(u_2A_2 )- \nabla_x\cdot(B_2)u_2,u_1-v_2 \right\rangle_{L^2(0,T;(H^{-1}(\Omega)),L^2(0,T;H^1_0(\Omega))}-\int_Qq_2(u_1-v_2)u_2dxdt\\
&=-\int_Q(A_2\cdot\nabla_x (u_1-v_2)u_2dxdt+\int_QB_2\cdot\nabla_x [(u_1-v_2)u_2]dxdt-\int_Qq_2(u_1-v_2)u_2dxdt.\end{aligned}$$

\end{proof}

\section{Carleman estimates}

We introduce two parameters $s,\rho\in(1,+\infty)$ and we consider, for $\rho>s>1$, the perturbed weight
\bel{phi}\phi_{\pm,s}(x,t):=\pm (\rho^2t+\rho\omega\cdot x)-s{((x+ x_0)\cdot\omega)^2\over 2}.\ee
We define
\[L_{\pm, A}=\pm\partial_t-\Delta_x\pm A\cdot\nabla_x,\quad P_{A,\pm,s}:=e^{-\phi_{\pm,s}}L_{\pm,A}e^{\phi_{\pm,s}}.\]
Here $x_0\in\R^n$ is chosen in such a way that 
\bel{x0}x_0\cdot\omega=2+\sup_{x\in\Omega}|x|.\ee
The goal of this section is to prove the following Carleman estimates.
\begin{prop}\label{pp1} Let $A\in L^\infty(Q)^n$ and $\Omega$ be $\mathcal C^2$. Then  there exist $s_1>1$ and, for $s>s_1$,  $\rho_1(s)$ such that for any $v\in\mathcal C^2(\overline{Q})$ satisfying the condition  
\begin{equation}\label{t2a}v_{\vert \Sigma}=0,\quad v_{\vert \Omega^0}=0,\end{equation}
the estimate
\bel{lll1a} \begin{aligned}&\rho\int_{\Sigma_{+,\omega}} |\partial_\nu v|^2|\omega\cdot\nu| d\sigma(x)dt+s\rho \int_\Omega|v|^2(x,T)dx+s^{-1}\int_Q|\Delta_x v|^2dxdt+s\rho^2\int_Q|v|^2dxdt\\
&\leq C\left[\norm{P_{A,+,s}v}^2_{L^2(Q)}+\rho\int_{\Sigma_{-,\omega}} |\partial_\nu v|^2|\omega\cdot\nu| d\sigma(x)dt\right]\end{aligned}\ee
holds true for $s>s_1$, $\rho\geq \rho_1(s)$  with $C$  depending only on  $\Omega$, $T$ and $M\geq \norm{A}_{L^\infty(Q)^n}$.
In the same way, there exist  $s_2>1$ and, for $s>s_2$,  $\rho_2(s)$ such that for any $v\in\mathcal C^2(\overline{Q})$ satisfying the condition \begin{equation}\label{t2c}v_{\vert \Sigma}=0,\quad v_{\vert \Omega^T}=0,\end{equation}
the estimate
\bel{lll1b} \begin{aligned}&\rho\int_{\Sigma_{-,\omega}} |\partial_\nu v|^2|\omega\cdot\nu| d\sigma(x)dt+s\rho \int_\Omega|v|^2(x,0)dx+s^{-1}\int_Q|\Delta_x v|^2dxdt+s\rho^2\int_Q|v|^2dxdt\\
&\leq C\left[\norm{P_{A,-,s}v}^2_{L^2(Q)}+\rho\int_{\Sigma_{+,\omega}} |\partial_\nu v|^2|\omega\cdot\nu| d\sigma(x)dt\right]\end{aligned}\ee
holds true for $s>s_2$, $\rho\geq \rho_2(s)$. Here $s_1$, $\rho_1$, $s_2$ and $\rho_2$ depend only on  $\Omega$, $T$ and $M\geq \norm{A}_{L^\infty(Q)^n}$. 
\end{prop}
\begin{proof}Without loss of generality we assume that $v$ is real valued. We start with \eqref{lll1a}.
For this purpose we will first show that, for $A=0$, there exists $c$ depending only on $\Omega$, $s_1$ depending on $\Omega$, $T$ such that for any $s>s_1$ we can find $\rho_1(s)$ for which the estimate
\bel{5}\begin{aligned}\norm{P_{A,+,s}v}^2_{L^2(Q)}\geq&\rho\int_{\Sigma_{+,\omega}} |\partial_\nu v|^2|\omega\cdot\nu| d\sigma(x)dt-8\rho\int_{\Sigma_{-,\omega}} |\partial_\nu v|^2|\omega\cdot\nu| d\sigma(x)dt+cs^{-1}\int_Q|\Delta_x v|^2dxdt\\
\ &+s\rho \int_\Omega|v|^2(x,T)dx+s\rho^2\int_Q|v|^2dxdt+2s\int_Q|\nabla_xv|^2dxdt\end{aligned}\ee
holds true when the condition $\rho>\rho_1(s)$ is fulfilled.
Using this estimate,  we will derive \eqref{lll1a}. We decompose $P_{A,+,s}$ into three terms
\[P_{A,+,s}=P_{1,+}+P_{2,+}+P_{3,+},\]
with
\[P_{1,+}=-\Delta_x+\pd_t \phi_{+,s}-|\nabla_x\phi_{+,s}|^2+\Delta_x\phi_{+,s},\quad P_{2,+}=\pd_t-2\nabla_x\phi_{+,s}\cdot\nabla_x-2\Delta_x\phi_{+,s},\quad P_{3,+}=A\cdot\nabla_x+A\cdot\nabla_x\phi_{+,s}+q.\]
Note that
\[\pd_t\phi_{+,s}=\rho^2,\quad \nabla_x\phi_{+,s}=[\rho-s(x+x_0)\cdot\omega]\omega,\quad -\Delta\phi_{+,s}=s\]
and 
\bel{caca1}P_{1,+}v=-\Delta_xv+[2\rho s(x+x_0)\cdot\omega-s^2((x+x_0)\cdot\omega)^2-s]v,\ee
$$P_{2,+}v=\partial_tv-2[\rho-s(x+x_0)\cdot\omega](\omega\cdot\nabla_xv)+2sv,$$
\bel{1}\begin{aligned}P_{1,+}vP_{2,+}v=&-\Delta_x v\pd_t v+2\Delta_x v[\rho-s(x+x_0)\cdot\omega](\omega\cdot\nabla_xv)-2s(\Delta v)v\\
\ &+[2\rho s(x+x_0)\cdot\omega-s^2((x+x_0)\cdot\omega)^2-s]v[\pd_tv-2[\rho-s(x+x_0)\cdot\omega](\omega\cdot\nabla_xv)+2sv]\end{aligned}\ee
For the first term on the right hand side of \eqref{1} we find
$$\int_Q(-\Delta_x v\pd_t v)dxdt=\int_Q\pd_t\nabla_x v\cdot\nabla_x vdxdt={1\over2}\int_Q\pd_t|\nabla_x v|^2dxdt={1\over2}\int_\Omega|\nabla_x v(x,T)|^2dx.$$
It follows that
\bel{2}\int_Q(-\Delta_x v\pd_t v)dxdt\geq0.\ee
We have also
\[\begin{aligned}&2\int_Q\Delta_x v[\rho-s(x+x_0)\cdot\omega](\omega\cdot\nabla_xv)dxdt\\
&=2\int_\Sigma \partial_\nu v[\rho-s(x+x_0)\cdot\omega](\omega\cdot\nabla_xv)d\sigma(x)dt+2s\int_Q(\omega\cdot\nabla_xv)^2dxdt-2\int_Q[\rho-s(x+x_0)\cdot\omega][\nabla_x v\cdot\nabla_x(\omega\cdot\nabla_x v)]dxdt\\
&=2\int_\Sigma \partial_\nu v[\rho-s(x+x_0)\cdot\omega](\omega\cdot\nabla_xv)d\sigma(x)dt+2s\int_Q(\omega\cdot\nabla_xv)^2dxdt-\int_Q[\rho-s(x+x_0)\cdot\omega]\omega\cdot\nabla_x|\nabla_x v|^2dxdt\\
\end{aligned}\]
and using the fact that $v_{|\Sigma}=0$, we get
$$\begin{aligned}&2\int_Q\Delta_x v[\rho-s(x+x_0)\cdot\omega](\omega\cdot\nabla_xv)dxdt\\
&=2\int_\Sigma [\rho-s(x+x_0)\cdot\omega]|\partial_\nu v|^2\omega\cdot\nu d\sigma(x)dt+2s\int_Q(\omega\cdot\nabla_xv)^2dxdt\\
&\ \ \ \ -\int_Q[\rho-s(x+x_0)\cdot\omega]|\partial_\nu v|^2\omega\cdot\nu d\sigma(x)dt-s\int_Q|\nabla_x v|^2dxdt\\
&=\int_\Sigma [\rho-s(x+x_0)\cdot\omega]|\partial_\nu v|^2\omega\cdot\nu d\sigma(x)dt-s\int_Q|\nabla_x v|^2dxdt+2s\int_Q(\omega\cdot\nabla_xv)^2dxdt.\end{aligned}$$
Choosing $\rho\geq 2s(1+\underset{x\in\Omega}{\sup}|x|)$, we obtain
$$\begin{aligned}2\int_Q\Delta_x v[\rho-s(x+x_0)\cdot\omega](\omega\cdot\nabla_xv)dxdt\\
\geq \rho\int_{\Sigma_{+,\omega}} |\partial_\nu v|^2|\omega\cdot\nu| d\sigma(x)dt-4\rho\int_{\Sigma_{-,\omega}} |\partial_\nu v|^2|\omega\cdot\nu| d\sigma(x)dt-s\int_Q|\nabla_x v|^2dxdt.\end{aligned}$$
Combining this with the fact that
$$-2s\int_Q(\Delta v)vdxdt=2s\int_Q|\nabla_x v|^2dxdt,$$
we find
\bel{3}\begin{aligned}2\int_Q\Delta_x v[\rho-s(x+x_0)\cdot\omega](\omega\cdot\nabla_xv)dxdt-2s\int_Q(\Delta v)vdxdt\\
\geq \rho\int_{\Sigma_{+,\omega}} |\partial_\nu v|^2|\omega\cdot\nu| d\sigma(x)dt-4\rho\int_{\Sigma_{-,\omega}} |\partial_\nu v|^2|\omega\cdot\nu| d\sigma(x)dt+s\int_Q|\nabla_x v|^2dxdt.\end{aligned}\ee
Now let us consider the last term on the right hand side of \eqref{1}. Note first that
$$\begin{aligned}&\int_Q[2\rho s(x+x_0)\cdot\omega-s^2((x+x_0)\cdot\omega)^2-s]v\pd_tvdxdt\\
&={1\over2}\int_Q[2\rho s(x+x_0)\cdot\omega-s^2((x+x_0)\cdot\omega)^2-s]\pd_t|v|^2dxdt\\
&\geq \rho s\int_\Omega(x+x_0)\cdot\omega|v|^2(x,T)dx-s^2\left(3+\underset{x\in\Omega}{\sup}|x|\right)^2\int_\Omega |v|^2(x,T)dx.\end{aligned}$$
Combining this with \eqref{x0} and choosing $\rho\geq s\left(3+\underset{x\in\Omega}{\sup}|x|\right)^2$, we find
\bel{4}\int_Q[2\rho s(x+x_0)\cdot\omega+s^2((x+x_0)\cdot\omega)^2-s]v\pd_tvdxdt\geq s\rho \int_\Omega|v|^2(x,T)dx.\ee
In addition, integrating by parts with respect to $x\in\Omega$, we get
$$\begin{array}{l}\int_Q[2\rho s(x+x_0)\cdot\omega+s^2((x+x_0)\cdot\omega)^2-s]v[-2[\rho-s(x+x_0)\cdot\omega](\omega\cdot\nabla_xv)]dxdt\\
\ \\
=-\int_Q[- s^3((x+x_0)\cdot\omega)^3-\rho s^2((x+x_0)\cdot\omega)^2+(2\rho^2 s+s^2)(x+x_0)\cdot\omega-s\rho]\omega\cdot\nabla_x|v|^2dxdt\\
\ \\
=\int_Q[- 3s^3((x+x_0)\cdot\omega)^2-2\rho s^2((x+x_0)\cdot\omega)+(2\rho^2 s+s^2)]|v|^2dxdt.\end{array}$$
It follows that
$$\begin{aligned}\int_Q[2\rho s(x+x_0)\cdot\omega-s^2((x+x_0)\cdot\omega)^2-s]v[-2[\rho-s(x+x_0)\cdot\omega](\omega\cdot\nabla_xv)+2sv]dxdt\\
=\int_Q[- 5s^3((x+x_0)\cdot\omega)^2+2\rho s^2((x+x_0)\cdot\omega)+(2\rho^2 s-s^2)]|v|^2dxdt.\end{aligned}$$
Then, fixing 
$$\rho\geq \sqrt{5s^2(2+\underset{x\in\Omega}{\sup}|x|)^2+s},$$
 we obtain
$$\int_Q[2\rho s(x+x_0)\cdot\omega+s^2((x+x_0)\cdot\omega)^2+s]v[-2[\rho-s(x+x_0)\cdot\omega](\omega\cdot\nabla_xv)+2sv]dxdt\geq s\rho^2\int_Q|v|^2dxdt.$$
Combining this estimate with \eqref{2}-\eqref{4}, we find
\bel{6}\begin{aligned}\norm{P_{1,+}v+P_{2,+}v}^2_{L^2(Q)}&\geq2\int_QP_{1,+}vP_{2,+}vdxdt+\norm{P_{1,+}v}_{L^2(Q)}^2\\
\ &\geq2\rho\int_{\Sigma_{+,\omega}} |\partial_\nu v|^2|\omega\cdot\nu| d\sigma(x)dt-8\rho\int_{\Sigma_{-,\omega}} |\partial_\nu v|^2|\omega\cdot\nu| d\sigma(x)dt+2s\int_Q|\nabla_x v|^2dxdt\\
\ &\ \ \ +2s\rho \int_\Omega|v|^2(x,T)dx+2s\rho^2\int_Q|v|^2dxdt+\norm{P_{1,+}v}_{L^2(Q)}^2.\end{aligned}\ee
Moreover, we have
$$\begin{aligned}\norm{P_{1,+}v}_{L^2(Q)}^2&= \norm{-\Delta_xv+[2\rho s(x+x_0)\cdot\omega-s^2((x+x_0)\cdot\omega)^2+s]v}_{L^2(Q)}^2\\
\ &\geq \frac{\norm{\Delta_xv}_{L^2(Q)}^2}{2}-\norm{[2\rho s(x+x_0)\cdot\omega-s^2((x+x_0)\cdot\omega)^2+s]v}_{L^2(Q)}^2\\
\ &\geq \frac{\norm{\Delta_xv}_{L^2(Q)}^2}{2}-36s^2\rho^2\left(2+\underset{x\in\Omega}{\sup}|x|\right)^4 \norm{v}_{L^2(Q)}^2. \end{aligned}$$
Fixing $c=\left(36\left(2+\underset{x\in\Omega}{\sup}|x|\right)^4\right)^{-1}$, we deduce that
$$\norm{P_{1,+}v}_{L^2(Q)}^2\geq cs^{-1}\norm{P_{1,+}v}_{L^2(Q)}^2\geq \frac{cs^{-1}}{2}\norm{\Delta_xv}_{L^2(Q)}^2-s\rho^2\norm{v}_{L^2(Q)}^2$$
and, combining this with \eqref{6}, we obtain \eqref{5} by fixing $$\rho_1(s)>s\left(3+\underset{x\in\Omega}{\sup}|x|\right)^2+ \sqrt{5s^2(2+\underset{x\in\Omega}{\sup}|x|)^2+s}.$$ Using \eqref{5}, we will complete the proof of the lemma.
For this purpose, we remark first that, for $\rho>\rho_1(s)$, we have
\[\begin{aligned}\norm{P_{A,+,s}v}^2_{L^2(Q)}&\geq {\norm{P_{1,+}v+P_{2,+}v}^2_{L^2(Q)}\over 2}-\norm{P_3v}_{L^2(Q)}^2\\
\ &\geq {\norm{P_{1,+}v+P_{2,+}v}^2_{L^2(Q)}\over 2}-2\norm{A}_{L^\infty(Q)}^2\int_Q|\nabla_x v|^2dxdt\\
\ &\ \ \ -8\rho^2\norm{A}_{L^\infty(Q)}^2 \int_Q| v|^2dxdt.\end{aligned}\]
Combining these estimates with \eqref{5}, we deduce that  for $s_1=32M^2$ and, for $s>s_1$,  $\rho>\rho_1(s)$, estimate \eqref{lll1a} holds true.

Now let us consider \eqref{lll1b}. We start by assuming that $A=0$. For this purpose we fix $v\in\mathcal C^2(\overline{Q})$ satisfying \eqref{t2c} and we consider $w\in\mathcal C^2(\overline{Q})$ defined by $w(x,t):= v(x,T-t)$. Clearly $w(x,0)=0$. Moreover, fixing
$$\phi_{+,s}^*(x,t):= (\rho^2t-\rho\omega\cdot x)-s{((x+ x_0)\cdot\omega)^2\over 2},$$
which corresponds to $\phi_{+,s}$ with $\omega$ replaced by $-\omega$, one can check that
$$e^{-\phi_{+,s}^*}(\partial_t-\Delta)e^{\phi_{+,s}^*}w(x,t)=P_{0,-,s}v(x,T-t),\quad (x,t)\in\overline{Q},$$
with $P_{0,-,s}=P_{A,-,s}$ for $A=0$. Therefore, applying \eqref{lll1a}, with $\omega$ replaced by $-\omega$, to $w$ we deduce \eqref{lll1b}. We can extend this result to the case $A\neq0$  by repeating the arguments used at the end of the proof of \eqref{lll1a}.
\end{proof}

\section{ GO solutions}
 Armed with the estimates \eqref{lll1a}-\eqref{lll1b} we will build suitable GO solutions for our problem. More precisely, for $j=1,2$, fixing the cofficient $(A_j,B_j,q_j)\in L^\infty(Q)^n\times L^\infty(Q)^n\times[L^\infty(0,T;L^{p}(\Omega))\cap \mathcal C([0,T];L^{\frac{2n}{3}}(\Omega))]$ with $p > 2n/3$ and  $\omega\in\mathbb S^{n-1}$, 
we look for $u_j$ solutions of 
\bel{Gsol1}\left\{\begin{array}{l}\pd_tu_1-\Delta_xu_1+A_1\cdot\nabla_xu_1+(\nabla_x\cdot B_1)u_1+q_1u_1=0,\ (x,t)\in Q,\\ u_1(x,0)=0,\quad x\in\Omega,\end{array}\right.\ee
\bel{Gsol2}\left\{\begin{array}{l}-\pd_tu_2-\Delta_xu_2-A_2\cdot\nabla_xu_2+(q_2+\nabla_x\cdot (B_2-A_2))u_2=0,\ (x,t)\in Q,\\ u_2(x,T)=0,\quad x\in\Omega,\end{array}\right.\ee
taking the form
\bel{GO12} u_1(x,t)=e^{\rho^2t+\rho x\cdot\omega}(b_{1,\rho}(x,t)+w_{1,\rho}(x,t)),\quad u_2(x,t)=e^{-\rho^2t-\rho x\cdot\omega}(b_{2,\rho}(x,t)+w_{2,\rho}(x,t)),\quad (x,t)\in Q.\ee
In these expressions, the term $b_{j,\rho}$, $j=1,2$, are the principal part of our GO solutions and they will be suitably designed for the recovery of the coefficients. The expression $w_{j,\rho}$, $j=1,2$, are the remainder term in this expression that admits a decay with respect to the parameter $\rho$ of the form
\bel{CGO11}\lim_{\rho\to+\infty}(\rho^{-1}\norm{w_{j,\rho}}_{L^2(0,T;H^1(\Omega))}+ \norm{w_{j,\rho}}_{L^2(Q)})=0.\ee
 We start by considering the principal parts of our GO solutions.

\subsection{Principal part of the GO solutions}

In this subsection we will introduce the form of the principal part $b_{j,\rho}$, $j=1,2$, of our GO solutions given by \eqref{GO12}. For this purpose, we consider $A_j\in L^\infty(Q)^n$, $j=1,2$ and we will consider $b_{j,\rho}$, $j=1,2$, to be an approximation of a solution $b_j$ of the transport equation
\bel{trans}-2\omega\cdot \nabla_xb_1+(A_1(x,t)\cdot \omega)b_1=0,\quad 2\omega\cdot \nabla_xb_2+(A_2(x,t)\cdot \omega)b_2=0,\quad (x,t)\in Q.\ee
 By replacing  the functions $b_1$, $b_2$, whose regularity depends on the one  of the coefficients $A_1$ and $A_2$, with their approximation $b_{1,\rho}$, $b_{2,\rho}$,  we can reduce the regularity  of the  coefficients $A_j$, $j=1,2$, from $L^\infty(0,T;W^{2,\infty}(\Omega))^n\cap W^{1,\infty}(0,T;L^{\infty}(\Omega))^n$ to $ L^\infty(Q)^n$. This approach, also considered in \cite{BKS,Ki3,KU1,Sa1}, remove also condition imposed to the  coefficients $A_j$, $j=1,2$,  on $\Sigma$. Indeed, if in our construction  we use the expression $b_j$ instead of  $b_{j,\rho}$, $j=1,2$,  then we can prove Theorem \ref{t5}  only for  coefficients $A_1,A_2\in L^\infty(0,T;W^{2,\infty}(\Omega))^n\cap H^1(0,T;L^{\infty}(\Omega))^n$  satisfying
$$\partial_x^\alpha A_1(x,t)=\partial_x^\alpha A_2(x,t),\quad (x,t)\in\Sigma,\ \alpha\in\mathbb N^n,\ |\alpha|\leq1,$$
where in our case we make no assumption on $A_j$ at $\Sigma$ for \eqref{t5b}, and we only assume \eqref{tttt5b} for \eqref{t5c}.

We start by considering a suitable approximation of the coefficients $A_j$, $j=1,2$. For all $r>0$ we set $B_r:=\{(x,t)\in\R^{1+n}:\ |(x,t)|<r\}$ and we fix $\chi\in\mathcal C^\infty_0(\R^{1+n})$ such that $\chi\geq0$, $\int_{\R^{1+n}}\chi(x,t)dtdx=1$, supp$(\chi)\subset B_1$. We introduce also $\chi_\rho$ given by
$\chi_\rho(x,t)=\rho^{{n+1\over 3}}\chi(\rho^{{1\over3}}x,\rho^{{1\over3}}t)$
and, for $j=1,2$, we fix
$$A_{j,\rho}(x,t):=\int_{\R^{1+n}}\chi_\rho(x-y,t-s)A_j(y,s)dsdy.$$
Here, we assume that $A_j=0$ on $\R^{1+n}\setminus Q$.
For $j=1,2$, since $A_j\in L^\infty(\R^{1+n})^n$ is supported in the compact set $\overline{Q}$, we have
\bel{a1a}\lim_{\rho\to+\infty}\norm{A_{j,\rho}-A_j}_{L^1(\R^{1+n})}=\lim_{\rho\to+\infty}\norm{A_{j,\rho}-A_j}_{L^2(\R^{1+n})}=0,\ee
and one can easily check the estimates
\bel{a1b}\norm{A_{j,\rho}}_{W^{k,\infty}(\R^{1+n})}\leq C_k\rho ^{{k\over 3}},\ee
with $C_k$ independent of $\rho$. Note that $$A_{\rho}(x,t):=\int_{\R^{1+n}}\chi_\rho(x-y,t-s)A(y,s)dsdy=A_{1,\rho}(x,t)-A_{2,\rho}(x,t),$$
with $A=A_1-A_2$.
Then, for $\xi\in\omega^\bot:=\{x\in\R^{n}:\ \omega\cdot x=0\}$, we fix 
\bel{conde1} b_{1,\rho}(x,t)=e^{-i(t\tau+x\cdot\xi)}\left(1-e^{-\rho^{1\over3}t}\right)\exp\left(-{\int_0^{+\infty} A_{1,\rho}(x+s\omega,t)\cdot\omega ds\over2}\right),\ee
\bel{conde2} b_{2,\rho}(x,t)=\left(1-e^{-\rho^{1\over3}(T-t)}\right)\exp\left({\int_0^{+\infty} A_{2,\rho}(x+s\omega,t)\cdot\omega ds\over2}\right).\ee
According to \eqref{a1a}-\eqref{a1b} and to the fact that, for $j=1,2$, supp$(A_{j,\rho})\subset [-1,T+1]\times B_{R+1}$, we have
\bel{cond31}\norm{b_{1,\rho}}_{L^\infty(0,T; W^{k,\infty}(\R^{n}))}+ \norm{b_{2,\rho}}_{L^\infty(0,T; W^{k,\infty}(\R^{n}))}\leq C_k\rho^{{k\over3}},\quad k\geq1\ee
\bel{cond33}\norm{b_{1,\rho}}_{W^{1,\infty}(0,T; W^{k,\infty}(\R^{n}))}+ \norm{b_{2,\rho}}_{W^{1,\infty}(0,T; W^{k,\infty}(\R^{n}))}\leq C_k\rho^{{k+1\over3}},\quad k\geq1\ee
and
\bel{cond4} b_{1,\rho}(x,0)=b_{2,\rho}(x,T)=0,\quad x\in\Omega.\ee
Here $C_k$, $k\in\mathbb N$, denotes a constant independent of $\rho>0$.
Moreover, conditions \eqref{a1a}-\eqref{a1b} and \eqref{cond31} imply that, for any open bounded subset $\tilde{\Omega}$ of $\R^n$ and for $\tilde{Q}=\tilde{\Omega}\times(0,T)$, we have
\bel{cond5} \lim_{\rho\to+\infty}\norm{(2\omega\cdot \nabla_x-(A_1\cdot \omega))b_{1,\rho}}_{L^2(\tilde{Q})}=\lim_{\rho\to+\infty}\norm{[(A_{1,\rho}-A_1)\cdot \omega]b_{1,\rho}}_{L^2(\tilde{Q})}=0,\ee
\bel{cond6} \lim_{\rho\to+\infty}\norm{(2\omega\cdot \nabla_x+(A_2\cdot \omega))b_{2,\rho}}_{L^2(\tilde{Q})}=\lim_{\rho\to+\infty}\norm{[(A_2-A_{2,\rho})\cdot \omega]b_{1,\rho}}_{L^2(\tilde{Q})}=0,\ee
\subsection{Carleman estimates in negative order Sobolev space}
In order to complete the construction of the GO taking the form \eqref{Gsol1}-\eqref{Gsol2}  we recall some preliminary tools and we derive two Carleman estimates in Sobolev space of negative order.  In a similar way to \cite{Ki3}, for all $m\in\R$, we introduce the space $H^m_\rho(\R^{n})$ defined by
\[H^m_\rho(\R^{n})=\{u\in\mathcal S'(\R^{n}):\ (|\xi|^2+\rho^2)^{m\over 2}\hat{u}\in L^2(\R^{n})\},\]
with the norm
\[\norm{u}_{H^m_\rho(\R^{n})}^2=\int_{\R^n}(|\xi|^2+\rho^2)^{m}|\hat{u}(\xi)|^2 d\xi .\]
Here for all tempered distributions $u\in \mathcal S'(\R^{n})$, we denote by $\hat{u}$ the Fourier transform of $u$ which, for $u\in L^1(\R^{n})$, is defined by
$$\hat{u}(\xi):=\mathcal Fu(\xi):= (2\pi)^{-{n\over2}}\int_{\R^{n}}e^{-ix\cdot \xi}u(x)dx.$$
From now on, for $m\in\R$ and $\xi\in \R^n$,  we set $$\left\langle \xi,\rho\right\rangle=(|\xi|^2+\rho^2)^{1\over2}$$
and $\left\langle D_x,\rho\right\rangle^m u$ defined by
\[\left\langle D_x,\rho\right\rangle^m u=\mathcal F^{-1}(\left\langle \xi,\rho\right\rangle^m \mathcal Fu).\]
For $m\in\R$ we define also the class of symbols
\[S^m_\rho=\{c_\rho\in\mathcal C^\infty(\R^{n}\times\R\times\R^{n}):\ |\pd_t^k\pd_x^\alpha\pd_\xi^\beta c_\rho(x,t,\xi)|\leq C_{k,\alpha,\beta}\left\langle \xi,\rho\right\rangle^{m-|\beta|},\  \alpha,\beta\in\mathbb N^n,\ k\in\mathbb N\}.\]
Following \cite[Theorem 18.1.6]{Ho3}, for any $m\in\R$ and $c_\rho\in S^m_\rho$, we define $c_\rho(x,t,D_x)$, with  $D_x=-i\nabla_x$, by
\[c_\rho(x,t,D_x)z(x)=(2\pi)^{-{n\over 2}}\int_{\R^{n}}c_\rho(x,t,\xi)\hat{z}(\xi)e^{ix\cdot \xi} d\xi,\quad z\in\mathcal S(\R^n).\]
For all $m\in\R$, we set also $OpS^m_\rho:=\{c_\rho(x,t,D_x):\ c_\rho\in S^m_\rho\}$.
We fix
$$P_{A,B,q,\pm}:=e^{\mp (\rho^2t+\rho x\cdot\omega)}(L_{\pm, A}+\nabla_x\cdot B+q)e^{\pm (\rho^2t+\rho x\cdot\omega)}$$
and we consider the following Carleman estimate. 

\begin{prop}\label{l1} Let $A,B\in L^\infty(Q)^n$ and $q\in L^\infty(0,T;L^{p}(\Omega))\cup\mathcal C([0,T];L^{\frac{2n}{3}}(\Omega))$ with $p > 2n/3$. Then, there exists $\rho_2'>\rho_2$, depending only on $\Omega$, $T$ and $M\geq \norm{A}_{L^\infty(Q)^n}+\norm{B}_{L^\infty(Q)^n}$,  such that  for all $v\in \mathcal C^1([0,T];\mathcal C^\infty_0(\Omega))$ satisfying $v_{|\Omega^T}=0$ we have 
\bel{car2}(\rho^{-\frac{1}{2}}\norm{v}_{L^2(0,T; H^{1}(\R^{n}))}+\norm{v}_{L^2(0,T; L^2(\R^{n}))})\leq C\norm{P_{A,B,q,-}v}_{L^2(0,T;H^{-1}_\rho(\R^{n}))},\quad \rho>\rho_2',\ee
with $C>0$ depending on $\Omega$, $T$ and $M\geq \norm{A}_{L^\infty(Q)^n}+\norm{B}_{L^\infty(Q)^n}+\norm{q}_{L^\infty(0,T;L^{p}(\Omega))}$, when $q\in L^\infty(0,T;L^{p}(\Omega))$ and $M\geq \norm{A}_{L^\infty(Q)^n}+\norm{B}_{L^\infty(Q)^n}+\norm{q}_{L^\infty(0,T;L^{\frac{2n}{3}}(\Omega))}$ when $q\in \mathcal C([0,T];L^{\frac{2n}{3}}(\Omega))$.
\end{prop}
\begin{proof} 
For $\phi_{\rho,s}$ given by \eqref{phi}, we  consider
$$P_{A,B,q,\pm,s}:=e^{-\phi_{\pm,s}}(L_{\pm, A}+q+\nabla_x\cdot B)e^{\phi_{\pm,s}},\quad P_{A,-,s}=e^{-\phi_{\pm,s}}L_{\pm, A}e^{\phi_{\pm,s}}$$
and in a similar way to Proposition \ref{pp1} we decompose $P_{A,B,q,-,s}$ into three terms
\[P_{A,B,q,-,s}=P_{1,-}+P_{2,-}+P_{3,-,A,B,q},\]
with
\[P_{1,-}=-\Delta_x+2\rho s((x+x_0)\cdot\omega+s^2((x+x_0)\cdot\omega)^2-s,\quad P_{2,-}=-\pd_t-2[\rho-s((x+x_0)\cdot\omega)]\omega\cdot\nabla_x+2s.\]
\[ P_{3,-,A,B,q}=A\cdot\nabla_x-(\rho+s((x+x_0)\cdot\omega))A\cdot\omega+\nabla_x\cdot B+q.\]
We pick $ \tilde{\Omega}$  a bounded open and smooth set of $\R^n$ such that $\overline{\Omega}\subset\tilde{\Omega}$ and we extend the function  $A$, $B$, $q$ by zero  to $\R^{n}\times (0,T)$. In order to prove \eqref{car2},  we fix $w\in \mathcal C^1([0,T];\mathcal C^\infty_0(\tilde{\Omega}))$ satisfying $w_{|\Omega^T}=0$ and we  consider the quantity
\[\left\langle D_x,\rho\right\rangle^{-1}(P_{1,-}+P_{2,-})\left\langle D_x,\rho\right\rangle w.\]
Here for any $z\in \mathcal C^\infty([0,T];\mathcal C^\infty_0(\tilde{\Omega}))$ we define 
\[\left\langle D_x,\rho\right\rangle^m z(x,t)=\mathcal F^{-1}_x(\left\langle \xi,\rho\right\rangle^m \mathcal F_xz(\cdot,t))(x).\]
where the partial Fourier transform $\mathcal F_x$ is defined by
$$\mathcal F_xz(t,\xi):= (2\pi)^{-{n\over2}}\int_{\R^{n}}e^{-ix\cdot \xi}z(x,t)dx.$$
In all the remaining parts of this proof $C>0$ denotes a generic constant depending on $\Omega$, $T$, $M$.
Combining the properties of composition of pseudoddifferential operators (e.g. \cite[Theorem 18.1.8]{Ho3}) with the fact that $\left\langle D_x,\rho\right\rangle^{-1}$ commutes with $\pd_t$, we find
\bel{l2c}\left\langle D_x,\rho\right\rangle^{-1}(P_{1,-}+P_{2,-})\left\langle D_x,\rho\right\rangle=P_{1,-}+P_{2,-}+R_\rho(x,D_x),\ee
where $R_\rho$ is defined by
\[R_\rho(x,\xi)=\nabla_\xi\left\langle \xi,\rho\right\rangle^{-1}\cdot D_x(p_{1,-}(x,\xi)+p_{2,-}(x,\xi))\left\langle \xi,\rho\right\rangle+\underset{\left\langle \xi,\rho\right\rangle\to+\infty}{ o}(1),\]
with
$$p_{1,-}(x,\xi)=|\xi|^2+2\rho s(x+x_0)\cdot\omega+s^2((x+x_0)\cdot\omega)^2-s,\quad p_{2,-}(x,\xi)=-2i[\rho-s((x+x_0)\cdot\omega)]\omega\cdot\xi+2s.$$
Therefore, we have
$$R_\rho(x,\xi)={i[2\rho s+2s^2(x+x_0)\cdot\omega+2is(\omega\cdot\xi)](\omega\cdot\xi)\over |\xi|^2+\rho^2}+\underset{\left\langle \xi,\rho\right\rangle\to+\infty}{ o}(1)$$
and it follows
\bel{l2d} \norm{R_\rho(x,D_x)w}_{L^2((0,T)\times \R^n)}\leq Cs^2\norm{w}_{L^2((0,T)\times \R^n)}.\ee
On the other hand,  applying \eqref{lll1b} to $w$ with $Q$ replaced by  $\tilde{Q}=(0,T)\times\tilde{\Omega}$,  we get
\bel{tre}\norm{P_{1,-}w+P_{2,-}w}_{L^2((0,T)\times\R^{n})}\geq C\left(s^{-1/2}\norm{\Delta_x w}_{L^2((0,T)\times\R^{n})}+s^{1/2}\rho\norm{ w}_{L^2((0,T)\times\R^{n})}\right).\ee
Moreover, using the fact that supp$(w)\subset\tilde{\Omega}$ and the elliptic regularity of the operator $\Delta$ we deduce that
$$\norm{ w}_{L^2(0,T;H^2(\R^{n}))}\leq C\norm{\Delta_x w}_{L^2((0,T)\times\R^{n})},$$
where in both of these estimates $C>0$ depends only on $\tilde{\Omega}$ and $T$. By interpolation,  we deduce that 
$$\begin{aligned}s^{1/2}\norm{ w}_{L^2(0,T;H^1(\R^{n}))}&\leq \left(s^{-1/2}\norm{ w}_{L^2(0,T;H^2(\R^{n}))}\right)^{\frac{1}{2}}\left(s^{3/2}\norm{ w}_{L^2(0,T;L^2(\R^{n}))}\right)^{\frac{1}{2}}\\
\ &\leq s^{-1/2}\norm{ w}_{L^2(0,T;H^2(\R^{n}))}+s^{1/2}\rho\norm{ w}_{L^2(0,T;L^2(\R^{n}))}.\end{aligned}$$
Combining these two estimates with \eqref{tre}, we get
\[\norm{P_{1,-}w+P_{2,-}w}_{L^2((0,T)\times\R^{n})}\geq C\left(s^{-1/2}\norm{ w}_{L^2(0,T;H^2(\R^{n}))}+s^{1/2}\norm{ w}_{L^2(0,T;H^1_\rho(\R^{n}))}\right).\]
Combining this estimate with \eqref{l2c}-\eqref{l2d}, for ${\rho\over s^2}$ sufficiently large, we obtain
 \bel{l2e}\begin{array}{l} \norm{(P_{1,-}+P_{2,-})\left\langle D_x,\rho\right\rangle w}_{L^2(0,T;H^{-1}_\rho(\R^{n}))}\\
=\norm{\left\langle D_x,\rho\right\rangle^{-1}(P_{1,-}+P_{2,-})\left\langle D_x,\rho\right\rangle w}_{L^2((0,T)\times \R^n)}\\ \geq  C\left(s^{-1/2}\norm{ w}_{L^2(0,T;H^2(\R^{n}))}+s^{1/2}\norm{ w}_{L^2(0,T;H^1_\rho(\R^{n}))}\right).\end{array}\ee
Moreover, we have
 \bel{l2f}\begin{aligned}&\norm{P_{3,-,A,B,q}\left\langle D_x,\rho\right\rangle w}_{L^2(0,T;H^{-1}_\rho(\R^{n}))}\\
&\leq \norm{A\cdot\nabla_x\left\langle D_x,\rho\right\rangle w}_{L^2(0,T;H^{-1}_\rho(\R^{n}))}+\norm{(\rho+s((x+x_0)\cdot\omega))A\cdot\omega\left\langle D_x,\rho\right\rangle w}_{L^2(0,T;H^{-1}_\rho(\R^{n}))}\\
&\ \ \ \ +\norm{(\nabla_x\cdot B)\left\langle D_x,\rho\right\rangle w}_{L^2(0,T;H^{-1}_\rho(\R^{n}))}+\norm{q\left\langle D_x,\rho\right\rangle w}_{L^2(0,T;H^{-1}_\rho(\R^{n}))}.\end{aligned}\ee
For the first term on the right hand side of \eqref{l2f}, we find
 \bel{l2g}\begin{aligned}\norm{A\cdot\nabla_x\left\langle D_x,\rho\right\rangle w}_{L^2(0,T;H^{-1}_\rho(\R^{n}))}&\leq\rho^{-1}\norm{A\cdot\nabla_x\left\langle D_x,\rho\right\rangle w}_{L^2(0,T;L^2(\R^{n}))}\\
\ &\leq \norm{A}_{L^\infty(Q)}\rho^{-1}\norm{\nabla_x\left\langle D_x,\rho\right\rangle w}_{L^2(0,T;L^2(\R^{n}))}\\
\ &\leq C\norm{A}_{L^\infty(Q)}\left(\rho^{-1}\norm{w}_{L^2(0,T;H^2(\R^n))}+\norm{w}_{L^2(0,T;H^1(\R^n))}\right).\end{aligned}\ee
For the second term on the right hand side of \eqref{l2f}, we get
 \bel{l2h}\begin{aligned}\norm{(\rho+s((x+x_0)\cdot\omega))A\cdot\omega\left\langle D_x,\rho\right\rangle w}_{L^2(0,T;H^{-1}_\rho(\R^{n}))}&\leq \rho^{-1}\norm{(\rho+s((x+x_0)\cdot\omega))A\cdot\omega\left\langle D_x,\rho\right\rangle w}_{L^2(0,T;L^2(\R^{n}))}\\
\ &\leq \left(1+|x_0|+\sup_{x\in\overline{\Omega}}|x|\right)\norm{A}_{L^\infty(Q)}\norm{\left\langle D_x,\rho\right\rangle w}_{L^2(0,T;L^2(\R^{n}))}\\
\ &\leq C\norm{A}_{L^\infty(Q)}\norm{w}_{L^2(0,T;H^1_\rho(\R^n))}.\end{aligned}\ee
 For the third term on the right hand side of \eqref{l2f}, we have
\bel{l2i}\begin{aligned}\norm{(\nabla_x\cdot B)\left\langle D_x,\rho\right\rangle w}_{L^2(0,T;H^{-1}_\rho(\R^{n}))}&\leq \norm{B\left\langle D_x,\rho\right\rangle w}_{L^2(0,T;L^2(\R^{n}))}+ \rho^{-1}\norm{B\cdot \nabla_x\cdot \left\langle D_x,\rho\right\rangle w}_{L^2(0,T;L^2(\R^{n}))}\\
\ &\leq C\norm{B}_{L^\infty(Q)}\left(\rho^{-1}\norm{w}_{L^2(0,T;H^2(\R^n))}+\norm{w}_{L^2(0,T;H^1_\rho(\R^n))}\right).\end{aligned}\ee
Finally, for the last term on the right hand side of \eqref{l2f}, we will prove that there exists $\rho_1''(s)>\rho_1(s)$,  with $\rho_1(s)$ given by Proposition \ref{pp1},  such that the estimate
\bel{l2k}\norm{q\left\langle D_x,\rho\right\rangle w}_{L^2(0,T;H^{-1}_\rho(\R^{n}))}\leq C[s^{-1}\norm{w}_{H^2(\R^n)}+\rho\norm{w}_{L^2(\R^n)}]\ee
holds true for $\rho>\rho_1''(s)$. For this purpose, let us first  assume that $n\geq3$ and $q\in\mathcal C([0,T];L^{\frac{2n}{3}}(\Omega))$. We consider
 \bel{qr}q_\rho(x,t):=\int_{\R^{1+n}}\rho^{\frac{n}{4}}h(\rho^{\frac{1}{4}} (x-y))q(y,t)dy,\ee
with $q$ extended by zero to $\R^n\times(0,T)$ and with $h\in\mathcal C^\infty_0(\R^n;[0,+\infty))$ satisfying  supp$(h)\subset\{x\in\R^n:\ |x|<1\}$ and
$$\int_{\R^n}h(x)dx=1.$$
We have the following result.
\begin{lem}\label{ll1} Let $p_2\in[1,+\infty)$, $q\in\mathcal C([0,T];L^{p_2}(\Omega))$ and $q_\rho$ given by \eqref{qr}. Then, we have
\bel{ll1a}\lim_{\rho\to+\infty}\norm{q_\rho-q}_{L^\infty(0,T;L^{p_2}(\R^n))}.\ee
\end{lem}
We will prove this result when finished the present proof. For all $\psi\in L^2(0,T;\mathcal C^\infty_0(\R^n))$ we have
$$\abs{\left\langle q\left\langle D_x,\rho\right\rangle w, \psi\right\rangle_{L^2(0,T;H^{-1}_\rho(\R^{n})), L^2(0,T;H^{1}_\rho(\R^{n}))}}\leq\int_0^T\int_{\R^n}(|q-q_\rho|+|q_\rho|)|\left\langle D_x,\rho\right\rangle w||\psi|dxdt.$$
Applying the H\"older inequality, for $n\geq3$, we get
\bel{l2j}\begin{aligned}&\abs{\left\langle q\left\langle D_x,\rho\right\rangle w, \psi\right\rangle_{L^2(0,T;H^{-1}_\rho(\R^{n})), L^2(0,T;H^{1}_\rho(\R^{n}))}}\\
&\leq \norm{q-q_\rho}_{L^\infty(0,T;L^{\frac{2n}{3}}(\R^n))}\norm{\left\langle D_x,\rho\right\rangle w}_{L^2(0,T;L^{\frac{2n}{n-2}}(\R^n))}\norm{\psi}_{L^2(0,T;L^{\frac{2n}{n-1}}(\R^n))}\\
&\ \ +\norm{q_\rho}_{L^\infty(Q)}\norm{\left\langle D_x,\rho\right\rangle w}_{L^2(\R^n\times(0,T))}\norm{\psi}_{L^2(\R^n\times(0,T))}\end{aligned}\ee
For the first term on the right hand side of \eqref{l2j}, applying the Sobolev embedding theorem, we find
$$\begin{aligned} &\norm{q-q_\rho}_{L^\infty(0,T;L^{\frac{2n}{3}}(\R^n))}\norm{\left\langle D_x,\rho\right\rangle w}_{L^2(0,T;L^{\frac{2n}{n-2}}(\R^n))}\norm{\psi}_{L^2(0,T;L^{\frac{2n}{n-1}}(\R^n))}\\
&\leq \norm{q-q_\rho}_{L^\infty(0,T;L^{\frac{2n}{3}}(\R^n))}\norm{\left\langle D_x,\rho\right\rangle w}_{L^2(0,T;H^1(\R^n)}\norm{\psi}_{L^2(0,T;H^{\frac{1}{2}}(\R^n))}.\end{aligned}$$
Moreover, by interpolation, we obtain
$$\begin{aligned}\norm{\psi}_{L^2(0,T;H^{\frac{1}{2}}(\R^n))}&\leq \norm{\psi}_{L^2(0,T;H_\rho^{\frac{1}{2}}(\R^n))}\\
\ &\leq \left(\norm{\psi}_{L^2(0,T;H_\rho^{1}(\R^n))}\right)^{\frac{1}{2}}\left(\norm{\psi}_{L^2(0,T;L^2(\R^n))}\right)^{\frac{1}{2}}\\
\ &\leq C\rho^{-\frac{1}{2}}\norm{\psi}_{L^2(0,T;H_\rho^{1}(\R^n))}\end{aligned}$$
and we deduce that 
$$\begin{aligned} &\norm{q-q_\rho}_{L^\infty(0,T;L^{\frac{2n}{3}}(\R^n))}\norm{\left\langle D_x,\rho\right\rangle w}_{L^2(0,T;L^{\frac{2n}{n-2}}(\R^n))}\norm{\psi}_{L^2(0,T;L^{\frac{2n}{n-1}}(\R^n))}\\
&\leq C\norm{q-q_\rho}_{L^\infty(0,T;L^{\frac{2n}{3}}(\R^n))}\left(\rho^{-\frac{1}{2}}\norm{\left\langle D_x,\rho\right\rangle w}_{L^2(0,T;H^1(\R^n))}\right)\norm{\psi}_{L^2(0,T;H_\rho^{1}(\R^n))}\\
&\leq C\norm{q-q_\rho}_{L^\infty(0,T;L^{\frac{2n}{3}}(\R^n))}\left[\rho^{-\frac{1}{2}}\norm{w}_{L^2(0,T;H^2(\R^n))}+\rho^{\frac{1}{2}}\norm{w}_{L^2(0,T;H^1(\R^n))}\right]\norm{\psi}_{L^2(0,T;H_\rho^{1}(\R^n))}\\
&\leq C\rho^{-\frac{1}{2}}\norm{q-q_\rho}_{L^\infty(0,T;L^{\frac{2n}{3}}(\R^n))}\norm{w}_{L^2(0,T;H^2(\R^n))}\\
&\ \ \ +C\norm{q-q_\rho}_{L^\infty(0,T;L^{\frac{2n}{3}}(\R^n))}(s^{-\frac{1}{2}}\norm{w}_{L^2(0,T;H^2(\R^n))})^{\frac{1}{2}}(s^{\frac{1}{2}}\rho\norm{w}_{L^2(\R^n\times(0,T))})^{\frac{1}{2}}\norm{\psi}_{L^2(0,T;H_\rho^{1}(\R^n))}\\
&\leq C\norm{q-q_\rho}_{L^\infty(0,T;L^{\frac{2n}{3}}(\R^n))}\left[s^{-\frac{1}{2}}\norm{w}_{L^2(0,T;H^2(\R^n))}+s^{\frac{1}{2}}\rho\norm{w}_{L^2(\R^n\times(0,T))}\right]\norm{\psi}_{L^2(0,T;H_\rho^{1}(\R^n))}.\end{aligned}$$
In the same way, for the second term on the right hand side of \eqref{l2j}, applying the Sobolev embedding theorem, we obtain
$$\begin{aligned}&\norm{q_\rho}_{L^\infty(Q)}\norm{\left\langle D_x,\rho\right\rangle w}_{L^2(Q)}\norm{\psi}_{L^2(Q)}\\
&\leq C\norm{q_\rho}_{L^\infty(0,T;W^{2,\frac{2n}{3}}(\R^n))}\norm{\left\langle D_x,\rho\right\rangle w}_{L^2(\R^n\times(0,T))}\norm{\psi}_{L^2(\R^n\times(0,T))}\\
&\leq C\rho^{\frac{1}{2}}\norm{ w}_{L^2(0,T;H^1_\rho(\R^n))}(\rho^{-1}\norm{\psi}_{L^2(0,T;H^1_\rho(\R^n))})\\
&\leq C\rho^{-\frac{1}{2}}\left[s^{-\frac{1}{2}}\norm{w}_{L^2(0,T;H^2(\R^n))}+s^{\frac{1}{2}}\rho\norm{w}_{L^2(0,T;H^1(\R^n))}\right]\norm{\psi}_{L^2(0,T;H^1_\rho(\R^n))})\end{aligned}$$
Combining these two estimates with \eqref{l2j}, we obtain 
$$\begin{aligned}&\abs{\left\langle q\left\langle D_x,\rho\right\rangle w, \psi\right\rangle_{L^2(0,T;H^{-1}_\rho(\R^{n})), L^2(0,T;H^{1}_\rho(\R^{n}))}}\\
&\leq C[\norm{q-q_\rho}_{L^\infty(0,T;L^{\frac{2n}{3}}(\R^n))}+\rho^{-\frac{1}{2}}]\left[s^{-\frac{1}{2}}\norm{w}_{L^2(0,T;H^2(\R^n))}+s^{\frac{1}{2}}\rho\norm{w}_{L^2(\R^n\times(0,T))}\right]\norm{\psi}_{L^2(0,T;H^1_\rho(\R^n))}\end{aligned}$$
and we deduce that 
$$\begin{aligned}&\norm{q\left\langle D_x,\rho\right\rangle w}_{L^2(0,T;H^{-1}_\rho(\R^{n}))}\\
&\leq C[\norm{q-q_\rho}_{L^\infty(0,T;L^{\frac{2n}{3}}(\R^n))}+\rho^{-\frac{1}{2}}]\left[s^{-\frac{1}{2}}\norm{w}_{L^2(0,T;H^2(\R^n))}+s^{\frac{1}{2}}\rho\norm{w}_{L^2(\R^n\times(0,T))}\right].\end{aligned}$$
On the other hand, using the fact that
$$\lim_{\rho\to+\infty}[\norm{q-q_\rho}_{L^\infty(0,T;L^{\frac{2n}{3}}(\R^n))}+\rho^{-\frac{1}{2}}]=0,$$
we can find $\rho_1''(s)>\rho_1(s)$ such that for $\rho>\rho_1''(s)$ we have
$$[\norm{q-q_\rho}_{L^\infty(0,T;L^{\frac{2n}{3}}(\R^n))}+\rho^{-\frac{1}{2}}]\leq s^{-\frac{1}{2}}.$$
Thus, we obtain \eqref{l2k}. In the same way we can deduce  \eqref{l2k} for $n=2$ and $q\in\mathcal C([0,T];L^{\frac{2n}{3}}(\Omega))$. Now let us show \eqref{l2k} for $n\geq3$ and $q\in L^\infty(0,T;L^p(\Omega))$, for $p< n$. In that case, applying the H\"older inequality,  we get
$$\begin{aligned}&\abs{\left\langle q\left\langle D_x,\rho\right\rangle w, \psi\right\rangle_{L^2(0,T;H^{-1}_\rho(\R^{n})), L^2(0,T;H^{1}_\rho(\R^{n}))}}\\
&\leq \norm{q}_{L^\infty(0,T;L^p(\Omega))}\norm{\left\langle D_x,\rho\right\rangle w}_{L^2\left(0,T;L^{\frac{2n}{n-2}}(\R^n)\right)}\norm{\psi}_{L^2\left(0,T;L^{\frac{2n}{n-2(\frac{n}{p}-1)}}(\R^n)\right)}.\end{aligned}$$
Using the Sobolev embedding theorem, we have
$$\begin{aligned}&\abs{\left\langle q\left\langle D_x,\rho\right\rangle w, \psi\right\rangle_{L^2(0,T;H^{-1}_\rho(\R^{n})), L^2(0,T;H^{1}_\rho(\R^{n}))}}\\
&\leq C\norm{q}_{L^\infty(0,T;L^p(\Omega))}\norm{\left\langle D_x,\rho\right\rangle w}_{L^2(0,T;H^{1}(\R^n))}\norm{\psi}_{L^2(0,T;H^{\frac{n}{p}-1}(\R^n))}.\end{aligned}$$
On the other hand, by interpolation we find
$$\begin{aligned}\norm{\psi}_{L^2(0,T;H^{\frac{n}{p}-1}(\R^n))}&\leq \norm{\psi}_{L^2(0,T;H_\rho^{\frac{n}{p}-1}(\R^n))}\\
\ &\leq \left(\norm{\psi}_{L^2(0,T;H_\rho^{1}(\R^n))}\right)^{\frac{n}{p}-1}\left(\norm{\psi}_{L^2(0,T;L^2(\R^n))}\right)^{2-\frac{n}{p}}\\
\ &\leq C\rho^{\frac{n}{p}-2}\norm{\psi}_{L^2(0,T;H_\rho^{1}(\R^n))}\end{aligned}$$
and we deduce that
$$\begin{aligned}&\norm{q\left\langle D_x,\rho\right\rangle w}_{L^2(0,T;H^{-1}_\rho(\R^{n}))}\\
&\leq C\norm{q}_{L^\infty(0,T;L^p(\Omega)}\rho^{\frac{n}{p}-2}\norm{\left\langle D_x,\rho\right\rangle w}_{L^2(0,T;H^1(\R^n))}\\
&\leq C\norm{q}_{L^\infty(0,T;L^p(\Omega)}\rho^{\frac{n}{p}-\frac{3}{2}}[s^{-\frac{1}{2}}\norm{w}_{H^2(\R^n)}+s^{\frac{1}{2}}\rho\norm{w}_{L^2(\R^n)}].\end{aligned}$$
Using the fact that $\frac{3}{2}>\frac{n}{p}$, we deduce \eqref{l2k}, for $n\geq3$ and $q\in L^\infty(0,T;L^p(\Omega))$, from this estimate. We prove in the same way, \eqref{l2k}, for $n=2$ and $q\in L^\infty(0,T;L^p(\Omega))$.
Combining \eqref{l2f}-\eqref{l2k} with \eqref{l2e}, for $s=C(\norm{A}_{L^\infty(Q)}+\norm{B}_{L^\infty(Q)})+1$, for some constant $C>0$ depending only on $\Omega$, $T$ but suitably chosen, 
we find
\bel{l2l}\norm{P_{A,B,q,-,s}\left\langle D_x,\rho\right\rangle w}_{L^2(0,T;H^{-1}_\rho(\R^{n}))}\geq C\left(\norm{w}_{L^2(0,T;H^2(\R^{n}))}+\norm{w}_{L^2(0,T;H^1_\rho(\R^{n}))}\right).\ee
We fix $\psi_0\in\mathcal C^\infty_0(\tilde{\Omega})$ satisfying $\psi_0=1$ on $\overline{\Omega_1}$, with  $\Omega_1$ an open neighborhood of $\overline{\Omega}$ such that $\overline{\Omega_1}\subset\tilde{\Omega}$. Then, we fix $w=\psi_0(x) \left\langle D_x,\rho\right\rangle^{-1} v(x,t)$ and for $\psi_1\in\mathcal C^\infty_0(\Omega_1)$ satisfying $\psi_1=1$ on $\Omega$, we get $(1-\psi_0 )\left\langle D_x,\rho\right\rangle^{-1} v=(1-\psi_0 )\left\langle D_x,\rho\right\rangle^{-1}\psi_1 v$. According to \cite[Theorem 18.1.8]{Ho3}, we have $(1-\psi_0) \left\langle D_x,\rho\right\rangle^{-1}\psi_1\in OpS^{-\infty}_\rho$ and it follows
\[ \begin{aligned}\norm{v}_{L^2((0,T)\times\R^{n})}&=\norm{\left\langle D_x,\rho\right\rangle^{-1} v}_{L^2(0,T;H^1_\rho(\R^{n}))}\\
\ &\leq \norm{w}_{L^2(0,T;H^1_\rho(\R^{n}))}+\norm{(1-\psi_0)\left\langle D_x,\rho\right\rangle^{-1}\psi_1 v}_{L^2(0,T;H^1_\rho(\R^{n}))}\\
\ &\leq \norm{w}_{L^2(0,T;H^1_\rho(\R^{n}))}+{C\norm{v}_{L^2((0,T)\times\R^{n})}\over\rho^2} .\end{aligned}\]
In addition, by interpolation, we get
$$\begin{aligned}&\rho^{-1}\norm{v}_{L^2(0,T;H^1(\R^{n}))}^2\\
&\leq \norm{\left\langle D_x,\rho\right\rangle^{-1} v}_{L^2(0,T;H^1_\rho(\R^{n}))}^2+\norm{\left\langle D_x,\rho\right\rangle^{-1} v}_{L^2(0,T;H^2(\R^{n}))}^2+\rho\norm{\left\langle D_x,\rho\right\rangle^{-1} v}_{L^2(0,T;H^1(\R^{n}))}^2\\
&\leq \norm{\left\langle D_x,\rho\right\rangle^{-1} v}_{L^2(0,T;H^1_\rho(\R^{n}))}^2+\norm{\left\langle D_x,\rho\right\rangle^{-1} v}_{L^2(0,T;H^2(\R^{n}))}^2+\rho\norm{\left\langle D_x,\rho\right\rangle^{-1} v}_{L^2(0,T;H^2(\R^{n}))}\norm{\left\langle D_x,\rho\right\rangle^{-1} v}_{L^2(0,T;L^2(\R^{n}))}\\
&\leq 2\norm{\left\langle D_x,\rho\right\rangle^{-1} v}_{L^2(0,T;H^1_\rho(\R^{n}))}^2+2\norm{\left\langle D_x,\rho\right\rangle^{-1} v}_{L^2(0,T;H^2(\R^{n}))}^2\end{aligned}$$
and it follows
\[ \begin{aligned}\rho^{-\frac{1}{2}}\norm{v}_{L^2(0,T;H^1(\R^{n}))}&\leq4\norm{\left\langle D_x,\rho\right\rangle^{-1} v}_{L^2(0,T;H^1_\rho(\R^{n}))}+4\norm{\left\langle D_x,\rho\right\rangle^{-1} v}_{L^2(0,T;H^2(\R^{n}))}\\
\ &\leq 4\norm{w}_{L^2(0,T;H^1_\rho(\R^{n}))}+\norm{(1-\psi_0)\left\langle D_x,\rho\right\rangle^{-1}\psi_1 v}_{L^2(0,T;H^1_\rho(\R^{n}))}\\
\ &\ \ \ +4\norm{w}_{L^2(0,T;H^2(\R^{n}))}+\norm{(1-\psi_0)\left\langle D_x,\rho\right\rangle^{-1}\psi_1 v}_{L^2(0,T;H^2(\R^{n}))}\\
\ &\leq 4\norm{w}_{L^2(0,T;H^1_\rho(\R^{n}))}+4\norm{w}_{L^2(0,T;H^2(\R^{n}))}+{C\norm{v}_{L^2((0,T)\times\R^{n})}\over\rho^2} .\end{aligned}\]

Thus, applying \eqref{l2l} for a fixed value of $s$, we deduce that there exists $\rho_2'>0$ such that \eqref{car2} is fulfilled.\end{proof}

Now that the proof of Lemma \ref{l1} is completed, let us consider the proof of Lemma \ref{ll1}.

\textbf{Proof of Lemma \ref{ll1}.} We fix $\epsilon_1>0$ and we will prove that 
\bel{ll1b}\limsup_{\rho\to+\infty}\norm{q_\rho-q}_{L^\infty(0,T;L^{p_2}(\R^n)}\leq 2\epsilon_1.\ee
For this purpose, using the fact that $t\mapsto q_\rho(\cdot,t)\in \mathcal C([0,T];L^{p_2}(\R^n))$, there exists $\delta>0$ such that for all $t,t'\in[0,T]$ satisfying $|t-t'|<\delta$ we have 
\bel{ll1c}\norm{q(\cdot ,t)]-q(\cdot ,t')}_{L^{p_2}(\R^n)}\leq \epsilon_1.\ee
Using the fact that $[0,T]$ is compact, we can find $t_1,\ldots,t_N$ such that
$$[0,T]\subset \bigcup_{j=1}^N(t_j-\delta,t_j+\delta)$$
and, using the fact that
$$\lim_{\rho\to+\infty}\norm{q_\rho(\cdot,t)-q(\cdot ,t)}_{L^{p_2}(\R^n)}=0,\quad t\in[0,T],$$
we get
\bel{ll1f}\lim_{\rho\to+\infty}\underset{j=1,\ldots,N}{\max}\norm{q_\rho(\cdot,t_j)-q(\cdot ,t_j)}_{L^{p_2}(\R^n)}=0,\quad j=1,\ldots,N.\ee
Thus, for all $t\in[0,T]$ there exists $k\in\{1,\ldots N\}$ such that $|t-t_k|<\delta$ and, applying  \eqref{ll1c} and the Young inequality, we get
$$\begin{aligned}&\norm{q_\rho(\cdot,t)-q(\cdot ,t)}_{L^{p_2}(\R^n)}\\
&\leq \norm{q(\cdot ,t)-q(\cdot ,t_k)}_{L^{p_2}(\R^n)}+\norm{q_\rho(\cdot ,t)-q_\rho(\cdot ,t_k)}_{L^{p_2}(\R^n)}+\norm{q_\rho(\cdot,t_k)-q(\cdot ,t_k)}_{L^{p_2}(\R^n)}\\
&\leq 2\norm{q(\cdot ,t)-q(\cdot ,t_k)}_{L^{p_2}(\R^n)}+\underset{j=1,\ldots,N}{\max}\norm{q_\rho(\cdot,t_j)-q(\cdot ,t_j)}_{L^{p_2}(\R^n)}\\
&\leq2\epsilon_1+\underset{j=1,\ldots,N}{\max}\norm{q_\rho(\cdot,t_j)-q(\cdot ,t_j)}_{L^{p_2}(\R^n)}.\end{aligned}$$
Therefore, we have
$$\norm{q_\rho-q}_{L^\infty(0,T;L^{p_2}(\R^n))}\leq 2\epsilon_1+\underset{j=1,\ldots,N}{\max}\norm{q_\rho(\cdot,t_j)-q(\cdot ,t_j)}_{L^{p_2}(\R^n)}$$
and using \eqref{ll1f}, we obtain \eqref{ll1b} from which we deduce \eqref{ll1a}.\qed

In a similar way to Proposition \ref{l1}, combining estimate \eqref{5} with the arguments of Lemma \ref{l1}, we deduce the following estimate.

\begin{prop}\label{l8} There exists $\rho_4'>\rho_3$ such that for $\rho>\rho_4'$ and for any $v\in \mathcal C^1([0,T];\mathcal C^\infty_0(\Omega))$ satisfying $v_{|\Omega^0}=0$,  we have
\bel{l8a}(\rho^{-\frac{1}{2}}\norm{v}_{L^2(0,T; H^{1}(\R^{n}))}+\norm{v}_{L^2(0,T; L^2(\R^{n}))})\leq C\norm{P_{A,B,q,+}v}_{L^2(0,T;H^{-1}_\rho(\R^{n}))},\quad \rho>\rho_4'\ee
with $C>0$  depending on $\Omega$, $T$ and $M\geq \norm{A}_{L^\infty(Q)^n}+\norm{B}_{L^\infty(Q)^n}+\norm{q}_{L^\infty(0,T;L^{p}(\Omega))}$, when $q\in L^\infty(0,T;L^{p}(\Omega))$  with $p>2n/3$ and $M\geq \norm{A}_{L^\infty(Q)^n}+\norm{B}_{L^\infty(Q)^n}+\norm{q}_{L^\infty(0,T;L^{\frac{2n}{3}}(\Omega))}$ when $q\in \mathcal C([0,T];L^{\frac{2n}{3}}(\Omega))$.\end{prop}

\subsection{Remainder term}

In this subsection we will complete the construction of exponentially growing solutions $u_1\in L^2(0,T;H^1(\Omega))$ of the equation \eqref{Gsol1} and exponentially decaying solutions $u_2\in L^2(0,T;H^1(\Omega))$ of the equation \eqref{Gsol2} taking the form \eqref{GO12}.
We state these results in the following way.
\begin{prop}\label{p4} There exists $\rho_3>\rho_2$ such that for $\rho>\rho_3$ we can find a solution $u_1\in L^2(0,T;H^1(\Omega))$ of \eqref{Gsol1}  taking the form \eqref{GO12} 
with $w_{1,\rho}\in H^1(0,T;H^{-1}(\Omega))\cap L^2(0,T;H^1(\Omega))$ satisfying 
\bel{CGO13}\lim_{\rho\to+\infty}\rho^{-1}(\norm{w_{1,\rho}}_{L^2(0,T;H^1(\Omega))}+\rho \norm{w_{1,\rho}}_{L^2(Q)})=0,\ee
with $C$ depending on $\Omega$, $T$ and $M\geq \norm{A_1}_{L^\infty(Q)^n}+\norm{B_1}_{L^\infty(Q)^n}+\norm{q_1}_{L^\infty(0,T;L^{p}(\Omega))}$, when $q_1\in L^\infty(0,T;L^{p}(\Omega))$  with $p>2n/3$ and $M\geq \norm{A_1}_{L^\infty(Q)^n}+\norm{B_1}_{L^\infty(Q)^n}+\norm{q_1}_{L^\infty(0,T;L^{\frac{2n}{3}}(\Omega))}$ when $q_1\in \mathcal C([0,T];L^{\frac{2n}{3}}(\Omega))$.\end{prop} 
\begin{prop}\label{p5} There exists $\rho_4>\rho_3$ such that for $\rho>\rho_4$ we can find a solution  $u_2\in L^2(0,T;H^1(\Omega))$ of \eqref{Gsol2} taking the form \eqref{GO12}
with $w_{2,\rho}\in H^1(0,T;H^{-1}(\Omega))\cap  L^2(0,T;H^1(\Omega))$ satisfying  
\bel{CGO14}\lim_{\rho\to+\infty}\rho^{-1}(\norm{w_{2,\rho}}_{L^2(0,T;H^1(\Omega))}+\rho \norm{w_{2,\rho}}_{L^2(Q)})=0,\ee
with $C$ depending on $\Omega$, $T$, $M\geq \norm{A_2}_{L^\infty(Q)^n}+\norm{B_2}_{L^\infty(Q)^n}+\norm{q_2}_{L^\infty(0,T;L^{p}(\Omega))}$, when $q_2\in L^\infty(0,T;L^{p}(\Omega))$ with $p>2n/3$ and $M\geq \norm{A_2}_{L^\infty(Q)^n}+\norm{B_2}_{L^\infty(Q)^n}+\norm{q_2}_{L^\infty(0,T;L^{\frac{2n}{3}}(\Omega))}$ when $q_2\in \mathcal C([0,T];L^{\frac{2n}{3}}(\Omega))$.\end{prop} 

The proof of these two propositions  being similar, we will only consider the one of Proposition \ref{p4}. 

\ \\
 \textbf{Proof of Proposition \ref{p4}.} Note first that the condition $L_{A_1,B_1}u_1+q_1u_1=0$ is fulfilled if and only if $w_{1,\rho}$
solves
$$P_{A_1,B_1,q_1,+}w_{1,\rho}=-P_{A_1,B_1,q_1,+}b_{1,\rho}=\rho(2\omega\cdot\nabla_xb_{1,\rho}-A_1\cdot\omega b_{1,\rho})-(L_{A_1}+\nabla_x\cdot B_1+q_1)b_{1,\rho}.$$
Therefore, fixing $\phi_1\in\mathcal C^\infty_0(\R^n)$, such that $\phi_1=1$ on $\overline{\Omega}$, and
$$F_\rho(x,t)=\phi_1(x)[\rho(2\omega\cdot\nabla_xb_{1,\rho}-A_1\cdot\omega b_{1,\rho})(x,t)-L_{A_1,B_1,q_1}b_{1,\rho}(x,t)]$$
we can consider $w_{1,\rho}$ as a solution of 
\bel{CGO12}P_{A_1,q,+}w_{1,\rho}(x,t)=F_\rho(x,t),\quad (x,t)\in Q.\ee
In the expression of $F_\rho$, we assume that $A_1$, $B_1$ and $q_1$ are extended by zero to a function of $\R^n\times(0,T)$.
Let us first show that,  we have
\bel{p4c}\lim_{\rho\to+\infty}\norm{F_\rho}_{L^2(0,T;H^{-1}_\rho(\R^n))}=0.\ee
For this purpose, note first that, applying \eqref{cond31} and fixing $\tilde{Q}=\tilde{\Omega}\times(0,T)$ with $\tilde{\Omega}$ a bounded open set of $\R^n$ such that supp$(\phi)\subset \tilde{\Omega}$, we find
\bel{p4d}\begin{aligned}&\norm{F_\rho}_{L^2(0,T;H^{-1}_\rho(\R^n))}\\
&\leq  \norm{2\omega\cdot\nabla_xb_{1,\rho}-A_1\cdot\omega b_{1,\rho}}_{L^2(\tilde{Q})}+\rho^{-1}\norm{L_{A_1}b_{1,\rho}}_{L^2(\tilde{Q})}+\norm{(\nabla_x\cdot B_1)b_{1,\rho}}_{L^2(0,T;H^{-1}_\rho(\R^n))}+\norm{q_1b_{1,\rho}}_{L^2(0,T;H^{-1}_\rho(\R^n))}\\
&\leq  \norm{2\omega\cdot\nabla_xb_{1,\rho}-A_1\cdot\omega b_{1,\rho}}_{L^2(\tilde{Q})}+C\rho^{-\frac{1}{3}}+\norm{(\nabla_x\cdot B_1)b_{1,\rho}}_{L^2(0,T;H^{-1}_\rho(\R^n))}+\norm{q_1b_{1,\rho}}_{L^2(0,T;H^{-1}_\rho(\R^n))},\end{aligned}\ee
with $C>0$ independent of $ \rho$.
Let us first consider the third term on the right hand side of this inequality. We fix $B_{1,\rho}$  given by
$$B_{1,\rho}(x,t):=\int_{\R^{1+n}}\chi_\rho(x-y,t-s)B_1(y,s)dsdy$$ 
with $B_1$ extended by zero to a function defined on $\R^{1+n}$. Then,  for any $\psi_1\in L^2(0,T;\mathcal C^\infty_0(\R^n))$, we obtain
\bel{p4e}\begin{aligned}&\abs{\left\langle (\nabla_x\cdot B_1)b_{1,\rho},\psi_1 \right\rangle_{L^2(0,T;H^{-1}_\rho(\R^n)),L^2(0,T;H^{1}_\rho(\R^n))}}\\
&\leq\abs{\left\langle ( B_1\cdot \nabla_x b_{1,\rho},\psi_1 \right\rangle_{L^2(Q)}}+\abs{\left\langle  b_{1,\rho},(B_1-B_{1,\rho})\cdot\nabla_x \psi_1 \right\rangle_{L^2(\R^n\times(0,T))}}+\abs{\left\langle  b_{1,\rho},B_{1,\rho}\cdot\nabla_x \psi_1 \right\rangle_{L^2(\R^n\times(0,T))}}\\
&\leq\abs{\left\langle ( B_1\cdot \nabla_x b_{1,\rho},\psi_1 \right\rangle_{L^2(Q)}}+\abs{\left\langle  b_{1,\rho},(B_1-B_{1,\rho})\cdot\nabla_x \psi_1 \right\rangle_{L^2(\R^n\times(0,T))}}+\abs{\left\langle  \nabla_x\cdot( b_{1,\rho}B_{1,\rho}), \psi_1 \right\rangle_{L^2(\R^n\times(0,T))}}\end{aligned}\ee
For the first term on the right hand side of this inequality, applying \eqref{cond31}, we find
\bel{p4f}\abs{\left\langle ( B_1\cdot \nabla_x b_{1,\rho},\psi_1 \right\rangle_{L^2(Q)}}\leq C\norm{B_1}_{L^\infty(Q)}\rho^{{1\over3}}\norm{\psi}_{L^2(Q)}\leq C\rho^{-{2\over3}}\norm{\psi}_{L^2(0,T;H^1_\rho(\R^n))}\ee
with $C$ independent of $\rho$. For the second term on the right hand side of \eqref{p4e}, we obtain

\bel{p4g}\begin{aligned}\abs{\left\langle  b_{1,\rho},(B_1-B_{1,\rho})\cdot\nabla_x \psi_1 \right\rangle_{L^2(\R^n\times(0,T))}}&\leq \norm{b_{1,\rho}}_{L^\infty(\R^n\times(0,T))}\norm{(B_1-B_{1,\rho})}_{L^2(\R^{1+n})}\norm{\psi_1}_{L^2(0,T;H^1_\rho(\R^n))}\\
\ &\leq C\norm{(B_1-B_{1,\rho})}_{L^2(\R^{1+n})}\norm{\psi_1}_{L^2(0,T;H^1_\rho(\R^n))}\end{aligned}\ee
For the last term on the right hand side of \eqref{p4e}, we get
$$\begin{aligned}\abs{\left\langle  \nabla_x\cdot( b_{1,\rho}B_{1,\rho}), \psi_1 \right\rangle_{L^2(\R^n\times(0,T))}}&\leq \norm{b_{1,\rho}}_{L^\infty(0,T;W^{1,\infty}(\R^n)}\norm{B_\rho}_{L^2(0,T;H^1(\R^n))}\norm{\psi_1}_{L^2(\R^n\times(0,T))}\\
\ &\leq C\rho^{-\frac{1}{3}}\norm{\psi_1}_{L^2(0,T;H^1_\rho(\R^n))}.\end{aligned}$$
Combining this estimate with \eqref{p4d}-\eqref{p4g}, we obtain
$$\begin{aligned}&\abs{\left\langle (\nabla_x\cdot B_1)b_{1,\rho},\psi_1 \right\rangle_{L^2(0,T;H^{-1}_\rho(\R^n)),L^2(0,T;H^{1}_\rho(\R^n))}}\\
&\leq C[\rho^{-\frac{1}{3}}+\norm{(B_1-B_{1,\rho})}_{L^2(\R^{1+n})}]\norm{\psi_1}_{L^2(0,T;H^1_\rho(\R^n))}\end{aligned}$$
and, using the fact that
$$\lim_{\rho\to+\infty}\norm{(B_1-B_{1,\rho})}_{L^2(\R^{1+n})}=0,$$ 
we obtain
\bel{p4h} \lim_{\rho\to+\infty}\norm{(\nabla_x\cdot B_1)b_{1,\rho}}_{L^2(0,T;H^{-1}_\rho(\R^n))}=0.\ee
For the last term on the right hand side of \eqref{p4d}, fixing $\psi_1\in L^2(0,T;\mathcal C^\infty_0(\R^n))$, we find
$$\begin{aligned}&\abs{\left\langle q_1b_{1,\rho},\psi_1\right\rangle_{L^2(0,T;H^{-1}_\rho(\R^n)),L^2(0,T;H^{1}_\rho(\R^n))}}\\
&\leq\norm{b_{1,\rho}}_{L^\infty(Q)}\left(\int_Q|q_1||\psi_1|dxdt\right)\leq C\left(\int_Q|q_1||\psi_1|dxdt\right).\end{aligned}$$
For $n=2$, we find
$$\int_Q|q_1||\psi_1|dxdt\leq C\norm{q_1}_{L^\infty(0,T;L^{\frac{4}{3}}(\Omega))}\norm{\psi_1}_{L^2(0,T;L^4(\R^n))}.$$
Applying the Sobolev embedding theorem, we get
$$\begin{aligned}\norm{\psi_1}_{L^2(0,T;L^4(\R^n)}\leq C\norm{\psi_1}_{L^2(0,T;H^{\frac{1}{2}}(\R^n))}&\leq C\norm{\psi_1}_{L^2(0,T;H^{1}(\R^n))}^{\frac{1}{2}}\norm{\psi_1}_{L^2(0,T;L^2(\R^n))}^{\frac{1}{2}}\\
\ &\leq C\rho^{-\frac{1}{2}}\norm{\psi_1}_{L^2(0,T;H^{1}_\rho(\R^n))}.\end{aligned}$$
It follows, 
\bel{p4i}\int_Q|q_1||\psi_1|dxdt\leq C\rho^{-\frac{1}{2}}\norm{q_1}_{L^\infty(0,T;L^{\frac{2n}{3}}(\Omega))}\norm{\psi_1}_{L^2(0,T;H^{1}_\rho(\R^n))}.\ee
In the same way, for $n\geq3$, we have
$$\int_Q|q_1||\psi_1|dxdt\leq C\norm{q_1}_{L^2(Q)}\norm{\psi_1}_{L^2(0,T;L^2(\R^n)}\leq C\rho^{-1}\norm{q_1}_{L^\infty(0,T;L^{\frac{2n}{3}}(\Omega))}\norm{\psi_1}_{L^2(0,T;H^{1}_\rho(\R^n))}.$$
Combining these estimates with \eqref{p4h}-\eqref{p4i}, we obtain
\bel{p4j}\lim_{\rho\to+\infty}\norm{q_1b_{1,\rho}}_{L^2(0,T;H^{-1}_\rho(\R^n))}=0.\ee
Putting conditions \eqref{cond5}, \eqref{cond6}, \eqref{p4d}, \eqref{p4h} and \eqref{p4j} together, we deduce \eqref{p4c}.

We will now apply estimate \eqref{car2} to build a solution $w_{1,\rho}\in L^2(0,T;H^1(\Omega))\cap H^1(0,T;H^{-1}(\Omega))$ to \eqref{CGO12} satisfying $w_{1,\rho}(0,\cdot)=0$ and \eqref{CGO13}. We fix $\tilde{\Omega}$ a smooth bounded open set  of $\R^n$ such that $\overline{\Omega}\subset \tilde{\Omega}$.
Applying the Carleman estimate \eqref{car2},  we define the linear form $\mathcal K_\rho$ on $\{P_{-A_1,B_1-A_1,q_1,-}z:\ z\in\mathcal C^\infty([0,T];\mathcal C^\infty_0(\tilde{\Omega})),\ z_{|\tilde{\Omega}^T}=0\}$, considered as a subspace of $L^2(0,T;H^{-1}_\rho(\R^{n}))$ by
\[\mathcal K_\rho(P_{-A_1,B_1-A_1,q_1,-}z)=\left\langle F_\rho,z\right\rangle_{L^2(0,T;H^{-1}_\rho(\R^n)),L^2(0,T;H^{1}_\rho(\R^n))},\quad z\in\mathcal C^\infty([0,T];\mathcal C^\infty_0(\tilde{\Omega})),\ z_{|\tilde{\Omega}^T}=0.\]
Then, \eqref{car2} implies that, for all $z\in\mathcal C^\infty([0,T];\mathcal C^\infty_0(\tilde{\Omega}))$ satisfying $z_{|\tilde{\Omega}^T}=0$, we have
\[\begin{aligned}|\mathcal K_\rho(P_{-A_1,B_1-A_1,q_1,-}z)|&\leq \rho\norm{F_\rho}_{L^2(0,T;H^{-1}_\rho(\R^n))}(\rho^{-1}\norm{z}_{L^2(0,T;H^{1}_\rho(\R^n))})\\
\ &\leq C\rho\norm{F_\rho}_{L^2(0,T;H^{-1}_\rho(\R^n))}\norm{P_{-A_1,B_1-A_1,q_1,-}z}_{L^2(0,T;H^{-1}_\rho(\R^{n}))}\end{aligned}.\]
Thus, by the Hahn Banach theorem we can extend $\mathcal K_\rho$ to a continuous linear form on $L^2(0,T;H^{-1}_\rho(\R^{n}))$ still denoted by $\mathcal K_\rho$ and satisfying $\norm{\mathcal K_\rho}\leq C\rho\norm{F_\rho}_{L^2(0,T;H^{-1}(\R^n))}$. Therefore, there exists $w_{1,\rho}\in L^2(0,T;H^{1}_\rho(\R^{n}))$ such that 
\[\left\langle h,w_{1,\rho}\right\rangle_{L^2(0,T;H^{-1}_\rho(\R^{n})),L^2(0,T;H^{1}_\rho(\R^{n}))}=\mathcal K_\rho(h),\quad h\in L^2(0,T;H^{-1}_\rho(\R^{n})).\]
Choosing $h=P_{-A_1,B_1-A_1,q_1,-}z$ with $z\in \mathcal C^\infty_0(Q)$ proves that $w_{1,\rho}$ satisfies $P_{A_1,B_1,q_1,+}w_{1,\rho}=F_\rho$ in $Q$. In particular, we deduce that $w_{1,\rho}\in H^1(0,T;H^{-1}(\Omega))\cap L^2(0,T;H^1(\Omega))$. Moreover, fixing $h=P_{-A_1,B_1-A_1,q_1,-}z$ with $z\in\mathcal C^\infty([0,T];\mathcal C^\infty_0(\tilde{\Omega}))$, $z_{|\tilde{\Omega}^T}=0$ and allowing $z_{|\tilde{\Omega}^0}$ to be arbitrary proves that $w_{1,\rho}=0$ on $\Omega^0$. In addition, applying \eqref{p4c}, we get $$\limsup_{\rho\to+\infty}\rho^{-1}\norm{w_{1,\rho}}_{L^2(0,T;H^{1}_\rho(\R^{n}))}\leq \limsup_{\rho\to+\infty}\rho^{-1}\norm{\mathcal K_\rho}\leq C\limsup_{\rho\to+\infty}\norm{F_\rho}_{L^2(0,T;H^{-1}_\rho(\R^n))}=0.$$ 
Therefore,  $w_{1,\rho}$ fulfills \eqref{CGO12},  $w_{1,\rho}(\cdot,0)=0$ and \eqref{CGO13}. This completes the proof of the proposition.\qed

\section{Recovery  from the  DN map}
In this section we will prove Theorem \ref{t5}. For this purpose, applying Proposition \ref{p4} and \ref{p5}, we fix a solution $u_1\in H^1(0,T;H^{-1}(\Omega))\cap L^2(0,T;H^1(\Omega))$ of \eqref{Gsol1} of the form \eqref{GO12} and a solution $u_2\in H^1(0,T;H^{-1}(\Omega))\cap L^2(0,T;H^1(\Omega))$ of \eqref{Gsol2} given by \eqref{GO12}, with $w_{j,\rho}$, $j=1,2$, satisfying the decay property \eqref{CGO11}

\subsection{Recovery of the first order coefficient}

According to \eqref{t5a} and \eqref{p3a}, we have
$$\int_QA\cdot\nabla_xu_1u_2dxdt-\int_QB\cdot\nabla_x(u_1u_2)dxdt+\int_Qqu_1u_2dxdt=0,$$
with $A=A_1-A_2$, $B=B_1-B_2$ and $q=q_1-q_2$.
On the other hand, we find
\bel{t1f}\int_QA\cdot\nabla_xu_1u_2dxdt-\int_QB\cdot\nabla_x(u_1u_2)dxdt+\int_Qqu_1u_2dxdt=\rho\int_Q(A\cdot\omega)b_{1,\rho}b_{2,\rho}dxdt+\int_QZ_\rho(x,t)dxdt\ee
with $$Z_\rho=A\cdot\nabla_xb_{1,\rho}(b_{2,\rho}+w_{2,\rho})+B\cdot\nabla_x[(b_{1,\rho}+w_{1,\rho})(b_{2,\rho}+w_{2,\rho})]+A\cdot\nabla_xw_{1,\rho}(b_{2,\rho}+w_{2,\rho})+q(b_{1,\rho}+w_{1,\rho})(b_{2,\rho}+w_{2,\rho}).$$
In view of \eqref{CGO11} and \eqref{cond31}-\eqref{cond33}, we have
\bel{t1g}\lim_{\rho\to+\infty}\rho^{-1}\abs{\int_QZ_\rho(x,t)dxdt}=0.\ee
Moreover, we deduce that
$$\begin{aligned}\int_Q(A\cdot\omega)b_{1,\rho}b_{2,\rho}dxdt=&\int_\R\int_{\R^n}((A-A_\rho)\cdot\omega)b_{1,\rho}b_{2,\rho}dxdt-\int_0^{+\infty}\int_{\R^n}(A_\rho\cdot\omega)e^{-\rho^{1\over3}t}b_{2,\rho}dxdt\\
\ &-\int_{-\infty}^T\int_{\R^n}(A_\rho\cdot\omega)b_{1,\rho}e^{-\rho^{1\over3}(T-t)}dxdt\\
\ &+\int_\R\int_{\R^n}e^{-i(t\tau+x\cdot\xi)}A_\rho(x,t)\cdot\omega\exp\left(-{\int_0^{+\infty} A_{\rho}(x+s\omega,t)\cdot\omega ds\over2}\right)dxdt.\end{aligned}$$
Combining this with  \eqref{a1a} and applying Lebesgue dominate convergence theorem, we deduce that
$$\limsup_{\rho\to+\infty}\abs{\int_\R\int_{\R^n}e^{-i(t\tau+x\cdot\xi)}A_\rho(x,t)\cdot\omega\exp\left(-{\int_0^{+\infty} A_{\rho}(x+s\omega,t)\cdot\omega ds\over2}\right)dxdt}=\limsup_{\rho\to+\infty}\abs{\int_Q(A\cdot\omega)b_{1,\rho}b_{2,\rho}dxdt}.$$
In addition, applying \eqref{t1f}-\eqref{t1g}, we obtain
\bel{t1j}\limsup_{\rho\to+\infty}\abs{\int_\R\int_{\R^n}e^{-i(t\tau+x\cdot\xi)}(A_\rho(x,t)\cdot\omega)\exp\left(-{\int_0^{+\infty} A_{\rho}(x+s\omega,t)\cdot\omega ds\over2}\right)dxdt}=0\ee
On the other hand, decomposing $\R^{n}$ into the direct sum $\R^{n}=\R\omega\oplus\omega^{\bot}$ and applying the Fubini's theorem we get
\bel{t1k}\begin{aligned}&\int_\R\int_{\R^n}e^{-i(t\tau+x\cdot\xi)}(A_\rho(x,t)\cdot\omega)\exp\left(-{\int_0^{+\infty} A_{\rho}(x+s\omega,t)\cdot\omega ds\over2}\right)dxdt\\
&=\int_\R\int_{\omega^{\bot}}\left[\int_\R (A_{\rho}(y+s_2\omega,t)\cdot\omega)\exp\left(-{\int_{s_2}^{+\infty}A_{\rho}(y+s_1\omega,t)\cdot\omega ds_1\over 2}\right)ds_2\right]e^{-it\tau-i\xi\cdot y}dydt.\end{aligned}\ee
Moreover, for all $t\in(0,T)$ and all $y\in \omega^{\bot}$ we have
$$\begin{aligned}\int_\R A_{\rho}(y+s_2\omega,t)\cdot\omega\exp\left(-{\int_{s_2}^{+\infty}A_{\rho}(y+s_1\omega,t)\cdot\omega ds_1\over 2}\right)ds_2&=2\int_\R \pd_{s_2}\exp\left(-{\int_{s_2}^{+\infty}A_{\rho}(y+s_1\omega,t)\cdot\omega ds_1\over 2}\right)ds_2\\
\ &=2\left(1-\exp\left(-{\int_\R A_{\rho}(y+s_1\omega,t)\cdot\omega ds_1\over 2}\right)\right).\end{aligned}$$
Combining this with \eqref{t1k}, we find
\bel{t1l}\begin{array}{l}\int_\R\int_{\R^n}e^{-i(t\tau+x\cdot\xi)}A_\rho(x,t)\exp\left(-{\int_0^{+\infty} A_{\rho}(x+s\omega,t)\cdot\omega ds\over2}\right)dxdt\\
\ \\
=2\int_\R\int_{\omega^{\bot}}\left(1-\exp\left(-{\int_\R A_{\rho}(y+s_1\omega,t)\cdot\omega ds_1\over 2}\right)\right)e^{-i\xi\cdot y-it\tau}dydt.\end{array}\ee
Now let us introduce the Fourier transform $\mathcal F_{\R\times\omega^{\bot}}$ on $\R\times\omega^{\bot}$ defined by
\[\mathcal F_{\R\times\omega^{\bot}}f(\xi,\tau)=(2\pi)^{-{n\over2}}\int_\R\int_{\omega^{\bot}}f(y,t)e^{-it\tau-iy\cdot\xi}dydt,\quad f\in L^1(\omega^{\bot}\times\R), \ \ \tau\in\R,\ \ \xi\in\omega^{\bot}.\]
We fix
$$G_\rho:\omega^{\bot}\times\R\ni (y,t)\mapsto\left(1-\exp\left(-{\int_\R A_\rho(t,y+s_1\omega)\cdot\omega ds_1\over 2}\right)\right)$$
and we remark that for
$$R=\sup_{x\in\overline{\Omega}}|x|$$
we have supp$(G_\rho)\subset \{x\in\omega^\bot:\ |x|\leq R+1\}\times [-1,T+1]$. We fix also
$$G:\omega^{\bot}\times\R\ni (y,t)\mapsto\left(1-\exp\left(-{\int_\R A(y+s_1\omega,t)\cdot\omega ds_1\over 2}\right)\right).$$
Using this and applying   the mean value theorem and \eqref{a1b}, for a.e $(x',t)\in \omega^\bot\times\R$, we obtain
$$\begin{aligned}&|G_\rho(x',t)-G(x',t)|\\
&\leq \exp\left(\abs{\int_\R A(x'+s_1\omega,t)\cdot\omega ds_1}+\abs{\int_\R A_\rho(x'+s_1\omega,t)\cdot\omega ds}\right)\abs{\int_\R A_\rho(x'+s_1\omega,t)\cdot\omega ds_1-\int_\R A(x'+s_1\omega,t)\cdot\omega ds_1}\\
\ &\leq \exp\left[2(R+1)(\norm{ A}_{L^\infty(\R^{1+n})^n}+\norm{ A_\rho}_{L^\infty(\R^{1+n})^n})\right]\int_\R|A(x'+s_1\omega,t)-A_\rho(x'+s_1\omega,t)|ds_1\\
\ &\leq C\left(\int_\R|A(x'+s_1\omega,t)-A_\rho(x'+s_1\omega,t)|ds_1\right),\end{aligned}$$
with $C>0$ independent of $\rho$. Thus, integrating this expression with respect to $x'\in\omega^\bot$ and $t\in\R$ and applying the Fubini theorem, we obtain
$$\int_\R\int_{\omega^\bot}|G_\rho(x',t)-G(x',t)|dx'dt\leq C\int_\R\int_{\omega^\bot}\int_\R|A(y+s_1\omega,t)-A_\rho(y+s_1\omega,t)|ds_1dx'dt\leq C'\norm{A-A_\rho}_{L^1(\R^{1+n})}.$$
Then, applying \eqref{a1a} we get
$$\lim_{\rho\to+\infty}\norm{G-G_\rho}_{L^1(\R\times\omega^\bot)}=0.$$
Combining this with \eqref{t1j}-\eqref{t1l}, we find
$$\begin{aligned}2\int_\R\int_{\omega^\bot} G(x',t)e^{-it\tau-ix'\cdot\xi}dx'dt&=\lim_{\rho\to+\infty}2\int_\R\int_{\omega^\bot} G_\rho(x',t)e^{-it\tau-ix'\cdot\xi}dx'dt\\
\ &=\lim_{\rho\to+\infty}\int_\R\int_{\R^n}e^{-i(t\tau+x\cdot\xi)}A_\rho(x,t)\exp\left(-{\int_0^{+\infty} A_{\rho}(x+s\omega,t)\cdot\omega ds\over2}\right)dxdt\\
\ &=0.\end{aligned}$$
Allowing $\xi\in\omega^\bot$ and $\tau\in\R$ to be arbitrary, we deduce that $\mathcal F_{\R\times\omega^{\bot}}G=0$. Using the injectivity of $\mathcal F_{\R\times\omega^{\bot}}$, for a.e $(x',t)\in \omega^\bot\times\R$, we deduce that
$$\exp\left(-{\int_\R A(y+s_1\omega,t)\cdot\omega ds_1\over 2}\right)=1$$
and, using the fact that $A$ takes value in $\R^n$, we obtain
\bel{t5g}\int_\R A(x'+s_1\omega,t)\cdot\omega ds_1=0.\ee
We recall that, here $\omega$ can be arbitrary chosen. 

Now fixing $(\xi,\tau)\in\R^n\times\R$ with $\xi\neq0$, we deduce from \eqref{t5g}, that, for $\omega\in \xi^\bot\cap\mathbb S^{n-1}$, we have
$$\int_\R\int_{\omega^\bot}\int_\R A(x'+s_1\omega,t)\cdot\omega e^{-it\tau-ix'\cdot\xi}ds_1dx'dt=0.$$
Applying Fubini theorem and a change of variable, we get
$$\int_{\R^{1+n}}A(x,t)\cdot\omega e^{-it\tau-ix\cdot\xi}dxdt =\int_\R\int_{\omega^\bot}\int_\R A(x'+s_1\omega,t)\cdot\omega e^{-it\tau-ix'\cdot\xi}ds_1dx'dt=0.$$
 This proves that
\bel{ttt5b}\mathcal F(A)(\xi,\tau)\cdot\omega=0,\quad \tau\in\R,\ \xi\in\R^n\setminus\{0\},\ \omega\in\xi^\bot\cap\mathbb S^{n-1}.\ee
Let $j,k\in\{1,\ldots,n\}$ be such that $j\neq k$ and consider the set $\mathcal I_j:=\{\xi=(\xi_1,\ldots,\xi_n)\in\R^n:\ \xi_j\neq0\}$. Let $\xi\in \mathcal I_j$, $\tau\in\R$ and let 
$$\omega={\xi_ke_j-\xi_j e_k\over\sqrt{\xi_j^2+\xi_k^2}},$$
with $e_j=(0,\ldots,0,\underbrace{1}_{\textrm{position\ } j},0,\ldots0)$, $e_k=(0,\ldots,0,\underbrace{1}_{\textrm{position } k},0,\ldots0)$. Then, for $A=(a_1,\ldots,a_n)$ we have
$$\mathcal F(\partial_{x_k}a_j-\partial_{x_j}a_k)(\xi,\tau)=  i\sqrt{\xi_j^2+\xi_k^2}\mathcal F(A)(\xi,\tau)\cdot\omega.$$
Thus, condition \eqref{ttt5b} implies that
$$\mathcal F(\partial_{x_k}a_j-\partial_{x_j}a_k)(\xi,\tau)=0,\quad \xi\in \mathcal I_j,\ \tau\in\R.$$
In the same way, we prove that
$$\mathcal F(\partial_{x_k}a_j-\partial_{x_j}a_k)(\xi,\tau)=0,\quad \xi\in \mathcal I_k,\ \tau\in\R$$
and it is clear that
$$\mathcal F(\partial_{x_k}a_j-\partial_{x_j}a_k)(\xi,\tau)=i(\xi_k\mathcal F(a_j)(\xi,\tau)-\xi_j\mathcal F(a_k)(\xi,\tau))=0,\quad \xi\in \R^n\setminus(\mathcal I_k\cup\mathcal I_j),\ \tau\in\R.$$
Therefore, we have $\mathcal F(\partial_{x_k}a_j-\partial_{x_j}a_k)=0$ which implies $\partial_{x_k}a_j-\partial_{x_j}a_k=0$ and by the same way that $dA=0$. This proves \eqref{t5b}.

\subsection{Recovery of the zero order coefficients}
In this subsection we assume that \eqref{t5b}-\eqref{tttt5b} are fulfilled. Our goal is to prove that  \eqref{t5a}  implies \eqref{t5c}. In this subsection, we denote by $A$, $B$ and $q$ the functions $A_1-A_2$, $B_1-B_2$ and $q_1-q_2$ extended by zeo to $\R^{1+n}$. We start, with the following intermediate result.
\begin{lem} \label{ll4} Let $\mathcal A\in  L^\infty(\R^{1+n})^n$ be compactly supported and  assume that $d\mathcal A=0$ in the sense of distributions taking value in 2-forms. Then, for
\bel{l4aa}\phi(x,t):=-\int_0^1\frac{\mathcal A(sx,t)\cdot x}{2}ds,\quad (x,t)\in\R^n\times\R,\ee
 we have $\phi \in L^\infty(\R_t; W^{1,\infty}(\R_x^n))$ and $\nabla_x \phi=-\frac{\mathcal A}{2}$.
\end{lem}

We refer to \cite[Lemma 4.2]{Ki5} for the proof of this result. From now on we fix $\phi\in L^\infty_{loc}(\R^{n+1})$ given by \eqref{l4aa}, with $\mathcal A=A$, and applying Lemma \ref{ll4} we deduce that $\nabla_x\phi=-\frac{A}{2}$ and  $\phi \in L^\infty(\R_t; W^{1,\infty}(\R_x^n))$.	
  Moreover, since $A\in W^{1,\infty}(0,T; L^{p_1}(\Omega))^n$,  by the Sobolev embedding theorem, we deduce that for any open bounded set $\tilde{\Omega}\subset \R^n$ we have
$$\phi\in W^{1,\infty}(0,T; W^{1,p_1}(\tilde{\Omega}))\subset W^{1,\infty}(0,T; L^{\infty}(\tilde{\Omega})).$$
 Thus, we have  $\phi\in L^\infty(0,T;W^{1,\infty}(\R^n))\cap W^{1,\infty}(0,T; L^{\infty}_{loc}(\R^n))$. Since $\R^n\setminus\Omega$ is connected and $A=0$ on $[\R^n\setminus\Omega]\times(0,T)$, there exists a function $h\in W^{1,\infty}(0,T)$ such that 
$$\phi(x,t)=h(t),\quad (x,t)\in (\R^n\setminus\Omega)\times(0,T).$$
Therefore, by replacing $\phi(x,t)$ with $\phi(x,t)-h(t)$, we may assume without lost of generality that $\phi=0$ on $(\R^n\setminus\Omega)\times(0,T)$. In particular, we have $\phi_{|\Sigma}=0$.
Therefore, we can apply the gauge invariance of the DN map  to get
$$\Lambda_{A_1,B_1,q_1}=\Lambda_{A_1+2\nabla_x\phi,B_1+\nabla_x\phi,q_1-\pd_t\phi-|\nabla_x\phi|^2-A_1\cdot\nabla_x\phi}=\Lambda_{A_2,B_1+\nabla_x\phi,q_1-\pd_t\phi+\Delta_x\phi-|\nabla_x\phi|^2-A_1\cdot\nabla_x\phi}.$$
Then, condition \eqref{t5a} implies that
\bel{DN1}\Lambda_{A_2,B_1+\nabla_x\phi,q_1-\pd_t\phi-|\nabla_x\phi|^2-A_1\cdot\nabla_x\phi}=\Lambda_{A_2,B_2,q_2}.\ee
 We will prove that this condition implies
\bel{tt5e} \nabla_x\cdot B_2+q_2=\nabla_x\cdot( B_1+\nabla_x\phi)+q_1-\pd_t\phi-|\nabla_x\phi|^2-A_1\cdot\nabla_x\phi.\ee
For this purpose we fix a solution $u_1\in L^2(0,T;H^1(\Omega))$ of \eqref{Gsol1}, with $(A_1, B_1, q_1)$ replaced by $( A_2,B_1+\nabla_x\phi,q_1-\pd_t\phi-|\nabla_x\phi|^2-A_1\cdot\nabla_x\phi)$, of the form \eqref{GO12} and a solution $u_2\in L^2(0,T;H^1(\Omega))$ of \eqref{Gsol2} given by \eqref{GO12}, with $w_{j,\rho}$, $j=1,2$, satisfying the decay property \eqref{CGO13}-\eqref{CGO14}. 
  In light of \eqref{p3a}, we have
\bel{tt5f}\int_Q(B_1+\nabla_x\phi-B_2)\cdot\nabla_x(u_1u_2)dxdt+\int_Q(q_1-\pd_t\phi-|\nabla_x\phi|^2-A_1\cdot\nabla_x\phi-q_2)u_1u_2dxdt=0.\ee
For the first term on left hand side of \eqref{tt5f}, applying \eqref{ttt5a}-\eqref{tttt5b} and the Green formula, we get 
$$\begin{aligned}&\int_Q(B_1+\nabla_x\phi-B_2)\cdot\nabla_x(u_1u_2)dxdt\\
&=\int_Q(B_1-\frac{A}{2}-B_2)\cdot\nabla_x(u_1u_2)dxdt\\
&=-\int_Q\nabla_x\cdot(B_1-\frac{A}{2}-B_2)u_1u_2dxdt+\left\langle [(B_1-B_2)\cdot\nu-\frac{(A_1-A_2)\cdot\nu}{2}]u_1,u_2\right\rangle_{L^2(0,T;H^{-\frac{1}{2}}(\partial\Omega)),L^2(0,T;H^{\frac{1}{2}}(\partial\Omega))}\\
\ &=-\int_Q\nabla_x\cdot(B_1-\frac{A}{2}-B_2)b_{1,\rho}b_{2,\rho}dxdt-\int_QZ_\rho dxdt\\
\ &=\int_Q(B_1+\nabla_x\phi-B_2)\cdot \nabla_x(b_{1,\rho}b_{2,\rho})dxdt-\int_QZ_\rho dxdt,\end{aligned}$$
with $Z_\rho=\nabla_x\cdot(B-\frac{A}{2})(b_{1,\rho}w_{2,\rho}+b_{2,\rho}w_{1,\rho}+w_{1,\rho}w_{2,\rho})$.   In view of \eqref{CGO13}, it is clear that
$$\lim_{\rho\to+\infty}\int_QZ_\rho dxdt=0.$$
Moreover, one can easily check that 
$$\begin{aligned}&\int_Q(B_1+\nabla_x\phi-B_2)\cdot \nabla_x(b_{1,\rho}b_{2,\rho})dxdt\\
&=-i\left[\int_Q\left(1-e^{-\rho^{1\over3}t}\right) \left(1-e^{-\rho^{1\over3}(T-t)}\right)e^{-i(t\tau+x\cdot\xi)}(B_1+\nabla_x\phi-B_2)(x,t)dxdt\right]\cdot \xi.\end{aligned}$$
Sending $\rho\to+\infty$ and applying the Lebesgue dominate convergence  theorem, we find
$$\begin{aligned}\lim_{\rho\to+\infty}\int_Q(B+\nabla_x\phi)\cdot \nabla_x(b_{1,\rho}b_{2,\rho})dxdt&=(2\pi)^{\frac{n+1}{2}}[\mathcal F(B+\nabla_x\phi)(\xi,\tau)]\cdot(-i\xi)\\
\ &=(2\pi)^{\frac{n+1}{2}}\mathcal F[\nabla_x\cdot(B+\nabla_x\phi)](\xi,\tau)\end{aligned}$$
Therefore, we have
$$\lim_{\rho\to+\infty}\int_Q(B+\nabla_x\phi)\cdot\nabla_x(u_1u_2)dxdt=(2\pi)^{\frac{n+1}{2}}\mathcal F[\nabla_x\cdot(B+\nabla_x\phi)](\xi,\tau).$$
In the same way, we can prove that
$$\lim_{\rho\to+\infty}\int_Q(q_1-\pd_t\phi-|\nabla_x\phi|^2-A_1\cdot\nabla_x\phi-q_2)u_1u_2dxdt=(2\pi)^{\frac{n+1}{2}}\mathcal F[(q-\pd_t\phi-|\nabla_x\phi|^2-A_1\cdot\nabla_x\phi)](\xi,\tau).$$
Combining this with \eqref{tt5f},  we obtain
$$\mathcal F[\nabla_x\cdot(B+\nabla_x\phi)+q-\pd_t\phi-|\nabla_x\phi|^2-A_1\cdot\nabla_x\phi)](\xi,\tau)=0,\quad (\xi,\tau)\in\R^{1+n}.$$
This proves \eqref{tt5e} and the proof of \eqref{t5c} is completed.

\subsection{Proof of Corollary \ref{c1}}
This subsection is devoted to the proof of Corollary \ref{c1}. For this purpose, we assume that $\Lambda_{A_1,B,q}=\Lambda_{A_2,B,q}$. Then Theorem \ref{t5} implies that there exists $\phi\in W^{1,\infty}(Q))$ such that
$$\left\{\begin{array}{ll}A_2=A_1+2\nabla_x\phi,\quad &\textrm{in}\ Q,\\  \nabla_x\cdot B+q=\nabla_x\cdot( B+\nabla_x\phi)+q-\pd_t\phi-|\nabla_x\phi|^2-A_1\cdot\nabla_x\phi,\quad &\textrm{in}\ Q,\\ \phi=0,\quad &\textrm{on}\ \Sigma.\end{array}\right.$$
Thus, fixing $A_3=A_1+\nabla_x\phi\in L^\infty(Q)$, we deduce that $\phi$ satisfies
$$\left\{\begin{array}{ll} \pd_t\phi-\Delta_x\phi+A_3\cdot\nabla_x\phi=0,\quad &\textrm{in}\ Q, \\ \phi=0,\quad &\textrm{on}\ \Sigma.\end{array}\right.$$
Note that since $\phi\in L^\infty(0,T;W^{1,\infty}(\Omega))$ and $\Delta \phi= \frac{\nabla_x\cdot(A_2-A_1)}{2}\in L^\infty(Q)$ we can define $\pd_\nu \phi$ as an element of $L^\infty(0,T;H^{-\frac{1}{2}}(\pd\Omega))$. Moreover, the conditions $\phi_{|\Sigma}=0$ and $(A_2-A_1)\cdot\nu_{|\Sigma}=0$ imply that  $\phi_{|\Sigma}=\pd_\nu \phi_{|\Sigma}=0$. Thus, fixing $\mathcal O$ a set with not empty interior such that $\tilde{\Omega}=\mathcal O\cup\Omega$ is an open bounded connected set of $\R^n$ with Lipschitz boundary, we can see that $\phi$ extended by zero to $\tilde{\Omega}\times (0,T)$ solves
$$\left\{\begin{array}{ll} \pd_t\phi-\Delta_x\phi+A_3\cdot\nabla_x\phi=0,\quad &\textrm{in}\ \tilde{\Omega}\times(0,T), \\ \phi=0,\quad &\textrm{on}\ \mathcal O\times(0,T),\end{array}\right.$$
with $A_3$ extended by zero to $\tilde{\Omega}\times (0,T)$. Then the unique continuation properties for parabolic equations (e.g.  \cite[Theorem 1.1]{SS}) implies that $\phi=0$. Note that such results of unique continuation are stated for solutions of parabolic equations lying $H^1(0,T;L^2(\Omega))\cap L^2(0,T;H^2(\Omega))$, but since they follow from  Carleman estimates  like \cite[Theorem 1.2]{SS}, they can be extended to solutions lying in $H^1(Q)$ and we can apply this result to $\phi$. This proves that $A_1=A_2$ and  the proof of Corollary \ref{c1} is completed.

\subsection{Proof of Corollary \ref{c11}}
In this section we will prove   Corollary \ref{c11}. For this purpose, we first recall that $\tilde{\Lambda}_{A_j}=\Lambda_{A_j,\frac{A_j}{2},\frac{\nabla_x\cdot(A_j)}{2}}$, $j=1,2$. Therefore, Theorem \ref{t5} implies that there exists $\phi\in W^{1,\infty}(Q)$ such that
$$\left\{\begin{array}{ll}A_2=A_1+2\nabla_x\phi,\quad &\textrm{in}\ Q,\\  \nabla_x\cdot(A_2)=\nabla_x\cdot(A_1)-\pd_t\phi+\Delta_x\phi-|\nabla_x\phi|^2-A_1\cdot\nabla_x\phi,\quad &\textrm{in}\ Q,\\ \phi=0,\quad &\textrm{on}\ \Sigma.\end{array}\right.$$
Then, fixing $A_3=-A_1-\nabla_x\phi\in L^\infty(Q)$ and applying the fact that $(A_2-A_1)\cdot\nu_{|\Sigma}=0$, we deduce that $\phi$ satisfies
$$\left\{\begin{array}{ll} -\pd_t\phi-\Delta_x\phi+A_3\cdot\nabla_x\phi=0,\quad &\textrm{in}\ Q, \\ \phi=\pd_\nu \phi=0,\quad &\textrm{on}\ \Sigma.\end{array}\right.$$
Therefore, applying again the unique continuation properties for parabolic equations we deduce that $\phi=0$ and  the proof of Corollary \ref{c11} is completed.

\subsection{Proof of Corollary \ref{c3}}

In this subsection we will show Corollary \ref{c3}.  Let us first consider the following intermediate result.
\begin{lem}\label{llll1} Let $\Omega$ be a bounded open set of $\R^n$ with Lipschitz boundary. Then, for every $F\in L^2(Q)$, the problem
\bel{llll1a}\left\{\begin{array}{ll}-\partial_tv-\Delta_x v=F,\quad &\textrm{in}\ Q,\\  v(\cdot,T)=0,\quad &\textrm{in}\ \Omega,\\ v=0,\quad &\textrm{on}\ \Sigma\end{array}\right.\ee
admits a unique solution $v\in H^1(0,T;L^2(\Omega))\cap L^2(0,T;H^1(\Omega))$.\end{lem}
\begin{proof} This result is classical but we prove it for sake of completeness. Applying \cite[Theorem 4.1, Chapter 3]{LM1} we know that \eqref{llll1a} admits a unique solution  $v\in L^2(0,T;H^1(\Omega))\cap H^1(0,T;H^{-1}(\Omega))$. So the proof of the lemma will be completed if we show that $v\in H^1(0,T;L^2(\Omega))$. Let $(\lambda_n)_{n\geq1}$ be the non-decreasing sequence of eigenvalues for the operator $H=-\Delta$ with Dirichlet boundary condition and $(\phi_n)_{n\geq1}$ an associated orthonormal basis of  eigenfunctions. We will prove that actually $v\in L^2(0,T;D(H))$ which will complete the proof of the lemma. We fix $v_n(t)= \left\langle v(\cdot,t),\phi_n\right\rangle_{L^2(\Omega)}$, $F_n(t)= \left\langle F(\cdot,t),\phi_n\right\rangle_{L^2(\Omega)}$ and we remark  that $v_n$ solves 
$$\left\{\begin{array}{l}-\partial_tv_n+\lambda_n v=F_n,\\  v_n(T)=0.\end{array}\right.$$
Therefore, we have
$$v_n(t)=\int_0^{T-t}e^{-\lambda_n(T-t-s)}F_n(s)ds=(e^{-\lambda_nt}\mathds{1}_{(0,+\infty)})*(F_n\mathds{1}_{(0,T)})(T-t),$$
with $*$ the convolution product and, for any interval $I$, $\mathds{1}_I$ the characteristic function of $I$. An application of Young inequality yields
$$\norm{v_n}_{L^2(0,T)}\leq \left(\int_0^\infty e^{-\lambda_nt}dt\right)\norm{F_n}_{L^2(0,T)}\leq \frac{\norm{F_n}_{L^2(0,T)}}{\lambda_n}.$$
Thus, we have
$$\begin{aligned}\int_0^T\left(\sum_{n=1}^\infty \lambda_n^2|v_n(t)|^2\right)dt&=\sum_{n=1}^\infty \left(\int_0^T\lambda_n^2|v_n(t)|^2\right)dt\\
\ &\leq \sum_{n=1}^\infty \left(\int_0^T|F_n(t)|^2\right)dt=\int_0^T\left(\sum_{n=1}^\infty |F_n(t)|^2\right)dt=\norm{F}_{L^2(Q)}^2.\end{aligned}$$
This proves that $v=\underset{n\in\mathbb N}{\sum} v_n\phi_n\in L^2(0,T;D(H))$ and using the fact that $\partial_tv=Hv-F$ we deduce that $v\in H^1(0,T;L^2(\Omega))$.
This  completes the proof of the lemma.\end{proof}

Let us observe that for $\Omega$ a $C^{1,1}$ bounded domain, by the elliptic regularity, the result of Lemma \ref{llll1} would correspond to existence of a strong solution $v\in L^2(0,T;H^2(\Omega))\cap H^1(0,T;L^2(\Omega))$ of \eqref{llll1a}.  However, we do not want to assume such regularity for $\partial \Omega$. 

From now on, we assume that  the conditions of Corollary \ref{c3} are fulfilled and, for $A,B\in L^\infty(Q)^n$ satisfying $\nabla_x\cdot A,\nabla_x\cdot B\in L^\infty(Q)$  and $q\in L^\infty(0,T;L^{p_1}(\Omega))$, we consider the following spaces
$$S_{+,A,B,q}:=\{u\in L^2(0,T;H^1(\Omega)):\ \partial_tu-\Delta u+A\cdot\nabla_xu+\nabla_x\cdot Bu+qu=0,\  u_{|\Omega^0}=0\},$$
$$S_{-,A,B,q}:=\{u\in L^2(0,T;H^1(\Omega)):\ -\partial_tu-\Delta u-A\cdot\nabla_xu+(q+\nabla_x\cdot (B-A))u=0,\ u_{|\Omega^T}=0\},$$
$$S_{+,A,B,q,\gamma_1}:=\{u\in S_{+,A,B,q}:\  \textrm{supp}(u_{|\Sigma})\subset [0,T]\times \gamma_1\},$$
$$S_{-,A,B,q,\gamma_2}:=\{u\in S_{-,A,B,q}:\  \textrm{supp}(u_{|\Sigma})\subset [0,T]\times \gamma_2\}.$$
Fixing $Q_1:=(\Omega\setminus\overline{\Omega_*})\times(0,T)$, we can consider the following density result.

\begin{lem}\label{lll4} Assume that $\nabla_x\cdot (B), \nabla_x\cdot (A)\in L^\infty(0,T; L^{p_1}(\Omega))$. Then the space $S_{+,A,B,q,\gamma_1}$ $($resp. $S_{-,A,B,q,\gamma_2}$$)$ is dense in the space $S_{+,A,B,q}$ $($resp. $S_{-,A,B,q}$$)$ with respect to the norm $L^2(Q_1)$.\end{lem}
\begin{proof} Since the proof of these two results are similar, we prove only the density of $S_{+,A,B,q,\gamma_1}$ in $S_{+,A,B,q}$. For this purpose, we assume the contrary.  Then, an application of Hahn Banach theorem implies that there exist $h\in L^2(Q_1)$ and $u_0\in S_{+,A,B,q}$ such that
\bel{lll4a}\int_{Q_1}hudxdt=0,\quad u\in S_{+,A,B,q,\gamma_1},\ee
\bel{lll4b}\int_{Q_1}hu_0dxdt=1.\ee
Now let us extend $h$ by zero to $h\in L^2(Q)$. According to \cite[Theorem 4.1, Chapter 3]{LM1} there exists a unique solution $w\in L^2(0,T;H^1(\Omega))\cap H^1(0,T;H^{-1}(\Omega))$ to the IBVP
$$\left\{\begin{array}{ll}-\partial_tw-\Delta_x w-A\cdot\nabla_xw+(q+\nabla_x\cdot (B-A))w=h,\quad &\textrm{in}\ Q,\\  w(\cdot,T)=0,\quad &\textrm{in}\ \Omega,\\ w=0,\quad &\textrm{on}\ \Sigma.\end{array}\right.$$
Moreover, fixing $F=A\cdot\nabla_xw-(q+\nabla_x\cdot (B-A))w+h\in L^2(Q)$, we deduce that $w$ solves
$$\left\{\begin{array}{ll}-\partial_tw-\Delta_x w=F,\quad &\textrm{in}\ Q,\\  w(\cdot,T)=0,\quad &\textrm{in}\ \Omega,\\ w=0,\quad &\textrm{on}\ \Sigma\end{array}\right.$$
and from Lemma \ref{llll1}, we deduce that  $w\in H^1(Q)$. In particular, we have $\Delta w\in L^2(0,T;L^2(\Omega))$ which implies that $\pd_\nu w\in L^2(0,T;H^{-\frac{1}{2}}(\pd\Omega))$.
In view of \eqref{lll4a}, choosing $u\in S_{+,A,B,q,\gamma_1}$, we get
$$\begin{aligned} &\left\langle \partial_\nu w, u\right\rangle_{L^2(0,T;H^{-\frac{1}{2}}(\pd\Omega)),L^2(0,T;H^{\frac{1}{2}}(\pd\Omega)) }\\
&=\int_Q \Delta w udxdt+\int_Q \nabla_x w\cdot \nabla_x udxdt+\int_\Sigma (A\cdot\nu)uwd\sigma(x)dt \\
&=\left\langle \pd_tu,w\right\rangle_{L^2(0,T;H^{-1}(\Omega)),L^2(0,T;H^{1}_0(\Omega))}+\int_Q \nabla_x w\cdot \nabla_x udxdt -\int_Q [-\pd_tw-\Delta w] udxdt+\int_Q \nabla_x\cdot(uwA)dxdt \\
&=\left\langle \partial_tu-\Delta_x u+A\cdot\nabla_xu+(\nabla_x\cdot B+q)u,w\right\rangle_{L^2(0,T;H^{-1}(\Omega)),L^2(0,T;H^{1}_0(\Omega))}\\
&\ \ \ - \int_Qu[-\partial_tw-\Delta_x w-A\cdot\nabla_xw+(q+\nabla_x(B-A))w]dxdt\\
\ &=-\int_{Q_1}hudxdt=0.\end{aligned}$$
Allowing $u\in S_{+,A,B,q,\gamma_1}$ to be arbitrary, we deduce that $\partial_\nu w_{|\gamma_1\times(0,T)}=0$. Thus, fixing $\Omega_1$ a set with nonempty interior such that $\Omega_1\cap\pd\Omega\subset\gamma_1$ and $\Omega_2=\Omega_*\cup\Omega_1$ is a connected open set of $\R^n$, we have
$$\left\{\begin{array}{ll}-\partial_tw-\Delta_x w-A\cdot\nabla_xw+(q-\textrm{div}_xA)w=0,\quad &\textrm{in}\ \Omega_2\times(0,T),\\  w=0,\quad &\textrm{on}\ \Omega_1\times (0,T) .\end{array}\right.$$
Then the unique continuation properties for parabolic equations implies that $w_{|\Omega_2\times(0,T)}=0$ which implies that $w_{|\Omega_*\times(0,T)}=0$. Note that here we consider an application of unique continuation to solutions of parabolic equations lying in $H^1(Q)$ and with a zero order coefficient $(q+\nabla_x\cdot (B-A))\in L^\infty(0,T;L^{p_1}(\Omega_2))$. For this purpose one needs to extend by density  Carleman estimates like \cite[Theorem 1.2]{SS} to such solutions and use Sobolev embedding theorem in  order to absorb the multiplication by $(q+\nabla_x\cdot (B-A))$ which corresponds to a bounded operator from $L^2(0,T;H^1(\Omega_2))$ to $L^2(\Omega_2\times(0,T))$. In particular, we have
$$w_{|\partial\Omega_* \times (0,T)}=\partial_\nu w_{|\partial\Omega_* \times (0,T)}=0 $$
and it follows that 
$$w_{|\partial(\Omega\setminus\Omega_*)\times (0,T)}=\partial_\nu w_{|\partial(\Omega\setminus\Omega_*)\times (0,T) }=0. $$
Therefore, we have
$$\int_{Q_1}\Delta wu_0dxdt+\int_{Q_1}\nabla_x w\cdot\nabla_xu_0dxdt=0,$$
$$\int_{Q_1}\nabla_x w\cdot\nabla_xu_0dxdt=-\left\langle \Delta u_0, w\right\rangle_{L^2(0,T;H^{-1}(\Omega\setminus\Omega_*),L^2(0,T;H^{1}_0(\Omega\setminus\Omega_*)},$$
$$-\int_{Q_1}\pd_t wu_0dxdt=\left\langle \pd_t u_0, w\right\rangle_{L^2(0,T;H^{-1}(\Omega\setminus\Omega_*),L^2(0,T;H^{1}_0(\Omega\setminus\Omega_*)}.$$
Thus, we find
$$\begin{aligned} & \int_{Q_1}u_0[-\partial_tw-\Delta_x w-A\cdot\nabla_xw+(q+\nabla_x\cdot (B-A))w]dxdt\\
&= \int_{Q_1}u_0[-\partial_tw-\Delta_x w-A\cdot\nabla_xw+(q+\nabla_x\cdot (B-A))w]dxdt\\
&\ \ \ -\int_{Q_1}[\partial_tu_0-\Delta_x u_0+A\cdot\nabla_xu_0+(q+\nabla_x\cdot (B))u_0]wdxdt\\
\ &=0.\end{aligned}$$
 According to this last formula, we have
$$\int_{Q_1}u_0hdxdt=\int_{Q_1}u_0[-\partial_tw-\Delta_x w-A\cdot\nabla_xw+(q+\nabla_x\cdot (B-A))w]dxdt=0$$
which contradicts \eqref{lll4b}. This proves the required density result.\end{proof}

Armed with this lemma we are now in position to complete the proof of Corollary \ref{c3}.\\
\ \\
\textbf{Proof of Corollary \ref{c3}.} 
Using arguments similar to those used for the derivation of \eqref{p3a}, we can prove that, for any $u_1\in S_{+,A_1,B_1,q_1,\gamma_1}$ and $u_2\in S_{-,A_2,B_2,q_2,\gamma_2}$, we have
$$\begin{aligned} &\left\langle(\Lambda_{A_1,B_1,q_1,\gamma_1,\gamma_2}-\Lambda_{A_2,B_2,q_2,\gamma_1,\gamma_2})g_+, g_-  \right\rangle_{\mathcal H_-^*,\mathcal H_-}\\
&=\int_Q(A_1-A_2)\cdot\nabla_xu_1u_2dxdt-\int_Q(B_1-B_2)\cdot\nabla_x(u_1u_2)dxdt+\int_Q(q_1-q_2)u_1u_2dxdt.\end{aligned}$$
with $g_+=u_1$ and $g_-=u_2$ on $\Sigma$. Then, \eqref{c3b} implies that, for any $u_1\in S_{+,A_1,B_1,q_1,\gamma_1}$, $u_2\in S_{-,A_2,B_2,q_2,\gamma_2}$, we get
\bel{c3c}\int_Q(A_1-A_2)\cdot\nabla_xu_1u_2dxdt-\int_Q(B_1-B_2)\cdot\nabla_x(u_1u_2)dxdt+\int_Q(q_1-q_2)u_1u_2dxdt=0.\ee
In view of \eqref{c3a}, we can rewrite \eqref{c3c} as
$$\int_{Q_1}(A_1-A_2)\cdot\nabla_xu_1u_2dxdt-\int_{Q_1}(B_1-B_2)\cdot\nabla_x(u_1u_2)dxdt+\int_{Q_1}(q_1-q_2)u_1u_2dxdt=0.$$
Then, using \eqref{c3c} and integrating by parts in $x\in\Omega\setminus\Omega_*$, for any $u_1\in S_{+,A_1,B_1,q_1,\gamma_1}$, $u_2\in S_{-,A_2,B_2,q_2,\gamma_2}$, we find
\bel{c3d}\int_{Q_1}(A_1-A_2)\cdot\nabla_xu_1u_2dxdt+\int_{Q_1}\nabla_x\cdot(B_1-B_2)u_1u_2dxdt+\int_{Q_1}(q_1-q_2)u_1u_2dxdt=0.\ee
Applying the density result of Lemma \ref{lll4}, we deduce that \eqref{c3d} holds true for  any $u_1\in S_{+,A_1,B_1,q_1,\gamma_1}$, $u_2\in S_{-,A_2,B_2,q_2}$. Then, using \eqref{c3c} and integrating by parts in $x\in\Omega\setminus\Omega_*$, for  any $u_1\in S_{+,A_1,B_1,q_1,\gamma_1}$, $u_2\in S_{-,A_2,B_2,q_2}$, we obtain
\bel{c3e}-\int_{Q_1}\nabla_x\cdot\left[(A_1-A_2)u_2\right]u_1dxdt+\int_{Q_1}\nabla_x\cdot(B_1-B_2)u_1u_2dxdt+\int_{Q_1}(q_1-q_2)u_1u_2dxdt=0.\ee
Applying again Lemma \ref{lll4}, we deduce that \eqref{c3e} holds for any $u_1\in S_{+,A_1,B_1,q_1}$, $u_2\in S_{-,A_2,B_2,q_2}$. Integrating again by parts, we deduce that \eqref{c3c} holds for any $u_1\in S_{+,A_1,B_1,q_1}$, $u_2\in S_{-,A_2,B_2,q_2}$. Finally, allowing  $u_1\in S_{+,A_1,B_1,q_1}$, $u_2\in S_{-,A_2,B_2,q_2}$ to correspond to the exponentially growing and decaying GO solutions used in Theorem \ref{t5}, we can complete the proof of the corollary. \qed

\section{Application to the recovery of nonlinear terms}
In this section $\Omega$ is of class $\mathcal C^{2+\alpha}$ and we denote by $\Sigma_p$ the parts of $\partial Q$ given by $\Sigma_p=\Sigma\cup \Omega^0$. Consider  the quasi-linear IBVP \eqref{1.1}.
Following \cite{LSU}, we start by fixing the condition for the well posedness of this problem. We consider, functions $F\in C^1(\overline{Q}\times\mathbb{R} \times\R^n)$ satisfying  the following conditions:

There exist three non-negative constants $c_0$,  $c_1$ and $c_2$ so that
\bel{nl1}
uF(x,t,u,v)\geq -c_0|v|^2-c_1|u|^2-c_2,\;\; (x,t,u,v)\in \overline{Q}\times \mathbb{R}\times\R^n.
\ee
\bel{nl2}
F(x,0,u,v)=0,\;\; (x,u,v)\in \partial\Omega\times \mathbb{R}\times\R^n.
\ee
Moreover, we assume that for $|u|\leq M_1$ and $(x,t)\in\overline{Q}$ there exists a constant $c_3(M_1)>0$, depending only on $T$, $\Omega$ and $M_1$, such that
\bel{nl3} |F(x,t,u,v)|\leq c_3(M_1)(1+|v|)^2.\ee
Here $M_1\mapsto c_3(M_1)$ is assumed to be monotonically increasing.

Now, for $G\in C^{2+\alpha ,1+\alpha /2}(\overline{Q})$ consider the compatibility condition
\bel{nnl}\partial_tG(x,0)=\Delta G(x,0),\quad x\in\partial\Omega.\ee
We consider the set $\mathcal X=\{G_{|\overline{\Sigma}_p};\; \mbox{for some}\; G\in C^{2+\alpha ,1+\alpha /2}(\overline{Q})\; \mbox{such that \eqref{nnl} is fulfilled}\}$ with the norm
$$\norm{G}_\mathcal X:=\|G_{|\Sigma}\|_{ C^{2+\alpha ,1+\alpha /2}(\overline{\Sigma})}+\|G_{|\Omega\times\{0\}}\|_{ C^{2+\alpha }(\overline{\Omega})}.$$
 According to \cite[Theorem 6.1, pp. 452]{LSU}, for any $G\in \mathcal X$ 
and for any $F\in C^1(\overline{Q}\times\mathbb{R} \times\R^n)$ satisfying \eqref{nl1}-\eqref{nl3}, problem \eqref{1.1} admits a unique solution $u_{F,G}\in C^{2+\alpha ,1+\alpha /2}(\overline{Q})$. Moreover, according to  \cite[Theorem 2.2, pp. 429]{LSU}, \cite[Theorem 4.1, pp. 443]{LSU} and \cite[Theorem 5.4, pp. 448]{LSU}, for any $r>0$ and for any  $G\in \mathcal X$ satisfying 
$$\norm{G}_{\mathcal X}\leq r,$$
there exists a constant $M_r$, depending on $\Omega$, $T$, $c_0$, $c_1$, $c_2$, $c_3$ and $r$ such that 
\bel{nl4}\|u_{F,G}\|_{C^{\alpha ,\alpha /2}(\overline{Q})}+\|\nabla_xu_{F,G}\|_{C^{\alpha ,\alpha /2}(\overline{Q})}\leq M_r.\ee
We associate to \eqref{1.1} the DN map
$$\mathcal N_F:\mathcal X\ni G\mapsto {\partial_\nu u_{F,G}}_{|\Sigma}\in L^2(\Sigma).$$
Since \eqref{1.1} is not linear, clearly $\mathcal N_F$ is also nonlinear. Therefore, in a similar way to \cite{CK2,I2,I3,I4}, we will start by linearizing this operator by considering the Fr\'echet derivative of $\mathcal N_F$. 
\subsection{Linearization procedure}

We fix $F\in \mathcal C^1(\overline{Q}\times\R^n\times \mathbb{R})$ satisfying \eqref{nl1}-\eqref{nl3} such that $\partial_uF\in \mathcal C^2(\overline{Q}\times\R^n\times \mathbb{R};\R)$ and  $\partial_vF\in  \mathcal C^2(\overline{Q}\times\R^n\times \mathbb{R};\R^n)$. Then, for $H\in\mathcal X$, we consider the IBVP
\bel{lin}
\left\{
\begin{array}{ll}
(\partial_tw-\Delta w+A_{F,G}(x,t)\cdot \nabla_x w+q_{F,G}(x,t)w=0\quad &\mbox{in}\; Q,
\\
w=H &\mbox{on}\; \Sigma _p,
\end{array}
\right.
\ee
with
$$A_{F,G}(x,t):=\partial _vF(x,t,u_{F,G }(x,t),\nabla_xu_{F,G }(x,t)), \quad (x,t)\in Q, $$
$$q_{F,G}(x,t):=\partial _uF(x,t,u_{F,G }(x,t),\nabla_xu_{F,G }(x,t)), \quad (x,t)\in Q. $$
 In light of \cite[Theorem 5.4, pp. 322]{LSU} and \eqref{nl2}, the IBVP \eqref{lin} has a unique solution $w=w_{F,G,H}\in  C^{2+\alpha ,1+\alpha /2}(\overline{Q})$ satisfying
\[
\|w_{F,G,H}\|_{C^{2+\alpha ,1+\alpha /2}(\overline{Q})}\le C\|H\|_{\mathcal X}
\]
for some constant $C$ depending only on $Q$, $F$ and $G$.  From now on, for $X=\Omega$ or $X=\pd\Omega$ and $r,s>0$ we consider the Sobolev spaces
$$H^{r,s}(X\times (0,T))=H^s(0,T;L^2(X))\cap L^2(0,T;H^r(X)).$$
Using solutions of \eqref{lin}, we will consider the linearization of $\mathcal N_F$ in the following way.

\begin{prop}\label{pS1}
$F\in C^1(\overline{Q}\times\R^n\times \mathbb{R})$ satisfying \eqref{nl1}-\eqref{nl3}  such that $\partial_uF\in \mathcal C^1(\overline{Q}\times\R^n\times \mathbb{R};\R)$ and  $\partial_vF\in  \mathcal C^1(\overline{Q}\times\R^n\times \mathbb{R};\R^n)$. Then, $N_F$ is  Fr\'echet continuously differentiable and 
\[
N'_F(G ) H=\partial _\nu w_{F,G,H},\;\; G,H\in \mathcal X.
\]
\end{prop}

\begin{proof}
Since  $u\in H^{2,1}(Q)\rightarrow \partial _\nu u\in  L^2(\Sigma)$ is a bounded linear operator, we only need to show that $$\mathcal M_F:G\in \mathcal X\rightarrow u_{F,G}\in H^{2,1}(Q)$$ is  differentiable with $\mathcal M'_F(G )(H)=w_{F,G,H}$, $G,H\in \mathcal X$. For this purpose, we fix $G,H\in \mathcal X$ with $\norm{H}_\mathcal X\leq1$ and we consider 
\[
z=u_{F,G+H}-u_{F,G} -w_{F,G,H}\in C^{2+\alpha ,1+\alpha /2}(\overline{Q})
\]
and set
\begin{align*}
&A(x,t)=\partial _vF(x,t,u_{F,G }(x,t),\nabla_xu_{F,G }(x,t)),\\
&q(x,t)=\partial _uF(x,t,u_{F,G }(x,t),\nabla_xu_{F,G }(x,t)),
\\
&A_1(x,t)=\int_0^1(1-\tau)\partial _v^2F(x,t,u_{F,G }(x,t),\nabla_xu_{F,G }(x,t)+\tau(\nabla_xu_{F,G +H}-\nabla_xu_{F,G })(x,t) )d\tau ,\\
&A_2(x,t)=\int_0^1(1-\tau)\partial _v\partial_uF(x,t,u_{F,G }(x,t),\nabla_xu_{F,G }(x,t)+\tau(\nabla_xu_{F,G +H}-\nabla_xu_{F,G })(x,t) )d\tau,\\
&q_1(x,t)=\int_0^1(1-\tau)\partial _u^2F(x,t,u_{F,G }(x,t)+\tau(u_{F,G +H}-u_{F,G })(x,t),\nabla_xu_{F,G+H }(x,t) )d\tau.\end{align*}
Applying Taylor's formula, we get
\begin{align*}
&F(x,t,u_{F,G }(x,t),\nabla_xu_{F,G+H }(x,t))-F(x,t,u_{F,G }(x,t),\nabla_xu_{F,G }(x,t))
\\ 
&=A(x,t)\cdot(\nabla_xu_{F,G+H }(x,t)-\nabla_xu_{F,G }(x,t)) \\
&\ \ \ +A_1(x,t)(\nabla_xu_{F,G+H }(x,t)-\nabla_xu_{F,G }(x,t),\nabla_xu_{F,G+H }(x,t)-\nabla_xu_{F,G }(x,t)).
\end{align*}
\begin{align*}
&F(x,t,u_{F,G+H }(x,t),\nabla_xu_{F,G+H }(x,t))-F(x,t,u_{F,G }(x,t),\nabla_xu_{F,G+H }(x,t))
\\ 
&=q(x,t)(u_{F,G+H }(x,t)-u_{F,G }(x,t))+q_1(x,t)(u_{F,G+H }(x,t)-u_{F,G }(x,t))^2 \\
&\ \ \ +A_2(x,t)(u_{F,G+H }(x,t)-u_{F,G }(x,t))(\nabla_xu_{F,G+H }(x,t)-\nabla_xu_{F,G }(x,t)).
\end{align*}
Thus, fixing
\begin{align*}
&K_H(x,t)
\\ 
&=q_1(x,t)(u_{F,G+H }(x,t)-u_{F,G }(x,t))^2+A_2(x,t)(u_{F,G+H }(x,t)-u_{F,G }(x,t))(\nabla_xu_{F,G+H }(x,t)-\nabla_xu_{F,G }(x,t))\\
&\ \ \ +A_1(x,t)(\nabla_xu_{F,G+H }(x,t)-\nabla_xu_{F,G }(x,t),\nabla_xu_{F,G+H }(x,t)-\nabla_xu_{F,G }(x,t))
\end{align*}
we deduce that $z$ is the solution of the IBVP
\begin{equation*}
\left\{
\begin{array}{ll}
\partial_tz-\Delta_x z+A\cdot\nabla_x z+qz=K_H\quad &\mbox{in}\; Q,
\\
z=0 &\mbox{on}\; \Sigma _p.
\end{array}
\right.
\end{equation*}
Combining this with \eqref{nl4}, \cite[Theorem 4.1, Chapter 3]{LM1}, \cite[Theorem 3.2, Chapter 4]{LM2} and the fact that $\norm{H}_\mathcal X\leq1$, we deduce that this last problem admits a unique solution $z\in H^{2,1}(Q)$ satisfying
\bel{ps1a}\norm{z}_{H^{2,1}(Q)}\leq C\norm{K_H}_{L^2(Q)}\leq C\norm{K_H}_{L^\infty(Q)}\ee
with $C$ depending on $\Omega$, $T$, $c_0$, $c_1$, $c_2$, $c_3$ and $\norm{G}_\mathcal X$. Moreover, applying again \eqref{nl4}, we obtain
$$\norm{K_H}_{L^\infty(Q)}\leq C\left(\norm{\nabla_xu_{F,G+H }-\nabla_xu_{F,G }}_{L^\infty(Q)}+\norm{u_{F,G+H }-u_{F,G }}_{L^\infty(Q)}\right)^2,$$
with $C$ depending on $\Omega$, $T$, $c_0$, $c_1$, $c_2$, $c_3$ and $\norm{G}_\mathcal X$.
Combining this with \eqref{ps1a}, we get
\bel{ps1b}\norm{z}_{H^{2,1}(Q)}\leq C\left(\norm{\nabla_xu_{F,G+H }-\nabla_xu_{F,G }}_{L^\infty(Q)}+\norm{u_{F,G+H }-u_{F,G }}_{L^\infty(Q)}\right)^2\ee
with $C$ depending on $\Omega$, $T$, $c_0$, $c_1$, $c_2$, $c_3$ and $\norm{G}_\mathcal X$.
On the other hand, fixing  $y=u_{F,G+H }-u_{F,G }$, one can check that $y$ solves 
\begin{equation*}
\left\{
\begin{array}{ll}
\partial_ty-\Delta_x y+\tilde{A}(x,t)\cdot \nabla_x y+\tilde{q}(x,t) y=0\quad &\mbox{in}\; Q,
\\
y=H &\mbox{on}\; \Sigma _p ,
\end{array}
\right.
\end{equation*}
with
\[
\tilde{A}(x,t)=\int_0^1\partial _v F(x,t,u_{F,G+H}(x,t),\nabla_xu_{F,G}(x,t)+\tau (\nabla_xu_{F,G+H}(x,t)-\nabla_xu_{F,G}(x,t)))d\tau,
\]
\[
\tilde{q}(x,t)=\int_0^1\partial _u F(x,t,u_{F,G}(x,t)+\tau (u_{F,G+H}(x,t)-u_{F,G}(x,t)),\nabla_xu_{F,G}(x,t))d\tau.
\]
Applying again \eqref{nl4}, we deduce that
$$\norm{\tilde{A}}_{C^{\alpha ,\alpha /2}(\overline{Q})}+\norm{\tilde{q}}_{C^{\alpha ,\alpha /2}(\overline{Q})}\leq C$$
with $C$ depending on $\Omega$, $T$, $c_0$, $c_1$, $c_2$, $c_3$ and $\norm{G}_\mathcal X$. Therefore, applying \cite[Theorem 5.3, pp. 320-321]{LSU}, we obtain
\bel{ps1c} \norm{\nabla_xy}_{L^\infty(Q)}+\norm{y}_{L^\infty(Q)}\leq C\norm{H}_\mathcal X,\ee
with $C$ depending on $\Omega$, $T$, $c_0$, $c_1$, $c_2$, $c_3$ and $\norm{G}_\mathcal X$.
Combining \eqref{ps1a}-\eqref{ps1c}, we find
$$\norm{u_{F,G+H}-u_{F,G} -w_{F,G,H}}_{H^{2,1}(Q)}\leq  C\norm{H}^2_\mathcal X.$$
From this last estimate one can easily check that $\mathcal M_F$ is differentiable at $G$ and $M'_F(G)(H)=w_{F,G,H}$, $H\in \mathcal X$. To complete the proof of the proposition,
we only need to check the continuity of $\mathcal X\ni G\rightarrow \mathcal M_F'(G)\in \mathscr{B}(\mathcal X, H^{2,1}(Q))$. For this purpose, we fix $G,K,H\in \mathcal X$, we consider $S:=w_{F,G+K,H}-w_{F,G,H}$, with $\norm{K}_\mathcal X\leq 1$, and we observe that $S$ solves
\begin{equation*}
\left\{
\begin{array}{ll}

\partial _tS -\Delta_x S +\tilde{A}_1(x,t)\cdot\nabla_x S+\tilde{q}_1(x,t) S =R_K\quad &\mbox{in}\; Q,
\\
S =0 &\mbox{on}\; \Sigma _p ,
\end{array}
\right.
\end{equation*}
where
$$\tilde{A}_1(x,t):=\partial _vF(x,t,u_{F,G+K}(x,t), \nabla_x u_{F,G+K}(x,t)),\quad (x,t)\in Q,$$
$$\tilde{q}_1(x,t):=\partial _uF(x,t,u_{F,G+K}(x,t), \nabla_x u_{F,G+K}(x,t)),\quad (x,t)\in Q,$$
$$R_K:= A_3 (\nabla_xu_{F,G}-\nabla_xu_{F,G+k},\nabla_xw_{F,G,H})+q_3 (u_{F,G}-u_{F,G+K})w_{F,G,H},$$
with
\[
A_3 (x,t)=\int_0^1\partial _v^2F(x,t,u_{F,G+K}(x,t),\nabla_x u_{F,G+K}(x,t)+\tau (\nabla_xu_{F,G+K}(x,t)-\nabla_xu_{F,G}(x,t))) d\tau,
\]

\[
q_3 (x,t)=\int_0^1\partial _u^2F(x,t,u_{F,G+K}(x,t)+\tau (u_{F,G+K}(x,t)-u_{F,G}(x,t)),\nabla_x u_{F,G+K}(x,t)) d\tau .
\]
Repeating the above arguments, we find 
$$\norm{S}_{H^{2,1}(Q)}\leq C\norm{R_K}_{L^2(Q)}\leq  C\norm{K}_\mathcal X,$$
with $C$ depending on $\Omega$, $T$, $c_0$, $c_1$, $c_2$, $c_3$ and $\norm{G}_\mathcal X+\norm{H}_\mathcal X$. This proves the continuity of $ G\mapsto \mathcal M_F'(G)\in \mathscr{B}(\mathcal X, H^{2,1}(Q))$ and it completes the proof of the proposition.
\end{proof}

We will apply this property of the DN map $\mathcal N_F$ in order to complete the proof of Theorem \ref{tnl1}, \ref{tnl3} and Corollary \ref{cnl2}, \ref{cnl}.
\subsection{Proof of Theorem \ref{tnl1} and Corollary \ref{cnl}}

This subsection is devoted to the proof of Theorem \ref{tnl1}, \ref{tnl3} and Corollary \ref{cnl2}, \ref{cnl}. We start by considering the following intermediate result.

\begin{lem}\label{nll1} Let $G\in \{K_{|\Sigma_p}:\ K\in\mathcal C^\infty(\overline{Q}),\ \partial_tK=0,\ \nabla_xK\ \textrm{is constant}\}$ and assume that
\bel{nll1a}\partial_x^\ell F(x,0,u,v)=0,\quad    x\in\partial\Omega,\ u\in\R,\ v\in\R^n,\ \ell\in\mathbb N^n,\ |\ell|\leq2.\ee
Then the problem \eqref{1.1} admits a unique solution $u_{F,G}\in \mathcal C^{2+\alpha,1+\frac{\alpha}{2}}(\overline{Q})$ satisfying $ \partial_t u_{F,G}\in \mathcal C^{2+\alpha,1+\frac{\alpha}{2}}(\overline{Q})$.

\end{lem}
\begin{proof} Let $u_{F,G}\in \mathcal C^{2+\alpha,1+\frac{\alpha}{2}}(\overline{Q})$ be the solution of \eqref{1.1} and fix $K\in\mathcal C^\infty(\overline{Q})$ satisfying $\partial_tK=0$, $\nabla_xK$ is constant such that $K_{|\Sigma_p}=G$. We start by fixing $z=\partial_tu_{F,G}$, $$A(x,t):= \partial_vF(x,t,u_{F,G}(x,t),\nabla_xu_{F,G}(x,t)),\quad q(x,t):=\partial_uF(x,t,u_{F,G}(x,t),\nabla_xu_{F,G}(x,t)),$$
$$R:(x,t)\mapsto-\partial_t F(x,t,u_{F,G}(x,t),\nabla_xu_{F,G}(x,t))\in \mathcal C^{\alpha,\frac{\alpha}{2}}(\overline{Q})$$
and $z_0\in \mathcal C^{2+\alpha}(\overline{\Omega})$, $g\in \mathcal C^{2+\alpha,1+\frac{\alpha}{2}}(\overline{\Sigma})$ such that
$$z_0(x):=\Delta_x K(x,0)-F(x,0,K(x,0),\nabla_xK(x,0))=-F(x,0,K(x,0),\nabla_xK(x,0)) ,\quad x\in\Omega,$$
$$g(x,t)=\partial_t K(x,t):=0,\quad (x,t)\in\Sigma.$$
Applying \eqref{nll1a}, one can check that $Z:=(g,z_0)\in\mathcal X$ satisfies
\bel{compa}\partial_tg(x,0)-\Delta_x z_0(x)+A(x,0)\cdot\nabla_xz_0(x)+q(x,0)z_0(x)-R(x,0)=\partial_tg(x,0)-\Delta_x z_0(x)=0,\quad x\in\partial\Omega.\ee
Moreover, $z$ solves the IBVP
\begin{equation}\label{nll1c}\left\{\begin{array}{ll}\partial_tz-\Delta_x z+A(x,t)\cdot\nabla_xz+q(x,t)z=R(x,t),\quad &\textrm{in}\ Q,\\  z=Z,\quad &\textrm{on}\ \Sigma_p.\end{array}\right.\end{equation}
Using the fact that $A,q, R\in \mathcal C^{\alpha,\frac{\alpha}{2}}(\overline{Q})$, $Z\in \mathcal X$ and the fact that \eqref{compa} is fulfilled, we deduce from \cite[Theorem 5.3, pp. 320-321]{LSU} that $\partial_tu_{F,G}\in \mathcal C^{2+\alpha,1+\frac{\alpha}{2}}(\overline{Q})$. \end{proof}

Armed with this lemma we will complete the proof of Theorem \ref{tnl3} and \ref{tnl1}.\\
\ \\
\textbf{Proof of Theorem \ref{tnl3}.} For $j=1,2$,  $a\in\R$, $v\in\R^n$, $(x,t)\in Q$, we fix
$$A_{j,a,v}(x,t):= \partial_vF_j(x,t,u_{F_j,h_{a,v}}(x,t),\nabla_xu_{F_j,h_{a,v}}(x,t)),\quad q_{j,a,v}(x,t):= \partial_uF_j(x,t,u_{F_j,h_{a,v}}(x,t),\nabla_xu_{F_j,h_{a,v}}(x,t))$$
and $B_{1,a,v}=0$.
According to \eqref{tnl3a} and Proposition \ref{pS1}, we have
$$\Lambda_{A_{1,v},B_{1,v}, q_{1,v}}H_{|\Sigma}=\Lambda_{A_{2,v},B_{1,v}, q_{2,v}}H_{|\Sigma},\quad a\in\R,\ v\in\R^n,\ H\in\mathcal X_0.$$
Combining this with the density of $\{H_{|\Sigma}:\ H\in \mathcal X_0\}$ in $\mathcal H_+$, we obtain
\bel{tnl3f} \Lambda_{A_{1,a,v},B_{1,a,v}, q_{1,a,v}}=\Lambda_{A_{2,a,v},B_{1,a,v}, q_{2,a,v}},\quad a\in\R,\ v\in\R^n\ee
and, in view of \eqref{tnl1aa} and Lemma \ref{nll1}, we have 
$$A_{j,a,v}\in \mathcal C^1(\overline{Q}),\quad q_{j,a,v}\in L^\infty(Q),\quad j=1,2,\  a\in\R,\ v\in\R^n.$$
Note that, due to \eqref{tnl1aa}, the fact that $\partial_t h_{a,v}=0$ and the fact that $\nabla_xh_{a,v}=v$, here we are actually in position to apply Lemma \ref{nll1}.
Moreover, according to \cite[Lemma 8.2]{I4}, \eqref{tnl3f} implies
\bel{tnl3c}A_{1,a,v}(x,t)=A_{2,a,v}(x,t),\quad (x,t)\in\Sigma,\ a\in\R, v\in\R^n.\ee
Therefore, Theorem \ref{t5}  implies that, for $A_{a,v}=A_{1,a,v}-A_{2,a,v}$ extended by zero to $\R^n\times(0,T)$ and for
$$\phi_{a,v}(x,t):=-\int_0^1\frac{A_{a,v}(sx,t)\cdot x}{2}ds,\quad (x,t)\in\R^n\times\R,\ (a,v)\in\R\times\R^n,$$
we have
\bel{tnl3d}A_{2,a,v}(x,t)=A_{1,a,v}(x,t)+2\nabla_x\phi_{a,v}(x,t),\quad (x,t)\in \R^n\times(0,T).\ee
In view of \eqref{tnl3c}, the fact  that $F_j\in C^{2+\alpha,1+\frac{\alpha}{2}}(\overline{Q};\mathcal C^3(\mathbb{R} \times\R^n))$ and the definition of $A_{j,a,v}$, $j=1,2$, one can easily
check that $(x,t,a,v)\mapsto A_{a,v}(x,t)\in \mathcal C^1(\overline{Q};\mathcal C(\R\times\R^n))$. Therefore, we have $$\phi:(x,t,a,v)\mapsto \phi_{a,v}(x,t)\in \mathcal C^1([0,T];\mathcal C(\overline{\Omega}\times\R\times\R^n))\cap\mathcal C^2(\overline{\Omega};\mathcal C([0,T]\times\R\times\R^n)). $$ In a similar way to the end of the proof of Theorem \ref{t5}, by eventually subtracting to $\phi$ a function $y$ depending only on $(t,a,v)\in(0,T)\times\R\times\R^n$, we may assume that
\bel{tnl3e}\phi(x,t,a,v)=0,\quad (x,t)\in \Sigma,\ (a,v)\in\R\times\R^n.\ee
Therefore, we can apply the gauge invariance of the DN map $\Lambda_{A_{1,a,v},B_{1,a,v}, q_{1,a,v}}$ to get
$$\Lambda_{A_{1,a,v},B_{1,a,v}, q_{1,a,v}}=\Lambda_{A_{2,a,v},B_{1,a,v}+\nabla_x\phi(\cdot,a,v),q_{1,a,v}+[-\pd_t\phi(a,v,+\Delta_x\phi-|\nabla_x\phi|^2-A_{1,a,v}\cdot\nabla_x\phi](\cdot,a,v)}.$$
Combining this with \eqref{tnl3f}, we get
$$\Lambda_{A_{2,a,v},B_{1,a,v}, q_{2,a,v}}=\Lambda_{A_{2,a,v},B_{1,a,v}+\nabla_x\phi(\cdot,a,v),q_{1,a,v}+[-\pd_t\phi(a,v,+\Delta_x\phi-|\nabla_x\phi|^2-A_{1,a,v}\cdot\nabla_x\phi](\cdot,a,v)}.$$
Using \eqref{tnl3c} and repeating the arguments used at the end of the proof of Theorem \ref{t5} we deduce that, for all $(a,v)\in\R\times\R^n$, we have
$$\left\{\begin{array}{ll}A_{2,a,v}(x,t)=A_{1,a,v}(x,t)+2\nabla_x\phi(x,t,a,v),\  &(x,t)\in Q,\\  q_{2,a,v}(x,t)=q_{1,a,v}(x,t)+[-\pd_t\phi(a,v,+\Delta_x\phi-|\nabla_x\phi|^2-A_{1,a,v}\cdot\nabla_x\phi](x,t,a,v),\quad &(x,t)\in Q,\\ \phi(x,t,a,v)=0,\quad &(x,t)\in \Sigma.\end{array}\right.$$
Sending $t\to0$ in the first two above equality, for all $(a,v)\in\R\times\R^n$, we obtain
$$\left\{\begin{array}{ll}\partial_v(F_2-F_1)(x,0,x\cdot v+a,v)=2\partial_x\phi(x,0,a,v),\  &x\in\Omega,\\  \partial_u(F_2-F_1)(x,0,x\cdot v+a,v)=-(\pd_t\phi-|\nabla_x\phi|^2-\partial_vF_1(x,0,x\cdot v+a,v)\partial_x\phi)(x,0,a,v),\  &x\in\Omega,\\ \phi(x,t,a,v)=0,\ &\ (x,t)\in \Sigma.\end{array}\right.$$
Finally, fixing $a=u-x\cdot v$ in the two first above equalities, we obtain \eqref{tnl3b}. This completes the proof of the theorem.\qed \\
\ \\
\ \\
\textbf{Proof of Theorem \ref{tnl1}.} For $j=1,2$,  $v\in\R^n$, $(x,t)\in Q$, we fix
$$A_{j,v}(x,t):= \partial_vF_j(x,t,u_{F_j,k_v}(x,t),\nabla_xu_{F_j,k_v}(x,t)),\quad q_{j,v}(x,t):= \partial_uF_j(x,t,u_{F_j,k_v}(x,t),\nabla_xu_{F_j,k_v}(x,t))$$
and $B_{j,v}=0$.
In a similar way to Theorem \ref{tnl3}, applying \eqref{tnl1b} and Proposition \ref{pS1}, we obtain
\bel{tnll1} \Lambda_{A_{1,v},B_{1,v}, q_{1,v}}=\Lambda_{A_{2,v},B_{1,v}, q_{2,v}},\quad v\in\R^n\ee
and, in view of \eqref{tnl1aa} and Lemma \ref{nll1}, we have 
$$A_{j,v}\in W^{1,\infty}(Q),\quad q_{j,v}\in L^\infty(Q),\quad j=1,2,\ v\in\R^n.$$
Note that, due to \eqref{tnl1aa}, the fact that $\partial_t k_v=0$ and the fact that $\nabla_xk_v=v$, here we are actually in position to apply Lemma \ref{nll1}.
Moreover, according to \cite[Lemma 8.2]{I4}, \eqref{tnll1} implies
$$A_{1,v}(x,t)=A_{2,v}(x,t),\quad (x,t)\in\Sigma,\ v\in\R^n.$$
In addition, from \eqref{tnl1a}-\eqref{tnl2a}, we deduce $(x,t,u,v)\mapsto \partial_uF_j(x,t,u,v)$, $j=1,2$, is a  function independent of $u$ and $v$ with
$$\partial_uF_1(x,t,0,0)=\partial_uF_1(x,t,u,v)=\partial_uF_2(x,t,u,v)=\partial_uF_2(x,t,0,0),\quad (x,t)\in Q,\ u\in\R,\ v\in\R^n.$$
It follows
\bel{tnll4}\begin{aligned}q_{1,v}(x,t)&= \partial_uF_1(x,t,u_{F_1,k_v}(x,t),\nabla_xu_{F_1,k_v}(x,t))\\
\ &=\partial_uF_1(x,t,0,0)\\
\ &=\partial_uF_2(x,t,0,0)=\partial_uF_2(x,t,u_{F_1,k_v}(x,t),\nabla_xu_{F_1,k_v}(x,t))=q_{2,v}(x,t),\quad (x,t)\in Q,\ v\in\R^n.\end{aligned}\ee
Thus, applying Corollary \ref{c1}, we deduce that
$$A_{1,v}(x,t)=A_{2,v}(x,t),\quad (x,t)\in Q,\ v\in\R^n.$$
Sending $t\to0$  in this formula, we obtain 
\bel{tnll5}F_1(x,0,x\cdot v,v)=F_2(x,0,x\cdot v,v),\quad x\in\Omega,\ v\in\R^n.\ee
On the other hand, according to \eqref{tnl1a}-\eqref{tnl2a} we have
$$F_j(x,t,u,v)=F_j(x,t,0,v)+\partial_uF_j(x,t,0,v)u=F_j(x,t,0,v)+\partial_uF_1(x,t,0,v)u$$
and \eqref{tnll5} clearly implies \eqref{tnl1c}. Assuming that \eqref{tnl1e} is fulfilled, we can easily deduce \eqref{tnl1f} from \eqref{tnl1c}. This completes the proof of the theorem.\qed

Now let us consider Corollary \ref{cnl2} and \ref{cnl} which follow from  Theorem \ref{tnl3} and \ref{tnl1}.\\

\textbf{Proof of Corollary \ref{cnl2}.} Let condition \eqref{tnl3a} be fulfilled. Then Theorem \ref{tnl3} implies that there exists $\phi:Q\times\R\times\R^n\ni (x,t,u,v)\mapsto\phi(x,t,u,v)\in\mathcal C^1([0,T];\mathcal C(\overline{\Omega}\times\R\times\R^n))\cap\mathcal C^2(\overline{\Omega};\mathcal C([0,T]\times\R\times\R^n))$ such that, for all $(u,v)\in \R\times\R^n$, conditions \eqref{tnl3b} are fulfilled. Note that, for all $x\in\Omega$, $(u,v)\in \R\times\R^n$, we have
$$2\Delta_x\phi(x,0,u-x\cdot v,v)=\nabla_x\cdot[2\partial_x\phi(x,0,u-x\cdot v,v)]+2\partial_u\partial_x\phi(x,0,u-x\cdot v,v)v.$$
Then, \eqref{tnl3b} implies
$$2\Delta_x\phi(x,0,u-x\cdot v,v)=\sum_{j=1}^n\left[\partial_{x_j}\partial_{v_j}(F_2-F_1)(x,0,u,v)+\partial_{u}\partial_{v_j}(F_2-F_1)(x,0,u,v)v_j\right].$$
Applying \eqref{cnl2a}, we get
$$\Delta_x\phi(x,0,u-x\cdot v,v)=0,\quad x\in\Omega,\  (u,v)\in \R\times\R^n$$
and, replacing $u$ by $u+x\cdot v$ and applying \eqref{tnl3b}, we find 
$$\left\{\begin{aligned}\Delta_x\phi(x,0,u,v)&=0,\quad x\in\Omega,\  (u,v)\in \R\times\R^n\\ \phi(x,0,u,v)&=0,\quad x\in\partial \Omega,\  (u,v)\in \R\times\R^n.\end{aligned}\right.$$
From the uniqueness of this boundary value problem, we obtain
$$\phi(x,0,u,v)=0,\quad x\in\Omega,\  (u,v)\in \R\times\R^n,$$
which, combined with \eqref{tnl3b}, implies \eqref{tnl1c}. In addition, assuming that \eqref{tnl1e} is fulfilled, we can easily deduce \eqref{tnl1f} from \eqref{tnl1c}.\qed\\
\ \\
\textbf{Proof of Corollary \ref{cnl}.} In a similar way to Theorem \ref{tnl1}, for $j=1,2$, $v\in\R^n$, $(x,t)\in Q$, we fix
$$A_{j,v}(x,t):= \partial_vF_j(x,t,u_{F_j,k_v}(x,t),\nabla_xu_{F_j,k_v}(x,t)),\quad q_{j,v}(x,t):= \partial_uF_j(x,t,u_{F_j,k_v}(x,t),\nabla_xu_{F_j,k_v}(x,t)).$$
Applying \eqref{tnl1a} we deduce \eqref{tnll4} and from \eqref{cnl1a} we obtain that
$$A_{1,v}(x,t)=0=A_{2,v}(x,t),\quad (x,t)\in \Omega_*\times(0,T),\ v\in\R^n.$$
Combining this with Corollary \ref{c3}, we deduce that  there exists $\phi_v\in L^\infty(0,T;W^{2,\infty}(\Omega))\cap W^{1,\infty}(0,T;L^\infty(\Omega))$ such that
$$\left\{\begin{array}{ll}A_{2,v}=A_{1,v}+2\nabla_x\phi_{v},\quad &\textrm{in}\ Q,\\  q_{2,v}=q_{1,v}-\pd_t\phi_v+\Delta_x\phi_v-|\nabla_x\phi_v|^2-A_{1,v}\cdot\nabla_x\phi_v,\quad &\textrm{in}\ Q,\\ \phi_v=\pd_\nu\phi_v=0,\quad &\textrm{on}\ \Sigma.\end{array}\right.$$
Then, using \eqref{tnll4}, we deduce that $\phi_v$ satisfies
$$\left\{\begin{array}{ll} \pd_t\phi_v-\Delta_x\phi_v+A_{3,v}\cdot\nabla_x\phi_v=0,\quad &\textrm{in}\ Q, \\ \phi_v=\pd_\nu\phi_v=0,\quad &\textrm{on}\ \Sigma,\end{array}\right.$$
with $A_{3,v}=A_{1,v}+\nabla_x\phi_{v}\in L^\infty(Q)$.
 Thus, from the unique continuation properties for parabolic equations, we deduce that  $\phi_v=0$ and
$$A_{1,v}(x,t)=A_{2,v}(x,t),\quad (x,t)\in Q,\ v\in\R^n.$$
Then in a similar way to the end of the proof of Theorem \ref{tnl1} we deduce \eqref{tnl1c}.\qed

\section*{Acknowledgements}
PC was supported by MTM2018, Severo Ochoa SEV-2017-0718 and BERC 2018-2021. YK would like to thank Luc Robbiano    for helpfull remarks about unique continuation properties of solutions of parabolic equations that allow to improve several results of this paper. The work of YK is supported by  the French National
Research Agency ANR (project MultiOnde) grant ANR-17-CE40-0029.


\begin{thebibliography}{99}
%

\bibitem{AP} {\sc K. Astala and L. P\"{a}iv\"{a}rinta}, {\em Calder\'{o}n's inverse conductivity problem in the plane}, Annals of Mathematics,  \textbf{163} (2006), 265-299.
\bibitem{AU} {\sc H. Ammari and G. Uhlmann}, {\em Reconstuction from partial Cauchy data for the
Schr\"odinger equation}, Indiana University Math J., \textbf{53} (2004), 169-184.
\bibitem{BB}{\sc M. Bellassoued and I. Ben Aicha}, {\em Stable determination outside a cloaking region of two time-dependent coefficients in an hyperbolic equation from Dirichlet to Neumann map}, Jour. Math. Anal. Appl. \textbf{446} (2017), 46-76.
\bibitem{BKS}{\sc M. Bellassoued, Y. Kian, E. Soccorsi}, {\em An inverse problem for the magnetic Schr\"odinger equation in infinite cylindrical domains},  PRIMS, \textbf{54} (2018), 679-728.
\bibitem{A} {\sc I. Ben Aicha}, {\em  Stability estimate for hyperbolic inverse problem with time-dependent coefficient}, Inverse Problems, \textbf{31} (2015), 125010.
\bibitem{CDR1} {\sc P. Caro, D. Dos Santos Ferreira, A. Ruiz}, {\em Stability estimates for the Radon transform with restricted data and applications}, Advances in Math., {\bf 267} (2014), 523-564.
\bibitem{CDR2} {\sc P. Caro, D. Dos Santos Ferreira, A. Ruiz}, {\em Stability estimates for the Calder\'on problem with partial data},  J. Diff. Equat., \textbf{260} (2016), 2457-2489.
\bibitem{CM}{\sc P. Caro and K. Marinov}, {\em Stability of inverse problems in an infinite slab with partial data}, Commun. Partial Diff. Eqns., \textbf{41} (2016), 683-704.
\bibitem{CP}{\sc P. Caro  V. Pohjola}, {\em Stability Estimates for an Inverse Problem for the Magnetic Schr\"odinger Operator}, IMRN, \textbf{2015} (2015), 11083-11116.
\bibitem{CR}{\sc P. Caro and K. M. Rogers}, {\em Global Uniqueness for The Calder\'on Problem with Lipschitz Conductivities}, Forum of Mathematics, Pi, \textbf{4} (2016),  p. e2. Cambridge University Press.
\bibitem{CY1}{\sc J. Cheng and M. Yamamoto}, {\em  The global uniqueness for determining two convection coefficients from Dirichlet to Neumann map in two dimensions},  Inverse Problems, \textbf{16} (2000), L25.
\bibitem{CY}{\sc J. Cheng and M. Yamamoto}, {\em Identification of convection term in a parabolic equation with a single measurement}, Nonlinear Analysis, \textbf{50} (2002), 163-171.
\bibitem{CY2}{\sc J. Cheng and M. Yamamoto}, {\em  Determination of Two Convection Coefficients from Dirichlet to Neumann Map in the Two-Dimensional Case},  SIAM J. Math. Anal., \textbf{35} (2004), 1371-1393.
 \bibitem{Ch}{\sc M. Choulli}, {\em Une introduction aux probl\`emes inverses elliptiques et paraboliques}, Math\'ematiques et Applications, Vol. 65, Springer-Verlag, Berlin, 2009.
\bibitem{CK}{\sc M. Choulli and Y. Kian}, {\em Stability of the determination of a time-dependent coefficient in parabolic equations}, MCRF, {\bf 3} (2) (2013), 143-160.
\bibitem{CK2}{\sc M. Choulli and Y. Kian}, {\em Logarithmic stability in determining the time-dependent zero order coefficient in a parabolic equation from a partial Dirichlet-to-Neumann map. Application to the determination of a nonlinear term}, J. Math. Pures Appl., 	\textbf{114} (2018), 235-261.
\bibitem{CKS}{\sc M. Choulli, Y. Kian, E. Soccorsi},  {\em Determining the time dependent external potential from the DN map in a periodic quantum waveguide},   SIAM J.Math. Anal., \textbf{47} (6) (2015), 4536-4558.
 \bibitem{COY} {\sc M. Choulli, E. M. Ouhabaz, M. Yamamoto}, {\em Stable determination of
a semilinear term in a parabolic equation}, Commun. Pure Appl. Anal. \textbf{5} (3)
(2006), 447-462.
\bibitem{DYY}{\sc Z-C. Deng, J-N. Yu, L. Yang},  {\em  Identifying the coefficient of first-order in parabolic
equation from final measurement data}, Mathematics and Computers in Simulation, \textbf{77} (2008), 421-435.
\bibitem{DKSU}{\sc D. Dos Santos Ferreira, C. E. Kenig, J. Sj\"ostrand,  G. Uhlmann,} {\em Determining a magnetic Schr\"odinger
operator from partial Cauchy data}, Comm. Math. Phys., \textbf{271} (2) (2007), 467-488.
 \bibitem{DKLS}{\sc D. Dos Santos Ferreira, Y. Kurylev, M. Lassas, M. Salo}{\em The Calder\'on problem in transversally anisotropic geometries}, J. Eur. Math. Soc.,  \textbf{18} (2016),  2579-2626.
\bibitem{GK}{\sc P. Gaitan and Y. Kian}, {\em A stability result for a time-dependent potential in a cylindrical domain}, Inverse Problems, \textbf{29} (6) (2013), 065006.
\bibitem{Gr}{\sc P. Grisvard}, {\em Elliptic problems in nonsmooth domains}, Pitman, London, 1985.
\bibitem{HH} {\sc B. Haberman},  {\em Unique determination of a magnetic Schrödinger operator with unbounded magnetic potential from boundary data}, Int. Math. Res. Not., (2018), no. 4,, 1080-1128.
\bibitem{H2015} {\sc B. Haberman},  {\em Uniqueness in Calder\'{o}n's problem for conductivities with unbounded gradient}, Comm. Math. Phys., \textbf{340} (2015), 639-659. 
\bibitem{HT} {\sc B. Haberman and D. Tataru}, {\em Uniqueness in Calder\'{o}n's problem with Lipschitz conductivities}, Duke Math. Journal, {\bf 162} (3) (2013), 497-516.
\bibitem{HUW} {\sc M. De Hoop, G. Uhlmann, Y. Wang}, {\em Nonlinear responses from the interaction of two progressing waves at an interface}, to appear in Annales de l'IHP (C) Anal. Non Lin\'eaire, https://doi.org/10.1016/j.anihpc.2018.04.005.
\bibitem{Ho3}{\sc L. H\"ormander}, {\em The Analysis of linear partial differential operators}, Vol III, Springer-Verlag, Berlin, Heidelberg, 1983.
\bibitem{HK}{\sc G. Hu, Y. Kian}, 	{\em Determination of singular time-dependent coefficients for wave equations from full and partial data}, Inverse Probl. Imaging, \textbf{12} (2018), 745-772.
\bibitem{I1}{\sc V. Isakov}, {\em Completness of products of solutions and some inverse problems for PDE}, J. Diff. Equat., \textbf{92} (1991), 305-316.
\bibitem{I2}{\sc  V. Isakov},  {\em On uniqueness in inverse problems for semilinear parabolic equations}, Arch. Rat. Mech. Anal., \textbf{124} (1993), 1-12.
\bibitem{I3}{\sc  V. Isakov},  {\em Uniqueness of recovery of some systems of semilinear partial differential equations}, Inverse Problems, \textbf{17} (2001), 607-618.
\bibitem{I4}{\sc  V. Isakov},  {\em Uniqueness of recovery of some quasilinear Partial differential equations}, Commun. Partial Diff. Eqns., \textbf{26} (2001), 1947-1973.
\bibitem{I5}{\sc  V. Isakov},  {\em  Inverse Problems for Partial Differential Equations}, Volume 127,  Springer-Verlag, Berlin, Heidelberg, 2006.
\bibitem{IN}{\sc  V. Isakov and A. Nachman},  {\em Global Uniqueness for a two-dimensional elliptic inverse problem}, Trans. of AMS, \textbf{347} (1995), 3375-3391.
\bibitem{KKL} {\sc A. Katchalov,  Y. Kurylev, and M. Lassas}, {\em Equivalence of time-domain inverse problems and boundary spectral problem}, Inverse problems, \textbf{20} (2004), 419-436.
\bibitem{KK} {\sc  B. Kaltenbacher and  M. Klibanov}, {\em  An inverse problem for a nonlinear parabolic equation with applications in population dynamics and magnetics}, SIAM J. Math. Anal., \textbf{39} (2008), 1863-1889.
\bibitem{KSU} {\sc C.E. Kenig, J. Sj\"ostrand, G. Uhlmann}, {\em The Calderon problem with partial data}, Ann. of Math., {\bf 165} (2007), 567-591.
 \bibitem{Ki2}{\sc Y. Kian}, {\em Unique determination of a time-dependent potential for  wave equations from partial data},   Ann. Inst. H. Poincar\'e (C) Anal. Non Lin\'eaire, \textbf{34} (2017), 973-990.
 \bibitem{Ki3}{\sc Y. Kian}, {\em Recovery of time-dependent damping coefficients and potentials appearing in wave equations from partial data},  SIAM J. Math. Anal., \textbf{48} (6), 4021-4046.
\bibitem{Ki5}{ Y. Kian}, {\em Determination of  non-compactly supported and non-smooth electromagnetic potentials appearing in an  elliptic equation on a general unbounded closed waveguide}, preprint, 	arXiv:1802.04185.
\bibitem{Ki6}{ Y. Kian}, {\em On the determination of nonlinear terms appearing in semilinear hyperbolic equations}, preprint, 	arXiv:1807.02165.
\bibitem{KO}{\sc Y. Kian and L. Oksanen}, {\em Recovery of time-dependent coefficient on Riemanian manifold for hyperbolic equations}, to appear in IMRN, 	https://doi.org/10.1093/imrn/rnx263.
 \bibitem{KS}{\sc Y. Kian,   E. Soccorsi}, 	{\em H\"older stably determining the time-dependent electromagnetic potential of the Schr\"odinger equation}, to appear in SIAM J. Math. Anal., arXiv:1705.01322.
\bibitem{Kl}{\sc M. Klibanov}, {\em Global uniqueness of a multidimensional inverse problem for a nonlinear parabolic equation by a Carleman estimate}, Inverse Problems, \textbf{20} (2004), 1003.
\bibitem{KU1} {\sc K. Krupchyk and G. Uhlmann}, {\em Uniqueness in an inverse boundary problem for a magnetic Schrodinger operator with a bounded magnetic potential}, Comm. Math. Phys., \textbf{327}  (2014), 993-1009.
 \bibitem{KU2} {\sc K. Krupchyk and G. Uhlmann}, {\em  Inverse problems for advection diffusion equations in admissible geometries},  Comm. Partial Differential Equations, https://doi.org/10.1080/03605302.2018.1446163.

\bibitem{KLU} {\sc Y. Kurylev, M. Lassas, G. Uhlmann}, {\em Inverse problems for Lorentzian manifolds and non-linear hyperbolic equations}, Inventiones mathematicae, \textbf{212} (2018),  781-857.
\bibitem{KLOU} {\sc Y. Kurylev, M. Lassas, L. Oksanen, G. Uhlmann}, {\em Inverse problem for Einstein-scalar field equations}, preprint, arXiv:1406.4776.
 \bibitem{LUW}{\sc M. Lassas, G. Uhlmann, Y. Wang}, {\em Inverse problems for semilinear wave equations on Lorentzian manifolds},
Communications in Mathematical Physics,  \textbf{360} (2018), 555-609.
\bibitem{LSU}{\sc O. A. Ladyzhenskaja, V. A. Solonnikov, N. N. Ural'tzeva}, {\em Linear and
quasilinear equations of parabolic type}, Nauka, Moscow, 1967 in Russian ;
English translation : American Math. Soc., Providence, RI, 1968.
\bibitem{LM1}{\sc J-L. Lions and E. Magenes}, {\em Non-Homogeneous Boundary Value Problems and Applications}, Vol. I, Dunod, Paris, 1968.
\bibitem{LM2}{\sc J-L. Lions and E. Magenes}, {\em Non-Homogeneous Boundary Value Problems and Applications}, Vol. II, Dunod, Paris, 1968.
\bibitem{P}{\sc V. Pohjola}, {\em A uniqueness result for an inverse problem of the steady state convection-diffusion equation}, SIAM J. Math. Anal., \textbf{47} (2015), 2084-2103.
\bibitem{Pot1}{\sc  L. Potenciano-Machado},  {\em Optimal stability estimates for a Magnetic Schr\"odinger operator with local data}, Inverse Problems, 	\textbf{33} (2017), 095001.
\bibitem{Pot2}{\sc  L. Potenciano-Machado, A. Ruiz},  {\em Stability estimates for a Magnetic Schrodinger operator with partial data}, Inverse Probl. Imaging, \textbf{12} (2018),	1309-1342.

\bibitem{SS}{\sc J. C. Saut and B. Scheurer}, {\em  Unique continuation for some evolution equations}, J. Diff. Equat., \textbf{66} (1987), 118-139.
\bibitem {Sa1} {\sc M. Salo}, {\em Inverse problems for nonsmooth first order perturbations of the
Laplacian}, Ann. Acad. Scient. Fenn. Math. Dissertations, Vol. 139, 2004.
\bibitem {STz} {\sc M. Salo, L. Tzou}, {\em Carleman estimates and inverse problems for Dirac operators}, Math.
Ann., \textbf{344} (2009), no. 1, 161-184.
\bibitem{Suuu}{\sc Z. Sun}, {\em An inverse boundary value problem for the Schr\"odinger operator with vector potentials}, Trans. Amer. Math. Soc.,  \textbf{338} No. 2 (1992), 953-969.
\bibitem{SuU}{\sc  Z. Sun and G.  Uhlmann}, {\em Inverse problems in quasilinear anisotropic
media}, Amer. J. Math., \textbf{119} (1997), 771-799.
\bibitem {St}{\sc T. Stocker}, {\em Introduction to Climate Modelling}, Springer Science \& Business Media, 2011.
\bibitem{SU} {\sc J. Sylvester and G. Uhlmann}, {\em A global uniqueness theorem for an inverse boundary value problem}, Ann. of Math., {\bf 125} (1987), 153-169.
\bibitem{ZL} {\sc G. Zhang and P. Li}, {\em An Inverse Problem of Derivative Security Pricing},  Proceedings of the International Conference on Inverse Problems,
2003, pp. 411-419.
\end{thebibliography}
\end{document}